%% file: main.tex
\documentclass[10pt]{amsart}

\usepackage[T1]{fontenc}
\usepackage{marginnote}
\usepackage{xfrac}
\usepackage{amsmath}
\usepackage{amssymb}
\usepackage{amsthm}
\usepackage{latexsym}
\usepackage{pstricks}
\usepackage{amsbsy}
\usepackage{tikz-cd}
\usepackage{mathtools}
\usepackage{faktor}  
\usepackage{soul}
\usepackage{tikz}
\usetikzlibrary{cd}
\usepackage{float}
\usepackage{xcolor}
\usepackage[dvipsnames]{xcolor}

\tikzset{curve/.style={settings={#1},to path={(\tikztostart)
    .. controls ($(\tikztostart)!\pv{pos}!(\tikztotarget)!\pv{height}!270:(\tikztotarget)$)
    and ($(\tikztostart)!1-\pv{pos}!(\tikztotarget)!\pv{height}!270:(\tikztotarget)$)
    .. (\tikztotarget)\tikztonodes}},
    settings/.code={\tikzset{quiver/.cd,#1}
        \def\pv##1{\pgfkeysvalueof{/tikz/quiver/##1}}},
    quiver/.cd,pos/.initial=0.35,height/.initial=0}
\usetikzlibrary{calc}
\tikzset{tail reversed/.code={\pgfsetarrowsstart{tikzcd to}}}
\tikzset{2tail/.code={\pgfsetarrowsstart{Implies[reversed]}}}
\tikzset{2tail reversed/.code={\pgfsetarrowsstart{Implies}}}

\newcommand{\redbox}[1]{\fcolorbox{red}{white}{$\displaystyle#1$}}
\newcommand{\magentabox}[1]{\fcolorbox{magenta}{white}{$\displaystyle#1$}}
\newcommand{\bluebox}[1]{\fcolorbox{blue}{white}{$\displaystyle#1$}}

\usepackage{array}
\usepackage[all]{xy}
\usepackage{amscd}
\usepackage{mathrsfs}
\usepackage{txfonts}
\usepackage{amsfonts}
\usepackage{graphicx}
\usepackage{listings}
\usepackage{url}
\usepackage{bookmark}
\newtheorem{theorem}{Theorem}[section]
\newtheorem{lemma}[theorem]{Lemma}
\newtheorem{corollary}[theorem]{Corollary}
\newtheorem{proposition}[theorem]{Proposition}
\newtheorem{definition}[theorem]{Definition}
\newtheorem{defprop}[theorem]{Definition-Proposition}
\newtheorem{remark}[theorem]{Remark}
\newtheorem{example}[theorem]{Example}

\usepackage{marginnote}

\begin{document}
\include{macro}

\title{PRESILTING SEQUENCES FOR 0-AUSLANDER EXTRIANGULATED CATEGORIES}
\author[I. Nonis]{Iacopo Nonis}
\address{School of Mathematics \\ 
University of Leeds \\ 
Leeds, LS2 9JT \\ 
United Kingdom \\
}
\email{mmin@leeds.ac.uk}

\begin{abstract}
Let $\C$ be a reduced $0$-Auslander extriangulated category. Motivated by Pan--Zhu silting reduction for such categories, we introduce the notion of (signed) presilting sequences in $\C$ and establish a bijection between (signed) presilting sequences in $\C$ and (signed) $\tau$-exceptional sequences over $\Lambda = \End_{\C}(P)$, where $P$ is a projective generator of $\C$. This correspondence provides a new perspective on the Buan--Marsh bijection between signed $\tau$-exceptional sequences and ordered support $\tau$-rigid objects. Furthermore, we introduce a new category $\Mcluster(\C)$, called the $\tau$-cluster morphism category of $\C$, whose objects are certain extension-closed subcategories of $\C$ and whose morphisms are described in terms of signed presilting sequences. As an application, we recover the $\tau$-cluster morphism category of $\Lambda$ from $\Mcluster(\C)$.
\end{abstract}

\maketitle
\tableofcontents  

\section{Introduction}

The notion of extriangulated categories, introduced by Nakaoka and Palu in \cite{Nakaoka-Palu_extriangulated_categories}, provides a unified generalization of both exact and triangulated categories. Notable instances of extriangulated categories include extension-closed subcategories of triangulated categories. Moreover, the class of extriangulated categories is closed under taking ideal quotients formed by projective-injective objects \cite{Nakaoka-Palu_extriangulated_categories}, as well as certain relative structures \cite{TriangulatedQuotientRevisited, n-exangulated_categories_I, Sakai} and localizations \cite{Localization_of_extriangulated_categories}. 

A fundamental technique in silting theory for triangulated categories \cite{Keller-Vossieck, SiltinMutationInTriangulatedCat} is that of reduction \cite{Iyama-Yang}. Silting reduction studies collections of silting objects or subcategories that share a common direct summand, and produces a bijective correspondence with silting objects in a smaller category obtained by “factoring out” this common part. Recently, this notion has been extended to the broader setting of extriangulated categories; see \cite{HereditaryCotorsionPairsAndSiltingSubcatInExtriCat, TT, Siltin_reduction_in_exact_cat, PanZhu_SiltingReduction, borve2024silting}.  

Motivated by the additive categorification of cluster algebras \cite{ClusterCategories}, Adachi, Iyama, and Reiten introduced $\tau$-tilting theory in \cite{tau-tiling-theory}. Let $\Lambda$ be a finite dimensional algebra, and let $\tau$ denote the Auslander–Reiten translation in $\modd \Lambda$. Recall that a $\Lambda$-module $M$ is called $\tau$-\textit{rigid} if $\Hom_\Lambda(M,\tau M)=0$, and $\tau$-\textit{tilting} if, in addition, $\abs{M}=\abs{\Lambda}$, where $\abs{X}$ denotes the number of isomorphism classes of indecomposable direct summands of a $\Lambda$-module $X$. Denote by $\CLambda(\Lambda)$ the full subcategory of the bounded derived category $D^b(\modd\Lambda)$ given by $\modd\Lambda\oplus\modd\Lambda[1]$. An object $U = M\oplus P[1]$ in $\CLambda(\Lambda)$ is said to be \textit{support $\tau$-rigid} \cite{tauExcSeq_BM} if $M$ is $\tau$-rigid, $P$ is projective, and $\Hom_\Lambda(P,M) = 0$. Additionally, $U$ is called \textit{support $\tau$-tilting} if it is support $\tau$-rigid and $\abs{M}+\abs{P} = \abs{\Lambda}$. This corresponds to $(M,P)$ being a $\tau$-rigid (resp $\tau$-tilting) pair in the terminology of \cite{tau-tiling-theory}.

Reduction techniques for $\tau$-tilting modules were subsequently developed by Jasso \cite{Jasso_Reduction}. For a $\Lambda$-module $M$, let $M^{\perp_0}$ denote the full subcategory of $\modd \Lambda$ consisting of all modules $X$ such that $\Hom_{\Lambda}(M,X)=0$, and define ${}^{\perp_0}M$ dually. Given a basic $\tau$-rigid $\Lambda$-module $M$, Jasso introduced the \textit{$\tau$-perpendicular category} $J(M) = M^{\perp_0} \cap {}^{\perp_0}\tau M$, thereby extending the notion of perpendicular categories of Geigle--Lenzing from the hereditary setting. He showed that $J(M)$ is equivalent to the module category of a finite dimensional algebra $\Gamma_M$. The algebra $\Gamma_M$ is known as a \textit{$\tau$-tilting reduction}.
This perspective was subsequently broadened to the setting of support $\tau$-rigid objects in \cite{Beyond_tau-tilting_theory}. More precisely, for a support $\tau$-rigid object $U = M \oplus P[1]$ in $\CLambda(\Lambda)$, one defines $J(U) = {}^{\perp_0}\tau M \cap (P \oplus M)^{\perp_0}$. In this case, the equivalence between $J(U)$ and the module category of a finite dimensional algebra $\Gamma_U$ follows by a direct generalization of Jasso’s result which is the special case $P=0$ (see \cite[Thm. 4.12]{Beyond_tau-tilting_theory}). We refer to a projective object in a $\tau$-perpendicular subcategory $J(U)$ as \textit{relative projective}.

Exceptional sequences were first introduced in algebraic geometry \cite{Bondal, GR, Gorodentsev}, and were later studied in the representation theory of hereditary algebras by Crawley-Boevey and Ringel \cite{Boevey-ExcSeq, braid_group_action_on_exc_seq_Ringel}. An exceptional sequence $(M_1,\cdots, M_t)$ is a sequence of indecomposable exceptional objects, where an object is called exceptional if its endomorphism algebra is a division ring and it has no self-extensions in any degree. Moreover, the sequence is required to satisfy $\Hom(M_i,M_j) = 0 = \Ext^{\geq 1}(M_i,M_j)$ for $i>j$.

Motivated by $\tau$-tilting theory and Jasso’s reduction for $\tau$-tilting modules, Buan and Marsh \cite{tauExcSeq_BM} introduced $\tau$-exceptional sequences for a finite dimensional algebra $\Lambda$ as a generalization of exceptional sequences in the hereditary setting. A sequence $(X_1,\dots,X_t)$ of indecomposable $\Lambda$-modules is called \textit{$\tau$-exceptional} if $X_t$ is $\tau$-rigid in $\modd \Lambda$, and $(X_1,\dots,X_{t-1})$ is a $\tau$-exceptional sequence in $J(X_t)$. If $t=\abs{\Lambda}$, the sequence is said to be \textit{complete}. Complete $\tau$-exceptional sequences always exist for arbitrary finite dimensional algebras, in contrast to complete classical exceptional sequences, which in general exist only in the hereditary case. Moreover, over hereditary algebras, exceptional sequences and $\tau$-exceptional sequences coincide.

More generally, a \textit{signed $\tau$-exceptional sequence}, is a $\tau$-exceptional sequence in $\CLambda(\Lambda)$ in which relative projective objects can be assigned a sign (see Section \ref{Sec_signed_tau_exc_seq}). Signed $\tau$-exceptional sequences generalise classical signed exceptional sequences introduced by Igusa and Todorov in \cite{signed_exc_seq}. Signed exceptional sequences were originally defined to describe morphisms in the \textit{cluster morphism category}, a category defined by Igusa--Todorov whose objects are finitely generated wide subcategories of the module category of a hereditary algebra. Since then, cluster morphism categories have been actively studied in the representation theory of finite dimensional algebras. The definition of the cluster morphism category was generalised by Buan and Marsh \cite{A_category_of_wide_subcategories} to the case of $\tau$-tilting finite algebras, that is, algebras with only finitely many isoclasses of indecomposable $\tau$-rigid modules \cite{tauTiltingFiniteAlgebrasAnd_g-vectors}, and given the name of \textit{$\tau$-cluster morphism category} in \cite{Hanson-Igusa_tau-cluster_morphism_categories_and_picture_groups}. The definition was extended to arbitrary finite dimensional algebras by Buan and Hanson in \cite{tau-perpendicular_wide_subategories}. The objects of the $\tau$-cluster morphism category are $\tau$-perpendicular subcategories of the module category of a finite dimensional algebra, and morphisms can be described in terms of signed $\tau$-exceptional sequences. The $\tau$-cluster morphism category exhibits notable connections with topology. In particular, Hanson and Igusa \cite{Hanson-Igusa_tau-cluster_morphism_categories_and_picture_groups} showed that its classifying space admits the structure of a cube complex. Furthermore, in many cases, this classifying space is a $K(\pi,1)$ space, where $\pi$ is the picture group of the algebra \cite{igusa2016picture} (see also \cite{Pairwise_compatibility_for_2-simple_minded_collections, BarnardHanson2026}). Finally, the $\tau$-cluster morphism category has also been studied from both geometric \cite{tau_cluster_geometric, Kaipel_the_category_of_a_partitioned_fan} and silting-theoretic perspectives \cite{borve2021two}.

Gorsky, Nakaoka, and Palu \cite{GNP_HereditaryExtriangulated} introduced 0-Auslander extriangulated categories as a generalisation of the homotopy categories of two-term complexes of projectives, and more broadly, two-term categories of the form $\mathcal{R}*\mathcal{R}[1]$ associated with a rigid subcategory $\mathcal{R}$ in a triangulated category. Moreover, 0-Auslander extriangulated categories admit a mutation theory of maximal rigid objects which recovers many well-known mutations in representation theory, see \cite[Section 4]{GNP_HereditaryExtriangulated} and \cite[Section 5]{some_appl_of_extriang_cat}.  
Throughout this paper, let $k$ be a field, and let $(\C,\EE,\mathfrak{s})$ be a $k$-linear, Hom-finite, Krull–Schmidt, and reduced $0$-Auslander extriangulated category (see Definition \ref{0-Auslander-definition}) with $\proj \C = \add P$, and set $\Lambda = \End_\C(P)$. Motivated by the connection between silting theory in $\C$ and the $\tau$-tilting theory of $\Lambda$, as well as the compatibility between Pan–Zhu silting reduction and Jasso’s reduction for $\tau$-tilting modules \cite{PanZhu_SiltingReduction}, this paper aims to establish a link between the theory of $0$-Auslander extriangulated categories and the theory of (signed) $\tau$-exceptional sequences over $\Lambda$.

Let $\Lambda$ be a finite dimensional algebra. A tuple $(\U_1,\dots,\U_t)$ of indecomposable objects in $\CLambda(\Lambda)$ is called an \textit{ordered support $\tau$-rigid object} if $\bigoplus_{i=1}^t \U_i$ is a support $\tau$-rigid object in $\CLambda(\Lambda)$. If, in addition, $t=\abs{\Lambda}$, the tuple is called an \textit{ordered support $\tau$-tilting object}; see \cite{tauExcSeq_BM}.
Buan--Marsh established a bijection between ordered support $\tau$-rigid objects of length $t$ in $\CLambda(\Lambda)$ and signed $\tau$-exceptional sequences of length $t$ in $\CLambda(\Lambda)$. Let $U = M \oplus P[1]$ be a support $\tau$-rigid object in $\CLambda(\Lambda)$. The above bijection is built from a correspondence, denoted by $\Eps_U$, which associates an indecomposable support $\tau$-rigid object in $\CLambda(J(U))$ to each indecomposable support $\tau$-rigid object $X$ in $\CLambda(\Lambda)$, such that $X \oplus U$ is support $\tau$-rigid. This construction can be viewed as a reduction procedure for support $\tau$-rigid objects, extending Jasso’s $\tau$-tilting reduction.

Recall that an object $U \in \C$ is called \textit{presilting} (or \textit{rigid}) if $\EE(U,U)=0$. In \cite{PanZhu_SiltingReduction}, the authors established a bijection between presilting objects in $\C$ and support $\tau$-rigid objects over $\Lambda$, where $\Lambda = \End_{\C}(P)$ and $P$ is a projective generator of $\C$. More precisely, given a presilting object $X \in \C$, write $X = X' \oplus I_X$, where $I_X$ is the maximal injective direct summand of $X$, and let $\Omega I_X$ denote the corresponding projective object under the equivalence $\Omega \colon \inj \C \to \proj \C$ (see Section \ref{Section_0-Auslander}). Then, $H_P(X) = \C(P,X') \oplus \C(P,\Omega I_X)[1]$ defines a bijection between presilting objects in $\C$ and support $\tau$-rigid objects in $\CLambda(\Lambda)$ (see Proposition \ref{presilting_supp.tau-rigid}).

Let $U$ be a presilting object in $\C$. Inspired by the definition of (co-)Bongartz complements of presilting objects in $2$-Calabi--Yau triangulated categories \cite{Jasso_Reduction}, the notions of Bongartz and co-Bongartz complements of $U$ in $\C$ were introduced in \cite{GNP_HereditaryExtriangulated} (see Def.-Prop. \ref{BongartzCoBongartz_of_presilting}).  Let $B_U = B[U]\oplus U$ and $C_U = C[U]\oplus U$ denote the Bongartz and co-Bongartz completions of $U$ in $\C$, respectively. Denote by $U^\perp$ the full subcategory of $\C$ consisting of all objects $X$ such that $\EE(U,X)=0$, and define ${}^\perp U$ dually.  

Pan and Zhu \cite{PanZhu_SiltingReduction} considered the subcategory $\C_U = {}^\perp C_U \cap B_U^\perp$ and its additive quotient $\widetilde{\C}_U = \C_U/[U]$, where $[U]$ denotes the ideal of morphisms in $\C_U$ factoring through $\add U$. They showed that $\C_U$ is $0$-Auslander and that $\widetilde{\C}_U$ is a reduced $0$-Auslander extriangulated category with $\proj \widetilde{\C}_U = \add B[U]$. Moreover, $\C_U$ coincides with the subcategory ${}^\perp U \cap U^\perp$ considered in \cite{GNP_HereditaryExtriangulated} (see \cite[Remark 4.20(2)]{PanZhu_SiltingReduction}), and $\widetilde{\C}_U$ is extriangulated equivalent to the Verdier quotient $\C/{\text{thick}(\add U)}$ studied by B{\o}rve in \cite{borve2024silting}.  

The projection functor $\pi_U \colon \C_U \to \widetilde{\C}_U$ induces a bijection between $\ind\presilt \C_U \setminus \add U$ and $\ind\presilt \widetilde{\C}_U$, where $\ind\presilt\C_U$ and $\ind\presilt\widetilde{\C}_U$ denote the set of indecomposable presilting objects in $\C_U$ and $\widetilde{\C}_U$, respectively (see Lemma \ref{proj_functior_bij}). Furthermore, Pan and Zhu \cite[Subsection~4.4]{PanZhu_SiltingReduction} established a connection between $\widetilde{\C}_U$ and the $\tau$-tilting reduction $\Gamma_{H_P(U)}$ of $H_P(U)$. Namely, the functor $\widetilde{\C}_U(B_U,-)$ induces an equivalence $\widetilde{\C}_U / [C[U]] \simeq \modd \Gamma_{H_P(U)}$, where $\add C[U] = \inj \widetilde{\C}_U$. Composing this with the equivalence between $\modd \Gamma_{H_P(U)}$ and $J(H_P(U))$ yields a bijection $\widetilde{H}_{B_U}$ between presilting objects in  $\widetilde{\C}_U$ and support $\tau$-rigid objects in  $\CLambda(J(H_P(U))$ (see Lemma \ref{H_tilda}).  

Motivated by the compatibility between Pan--Zhu reduction and Jasso’s $\tau$-tilting reduction, as well as results in \cite{borve2021two}, we prove our first main theorem, which can be viewed as a compatibility theorem between Pan--Zhu (pre)silting reduction and the Buan--Marsh reduction of support $\tau$-rigid objects. This generalizes \cite[Theorem~4.22]{PanZhu_SiltingReduction} to the level of support $\tau$-rigid objects.

\begin{theorem}[{Theorem \ref{compatibility_thm}}]\label{introThm1}
    Let $U$ be a presilting object in $\C$. There is a commutative diagram of bijections given by 
    \begin{equation*}
    \begin{tikzcd}[ampersand replacement=\&,cramped]
    	{\ind\presilt \C_U\setminus\add U} \& {\{W\in\ind\CLambda(\Lambda)\setminus\add H_P(U)\mid W\oplus H_P(U) \text{ is supp. }\tau\text{-rigid}\}} \\
    	{\ind\presilt\widetilde{\C}_U} \& {\{Z\in\ind\CLambda(J(H_P(U)))\mid Z \text{ is supp. }\tau_{J(H_P(U))}\text{-rigid}\}.}
    	\arrow["{H_P}", from=1-1, to=1-2]
    	\arrow["{\pi_U}"', from=1-1, to=2-1]
    	\arrow["{\Eps_{H_P(U)}}", from=1-2, to=2-2]
    	\arrow["{\widetilde{H}_{B_U}}", from=2-1, to=2-2]
    \end{tikzcd}
    \end{equation*}
\end{theorem}

Motivated by Theorem  \ref{introThm1} and the recursive definition of signed $\tau$-exceptional sequences, we propose the following recursive definition as an analogue of signed $\tau$-exceptional sequences in $0$-Auslander extriangulated categories.

\begin{definition}[Definition {\ref{def_signed_presilt_seq}}]\label{def_signed_presilt_intro}
    Let $t\in\{1,\cdots,\abs{P}\}$. A $t$-tuple of indecomposable objects $(U_1,\cdots,U_t)$ in $\C$ is called a \emph{signed presilting sequence} if 
    \begin{itemize}
        \item[(a)] The object $U_t$ is non-injective presilting in $\C$, or $U_t = \Sigma P_{U_t}\in \inj\C$ for an indecomposable projective object $P_{U_t}$ of $\C$; 
        \item[(b)] The sequence $(U_1,\cdots,U_{t-1})$ is a signed presilting sequence in $\widetilde{\C}_{U_t}$.
    \end{itemize}
    If $t = \abs{P}$, $(U_1,\cdots,U_t)$ is called a \emph{signed silting sequence}. A signed presilting sequence in which every object $U_{i}$ is non-injective in $\C_{(U_{i+1},\cdots, U_t)}$ where  $$ \C_{(U_{i+1},\cdots, U_t)}  = \C^{\C_{(U_{i+2},\cdots,U_t)}}_{U_{i+1}},$$ is called a \emph{presilting sequence}. 
\end{definition}

Our second main result establishes a bijection between signed presilting sequences in $\C$ and signed $\tau$-exceptional sequences in $\CLambda(\Lambda)$, which restricts to a bijection between presilting sequences in $\C$ and $\tau$-exceptional sequences in $\modd \Lambda$. In particular, this correspondence provides a new interpretation of the Buan--Marsh bijection between support $\tau$-rigid objects and signed $\tau$-exceptional sequences.

\begin{theorem}[{Proposition \ref{signed_presilt_seq-signed_tau-exc_seq_bijection}, Corollary \ref{presilt_seq-tau-exc_seq-bijection}, Theorem \ref{xi-H_P-varphi}}]\label{thm1.4intro}
        Let $t\in\{1,\cdots,\abs{P}\}$. There is a bijection 
    \[
        \xi_t :
        \vcenter{\hbox{$
        \left\{
        \begin{array}{c}
        \text{(signed) presilting sequences} \\
        \text{of length } t \text{ in } \C
        \end{array}
        \right\}
        $}}
        \;\longrightarrow\;
        \vcenter{\hbox{$
        \left\{
        \begin{array}{c}
        \text{(signed) $\tau$-exceptional sequences} \\
        \text{of length } t \text{ in } \mathcal{C}(\Lambda)
        \end{array}
        \right\},
        $}}
        \]
        given by $$ (U_1,\cdots, U_t)\mapsto \left(\widetilde{H}_{B_{\bigoplus_{i=2}^t U_i}}(U_1), \widetilde{H}_{B_{\bigoplus_{i=3}^t U_i}}(U_2),\cdots, \widetilde{H}_{B_{U_t}}(U_{t-1}), H_P(U_t)\right).$$ Moreover, the bijection $\xi_t$ fits into a commutative diagram of bijections
        \[\begin{tikzcd}[ampersand replacement=\&,cramped]
	\begin{array}{c} \vcenter{\hbox{$         \left\{         \begin{array}{c}         \text{ordered presilting objects} \\   \text{of length } t \text{ in } \C  \end{array}         \right\}         $}} \end{array} \& \begin{array}{c}  \vcenter{\hbox{$         \left\{         \begin{array}{c}         \text{ordered support $\tau$-rigid objects} \\         \text{of length } t \text{ in } \mathcal{C}(\Lambda)         \end{array}         \right\}        $}} \end{array} \\
	\begin{array}{c} \vcenter{\hbox{$         \left\{         \begin{array}{c}         \text{signed presilting sequences} \\         \text{of length } t \text{ in } \C         \end{array}         \right\}         $}} \end{array} \& \begin{array}{c}  \vcenter{\hbox{$         \left\{         \begin{array}{c}         \text{signed $\tau$-exceptional sequences} \\         \text{of length } t \text{ in } \mathcal{C}(\Lambda)         \end{array}         \right\}.        $}} \end{array}
	\arrow["{H_P}", from=1-1, to=1-2]
	\arrow[equals, from=1-1, to=2-1]
	\arrow["{\psi_t}", from=1-2, to=2-2]
	\arrow["{\xi_t}", from=2-1, to=2-2]
\end{tikzcd}\]
\end{theorem}

Inspired by the construction of the $\tau$-cluster morphism category $\Mcluster(\Lambda)$ and the bijection $\xi_t$ from Theorem \ref{thm1.4intro}, we define a new category, denoted by $\Mcluster(\C)$, which we call the \textit{$\tau$-cluster morphism category} of $\C$. The objects of $\Mcluster(\C)$ are subcategories of $\C$ of the form $\C_U$, where $U$ is a presilting object in $\C$, and the morphisms can be described in terms of signed presilting sequences in $\C$.

We further construct a functor $F \colon \Mcluster(\C) \to \Mcluster(\Lambda)$ which is dense, faithful, and a discrete fibration. Finally, we show how to recover the $\tau$-cluster morphism category $\Mcluster(\Lambda)$ from $\Mcluster(\C)$ as follows.

\begin{theorem}[{Proposition \ref{functor-tau-cluster-morph-cat}, Theorem \ref{F_dense_faithful_discrete_fibration}, Proposition \ref{discete_fibration_equivalence}}]\label{tau-cluster-intro}
    There exists a functor $F:\Mcluster(\C)\to\Mcluster(\Lambda)$ which is dense, faithful, and a discrete fibration. Let ${\Mcluster(\C)}/{\sim}$ be quotient category defined via an equivalence relation on the objects and morphisms of $\Mcluster(\C)$ that identifies objects and morphisms having the same image in $\Mcluster(\Lambda)$ under $F$ (see Definition \ref{congruence_cat_def}). Then the categories ${\Mcluster(\C)}/{\sim}$ and $\Mcluster(\Lambda)$ are equivalent.
\end{theorem}

The paper is organised as follows. In Section \ref{Section_0-Auslander}, we recall the necessary terminology and results on ($0$-Auslander) extriangulated categories. Section \ref{Sec_signed_tau_exc_seq} is devoted to the theory of (signed) $\tau$-exceptional sequences. In particular, we recall the construction of the Buan--Marsh bijection between the indecomposable direct summands of the Bongartz complement and the co-Bongartz complement of a $\tau$-rigid $\Lambda$-module $M$.

Motivated by the exchange correspondence mentioned above, in Section \ref{Sec:Exchange} we construct a bijection between the indecomposable direct summands of the Bongartz and co-Bongartz complements of a presilting object $X$ in $\C$, and show that these constructions are compatible via the bijection $H_P$. In Section \ref{Sec:Compatibility}, we prove Theorem \ref{introThm1}. In Section \ref{Sec:presilting_seq}, we introduce the notion of a (signed) presilting sequence (see Definition \ref{def_signed_presilt_intro}) and prove Theorem \ref{thm1.4intro}. In Section \ref{sec:tau-cluster_morph_cat}, we introduce the $\tau$-cluster morphism category of $\C$, denoted by $\Mcluster(\C)$, and establish some of its basic properties. As an application, in Section \ref{sec:tau-cluster-of-Lambda} we recover the $\tau$-cluster morphism category of $\Lambda$, denoted by $\Mcluster(\Lambda)$, from $\Mcluster(\C)$ (see Theorem \ref{tau-cluster-intro}). We conclude with Section \ref{sec:Examples}, where we present an example illustrating the theory developed in this paper.

\subsection*{Future work} Buan, Hanson, and Marsh \cite{BHM_mutation} introduced a notion of right and left mutation (BHM mutation) for $\tau$-exceptional sequences over an arbitrary finite dimensional algebra $\Lambda$. In forthcoming work \cite{Nonis_mutation_presilting_sequences} the author develops a mutation theory for presilting sequences and proves that this new mutation is compatible with the BHM mutation of $\tau$-exceptional sequences via the map $\xi_t$.

\section{Background on extriangulated categories }\label{Section_0-Auslander}

This section collects some terminology and useful properties of extriangulated categories. For further details, we refer to the foundational work of Nakaoka and Palu \cite{Nakaoka-Palu_extriangulated_categories} or to Palu’s survey \cite{some_appl_of_extriang_cat}.

\begin{definition}
    Let $\C$ be an additive category, and let $\EE: \C^\mathrm{op}\times \C\to \mathrm{Ab}$ be an additive bifunctor, where $\mathrm{Ab}$ denotes the category of abelian groups. 
    \begin{itemize}
        \item[(a)] Let $A,B\in\C$. An element $\delta\in\EE(B,A)$ is called \emph{$\EE$-extension}. A zero element $0\in\EE(B,A)$ is called \emph{split} $\EE$-extension. Using the additivity of $\C$ and $\EE$, for $\delta\in\EE(B,A)$ and $\delta'\in\EE(B',A')$, we define the $\EE$-extension $\delta\oplus\delta'\in\EE(B\oplus B', A\oplus A')$ to be the image of $(\delta,0,0,\delta')$ under the natural isomorphism. 
        \item[(b)] Let $\delta\in\EE(B,A)$ be an $\EE$-extension, and let $f\in\C(A,A')$ and $g\in\C(B',B)$. By functoriality of $\EE$ there exist induced $\EE$-extensions $f_*\delta = \EE(f,B)( \delta)\in \EE(B,A')$ and $h^*\delta = \EE(h, A)(\delta)\in\EE(B',A)$. In particular, we have $\EE(h,f)(\delta) =h^*f_*\delta = f_*h^*\delta$ in $\EE(B',A')$. 
        \item[(c)] Let $f: A\to A'$ and $h: B\to B'$ be morphisms in $\C$, and let $\delta \in \EE(B,A)$ and $\delta'\in\EE(B',A')$ be  $\EE$-extensions. We call $(f, h): \delta\to\delta'$ a \emph{morphism} of $\EE$-extensions provided $f_*\delta = h^*\delta$.  
    \end{itemize}
\end{definition}

\begin{definition}
    Let $A$ and $C$ be objects in $\C$. We say that two sequences in $A\xrightarrow{a} B\xrightarrow{c} C$ and $A\xrightarrow{a'}B'\xrightarrow{c'} C$ in $\C$ are \emph{equivalent} if there exists an isomorphism $b: B\to B'$ making the following diagram commute 
    $$\begin{tikzcd}[ampersand replacement=\&, cramped]
	A \& B \& C \\
	A \& {B'} \& C. 
	\arrow["a", from=1-1, to=1-2]
	\arrow[equals, from=1-1, to=2-1]
	\arrow["b", from=1-2, to=1-3]
	\arrow["b"', from=1-2, to=2-2]
	\arrow[draw=none, "\cong", from=1-2, to=2-2]
	\arrow[equals, from=1-3, to=2-3]
	\arrow["{a'}", from=2-1, to=2-2]
	\arrow["{b'}", from=2-2, to=2-3]
\end{tikzcd}$$
We denote by $[A\xrightarrow{a}B\xrightarrow{c}C]$ the equivalent class of $A\xrightarrow{a}B\xrightarrow{c}C$.
\end{definition}

\begin{definition}
    Let $A,C$ be objects in $\C$. 
    \begin{itemize}
        \item[(a)] A \emph{realization} $\mathfrak{s}$ is a map that given an $\EE$-extension $\delta\in \EE(C,A)$ associates an equivalence class $\mathfrak{s}(\delta) = [A\xrightarrow{a}B\xrightarrow{b}C]$ such that the following condition holds:

        Let $\delta\in \EE(C,A)$ and $\delta'\in\EE(C',A')$ be $\EE$-extensions with $\mathfrak{s}(\delta) = [A\xrightarrow{a}B\xrightarrow{b}C]$ and $\mathfrak{s}(\delta') = [A'\xrightarrow{a'}B'\xrightarrow{b'}C']$. Then, for any morphism $(f,g):\delta\to\delta'$, there exists a morphism $g:B\to B'$ making the following diagram commute: 
        \[\begin{tikzcd}[ampersand replacement=\&,cramped]
        	A \& B \& C \\
        	{A'} \& {B'} \& {C'}.
        	\arrow["a", from=1-1, to=1-2]
        	\arrow["f"', from=1-1, to=2-1]
        	\arrow["b", from=1-2, to=1-3]
        	\arrow["g"', from=1-2, to=2-2]
        	\arrow["h", from=1-3, to=2-3]
        	\arrow["{a'}", from=2-1, to=2-2]
        	\arrow["{b'}", from=2-2, to=2-3]
        \end{tikzcd}\]
    \end{itemize}
    We say that $A\xrightarrow{a}B\xrightarrow{b}C$ \emph{realizes} $\delta$ if $\mathfrak{s}(\delta) = [A\xrightarrow{a}B\xrightarrow{b}C]$, and the triple $(f,g,h)$ \emph{realizes} $(f,h)$ if the above diagram commutes. 
    \item[(b)] A realization $\mathfrak{s}$ is \emph{additive} if: 
    \begin{itemize}
        \item[(i)] For any $A,B\in\C$, the realization of the split $\EE$-extension $0\in\EE(B,A)$ is given by $\mathfrak{s}(0) = 0$; 
        \item[(ii)] For $\delta\in\EE(B,A)$ and $\delta'\in\EE(B',A')$, the realization of $\delta\oplus\delta'$ is given by $\mathfrak{s}(\delta\oplus \delta ') = \mathfrak{s}(\delta)\oplus\mathfrak{s}(\delta')$. 
    \end{itemize}
\end{definition}

With the terminology established above, we are ready to state the main definition of this section. 

\begin{definition}[{\cite[Def. 2.12]{Nakaoka-Palu_extriangulated_categories}}]
    An \emph{extriangulated category} is a triple $(\C, \EE, \mathfrak{s})$ satisfying the following axioms: 
    \begin{itemize}
        \item[(ET1)] $\EE: \C^\mathrm{op}\times \C\to \mathrm{Ab}$ is an additive bifunctor; 
        \item[(ET2)] $\mathfrak{s}$ is an additive realization of $\EE$; 
        \item[(ET3)] Let $\delta\in \EE(C,A)$ and $\delta'\in\EE(C',A')$ be $\EE$-extensions realised by $A\xrightarrow{a} B\xrightarrow{b} C$ and $A'\xrightarrow{a'} B\xrightarrow{b'} C$, respectively. Then, for any commutative diagram 
        \[\begin{tikzcd}[ampersand replacement=\&,cramped]
        	A \& B \& C \\
        	{A'} \& {B'} \& {C'}
        	\arrow["a", from=1-1, to=1-2]
        	\arrow["f"', from=1-1, to=2-1]
        	\arrow["b", from=1-2, to=1-3]
        	\arrow["g"', from=1-2, to=2-2]
        	\arrow["{a'}"', from=2-1, to=2-2]
        	\arrow["{b'}"', from=2-2, to=2-3]
        \end{tikzcd}\]
        in $\C$, there exists a morpshism $(f,h): \delta\to \delta'$ such that $h\circ b = b'\circ g$; 
        \item[(ET3)$^\mathrm{op}$] Dual of $\mathrm{(ET3)}$; 
        \item[(ET4)] Let $\delta\in \EE(C',A)$ and $\delta'\in\EE(A',B)$ be $\EE$-extensions realised by $A\xrightarrow{f}B\xrightarrow{f'}C'$ and $B\xrightarrow{g}C\xrightarrow{g'}A'$, respectively. Then, there exists an object $B'$ in $\C$, a commutative diagram in $\C$ of the form 
        \[\begin{tikzcd}[ampersand replacement=\&,cramped]
        	A \& B \& C \& {\phantom{\bullet}} \\
        	A \& C \& {B'} \& {\phantom{\bullet}} \\
        	\& {A'} \& {A'} \\
        	\& {\phantom{\bullet}} \& {\phantom{\bullet}}
        	\arrow["f", from=1-1, to=1-2]
        	\arrow[equals, from=1-1, to=2-1]
        	\arrow["{f'}", from=1-2, to=1-3]
        	\arrow["g"', from=1-2, to=2-2]
        	\arrow["\delta", dashed, from=1-3, to=1-4]
        	\arrow["d", from=1-3, to=2-3]
        	\arrow["h", from=2-1, to=2-2]
        	\arrow["{h'}", from=2-2, to=2-3]
        	\arrow["{g'}"', from=2-2, to=3-2]
        	\arrow["{\delta''}", dashed, from=2-3, to=2-4]
        	\arrow["e", from=2-3, to=3-3]
        	\arrow[equals, from=3-2, to=3-3]
        	\arrow["{\delta'}"', dashed, from=3-2, to=4-2]
        	\arrow["{f'_*\delta'}", dashed, from=3-3, to=4-3]
        \end{tikzcd}\]
        and an $\EE$-extension $\delta''\in\EE(B', A)$ realized by $A\xrightarrow{h}C\xrightarrow{h'}B'$ satisfying the following property: 
        
        \begin{itemize}
            \item[(i)] $C'\xrightarrow{d}B'\xrightarrow{e}A'$ realizes $f'_*\delta'$; 
            \item[(ii)] $d^*\delta'' = \delta$; 
            \item[(iii)] $f_*\delta'' = e^*\delta$;  
        \end{itemize}
        
        \item[(ET4)$^\mathrm{op}$] Dual of $\mathrm{(ET4)}$. 
    \end{itemize}
\end{definition}

\begin{definition}
    Let $(\C, \EE, \mathfrak{s})$ be an extriangulated category. 
    \begin{itemize}
        \item[(a)] A sequence $A\xrightarrow{i}B\xrightarrow{p}C$ that realizes some $\EE$-extension in $\EE(C,A)$ is called \emph{conflation}. We refer to the morphism $A\xrightarrow{i}B$ as \emph{inflation}, denoted by $A\infl B$, and to the morphism $B\xrightarrow{p}C$ as \emph{deflation}, denoted by $B\defl C$; 
        \item[(b)] An \emph{$\EE$-triangle} (or \emph{extriangle}) is a diagram $A\overset{i}{\infl} B \overset{p}{\defl} C\overset{\delta}{\dashrightarrow}$ where  $A\overset{i}{\infl} B \overset{p}{\defl} C$ is a conflation realizing $\delta\in\EE(C,A)$ (we often omit the rightmost dashed arrow); 
        \item[(c)] A \emph{morphism of extriangles} is a diagram
        \[\begin{tikzcd}[ampersand replacement=\&,cramped]
        	A \& B \& C \& \phantom{\bullet} \\
        	X \& Y \& Z \& \phantom{\bullet}
        	\arrow["i", tail, from=1-1, to=1-2]
        	\arrow["f"', from=1-1, to=2-1]
        	\arrow["p", two heads, from=1-2, to=1-3]
        	\arrow["g"', from=1-2, to=2-2]
        	\arrow["\delta", dashed, from=1-3, to=1-4]
        	\arrow["h", from=1-3, to=2-3]
        	\arrow["j", tail, from=2-1, to=2-2]
        	\arrow["q", two heads, from=2-2, to=2-3]
        	\arrow["\delta'", dashed, from=2-3, to=2-4]
        \end{tikzcd}\]
        where $(f,h)\colon \delta \to \delta'$ is a morphism of the $\mathbb{E}$-extensions realized by $(f,g,h)$.
    \end{itemize}
\end{definition}

\begin{example}\label{extriang_examples}
    \begin{itemize}
        \item[(i)] Let $\C$ be an exact category \cite{exact_categories}. Then, $\C$ is an extriangulated category by setting $\EE(C,A) = \Ext^1_\C(C,A)$ for all $C,A\in \C$, and letting $\mathfrak{s}$ to be the realization that to each $\delta\in\Ext^1_\C(C,A)$ associates a short exact sequence $0\to A\overset{f}{\infl} B \overset{g}{\defl} C \to 0$. Conflations correspond to admissible short exact sequences, in which admissible monomorphisms are conflations, and similarly, admissible epimorphisms are deflations. 
        \item[(ii)] Let $\C$ be a triangulated category.  Then, $\C$ becomes an extriangulated category by taking $\EE(C,A) = \Hom_\C(C, \Sigma A)$, where $\Sigma$ denotes the shift functor on $\C$, and $\mathfrak{s}$ defined by completing $\delta\in\Hom_\C(C, \Sigma A)$ to a disinguished triangle $A\to B\to C\overset{\delta}{\dashrightarrow} \Sigma A$ and forgetting the morphism $\delta$. Notice that every morphism is both an inflation and a deflation by the rotation axiom in a triangulated category. 
        \item[(iii)] Let $\C$ be an extriangulated category and let $\B$ be an extension-closed subcategory of $\C$. Then, $\B$ inherits a natural extriangulated structure by restricting $\EE$ and $\mathfrak{s}$ from $\C$ \cite[Rmk. 2.18]{Nakaoka-Palu_extriangulated_categories}. In other words, $\EE_\B(-,-) = \EE(-,-)|_\B$ and conflations in $\B$ are conflations in $\C$ with all terms in $\B$. 
    \end{itemize}
\end{example}

\begin{definition}
    Let $\C$ be an extriangulated category. 
    \begin{itemize}
        \item[(a)] \cite[Def. 3.23]{Nakaoka-Palu_extriangulated_categories} An object $P\in\C$ is called \emph{projective} if $\C(P, g)$ is surjective for any deflation $g$, and dually an object $I\in\C$ is called \emph{injective} if $\C(h,I)$ is surjective for any inflation $h$. We denote the category of projective (resp. injective) objects in $\C$ by $\proj\C$ (resp. $\inj\C$).  
        \item[(b)] \cite[Def. 3.25]{Nakaoka-Palu_extriangulated_categories} We say that $\C$ has \emph{enough projectives} (resp. \emph{enough injectives}) if for every object $X\in\C$ there exists a deflation $P\defl X$ (resp. an inflation $I\infl X$) with $P$ projective (resp. $I$ injective).  
    \end{itemize}
\end{definition}

Projective and injective objects in $\C$ can be characterised in terms of the vanishing conditions of the bifunctor $\EE$. 

\begin{proposition}[{\cite[Prop. 3.24]{Nakaoka-Palu_extriangulated_categories}}]
    $P\in\C$ is projective (resp. injective) if and only if $\EE(P,-) = 0$ (resp. $\EE(-,I) = 0$).
\end{proposition}

Given an extriangulated category $\C$, one can construct a new extriangulated category by taking the quotient by the ideal of morphisms factoring through projective-injective objects.

\begin{example}[{\cite[Prop. 3.30]{Nakaoka-Palu_extriangulated_categories}}]
    Let $J = (\proj\C\cap \inj\C)$ be the ideal of morphisms factoring through projctive-injective objects. Then, $\widetilde{\C} = \C/J$ has the structure of an extriangulated category, induced from that of $\C$.
\end{example}

\subsection{0-Auslander extriangulated categories} $0$-Auslander extriangulated categories were defined by Gorsky, Nakaoka, and Palu \cite{GNP_HereditaryExtriangulated} as a generalization of the homotopy category of two-term complexes of projectives. In this subsection, we review their definition and the main properties that will be used throughout this paper. For examples of $0$-Auslander extriangulated categories, we refer to \cite{GNP_HereditaryExtriangulated} and \cite[Section 5]{some_appl_of_extriang_cat}.

\begin{definition}\label{0-Auslander-definition}
    An extriangulated category $\C$ is called \emph{$0$-Auslander} if: 
    \begin{itemize}
        \item[(a)] $\C$ has enough projective objects; 
        \item[(b)] For every object $X\in\C$, there exists a conflation $P_1\infl P_0\defl X$, with $P_1, P_0\in\proj\C$; 
        \item[(c)] For every object $P\in\proj\C$, there exists a conflation $P\infl Q\defl I$, with $Q$ projective-injective and $I$ injective. 
    \end{itemize}
    We say that $\C$ is \emph{reduced} if, moreover, its only projective-injective object is $0$. 
\end{definition}

The following results summarize some of the main properties of reduced $0$-Auslander extriangulated categories. These properties will be used throughout the paper without further comment.

\begin{proposition}[{\cite[Section 4.5]{PPPP}}]\label{0-Auslander-properties}
    Let $\C$ be a reduced $0$-Auslander extriangulated category. The following statements hold. 
    \begin{itemize}
        \item[(i)] $\C(P,I) = 0$ for every $P\in\proj\C$ and $I\in\inj\C$; 
        \item[(ii)] For all $X\in\C$, there is an $\EE$-triangle $P\infl X \defl I$ with $P\in\proj\C$ and $I\in\inj\C$; 
        \item[(iii)] Any morphism $P\to X$ (resp. $X\to I$) with $P\in\proj\C$ (resp. $I\in\inj\C$) is a deflation (resp. inflation); 
        \item[(iv)] For each $P\in\proj\C$ (resp. $I\in\inj\C$) fix an $\EE$-triangle $P\infl 0\defl \Sigma P$ (resp. $\Omega I\infl 0\defl I$). Then, there are mutually quasi-inverse equivalences 
        $$ \Sigma: \proj\C \rightleftarrows \inj\C: \Omega$$
        \item[(v)] Suppose that $\C$ admits a projective generator $P$, that is, $\proj\C = \add P$ for some projective object $P\in\C$. Let $\Lambda = \End_\C(P)$. Then, the functor
        $$ \C(P,-): \C\to\modd\Lambda$$ induces an equivalence of categories $\C/[\Sigma P]\to \modd\Lambda$.
    \end{itemize}
\end{proposition}

Higher extensions were defined in \cite{Positive_and_negative_extensions_in_extr_cat}. For example, when $\C$ is a triangulated category viewed as an extriangulated category with $\EE(-,-) = \C(-,\Sigma-)$ (see Example \ref{extriang_examples}), then $\EE^i(-,-) = \C(-,\Sigma^i-)$ \cite[Cor. 3.23]{Positive_and_negative_extensions_in_extr_cat}. We say that an extriangulated category is \textit{hereditary} \cite[Def. 2.3]{GNP_HereditaryExtriangulated} if $\EE^2(-,-) = 0$.   If $\C$ has enough projectives, this condition is equivalent to condition (i) in Definition \ref{0-Auslander-definition} by \cite[Prop. 2.1(2)]{GNP_HereditaryExtriangulated}. Hence, $0$-Auslander extriangulated categories are hereditary. Combining these facts with \cite[Theorem 3.5]{Positive_and_negative_extensions_in_extr_cat}, we obtain the following 

\begin{theorem}[{\cite[Theorem 3.5]{Positive_and_negative_extensions_in_extr_cat}}]
    Let $\C$ be a $0$-Auslander extriangulated category and let $A\infl B\defl C$ be any $\EE$-triangle. Then, for every $X\in\C$ there are exact sequences of abelian groups
    $$\C(X,A)\to \C(X,B)\to \C(X,C)\to \EE(X,A)\to \EE(X,B)\to \EE(X,C)\to 0,$$
    and 
    $$\C(C,X)\to \C(B,X)\to \C(A,X)\to \EE(C,X)\to \EE(B, X)\to \EE(A,X)\to 0.$$
\end{theorem}

\section{Signed $\tau$-exceptional sequences}\label{Sec_signed_tau_exc_seq}

In this section, we review the theory of signed $\tau$-exceptional sequence established in \cite{tauExcSeq_BM}. 

Let $\Lambda$ be a finite dimensional algebra. Denote by $\modd \Lambda$ the category of finitely generated (left) $\Lambda$-modules, and let $\proj \Lambda$ be the full subcategory of projective $\Lambda$-modules. Throughout this paper, all subcategories are assumed to be full and closed under isomorphisms. For a subcategory $\mathcal{X} \subseteq \modd \Lambda$, we write $\ind(\mathcal{X})$ for the set of isomorphism classes of indecomposable objects in $\mathcal{X}$. We set $\mathcal{X}^{\perp_0} = \{ Y \in \modd \Lambda \mid \Hom_\Lambda(\mathcal{X}, Y) = 0 \}$, and define ${}^{\perp_0}\mathcal{X}$ dually.  

We denote by $\Gen \mathcal{X}$ the smallest subcategory containing $\mathcal{X}$ that is closed under factor objects. For $M \in \modd \Lambda$, we write $\add M$ for the full subcategory consisting of all direct summands of finite direct sums of copies of $M$, and set $\Gen M := \Gen(\add M)$.  

We assume that $\Lambda$-modules are basic when possible, and we denote by $\abs{X}$ the number of pairwise non-isomorphic indecomposable direct summands of a $\Lambda$-module $X$. We set $\abs{\Lambda}:= n$. For an arbitrary module category $\W$, we denote by $\tau_{\W}$ (or simply $\tau$ when no confusion arises) the Auslander--Reiten translation in $\W$.

Following \cite{tau-tiling-theory} a pair of $\Lambda$-modules $(M,P)$ is called $\tau$\textit{-rigid} if $M$ is $\tau$-rigid, i.e. $\Hom_\Lambda(M,\tau M) = 0$, $P\in\proj \Lambda$, and $\Hom_\Lambda(P,M) = 0$. A $\tau$-rigid pair $(M,P)$ is called $\tau$-\textit{tilting} if $\abs{M}+\abs{P} = \abs\Lambda$. Buan and Marsh considered the corresponding object $M\oplus P[1]$ in the full subcategory $\CLambda(\Lambda):= \modd \Lambda\oplus \modd\Lambda[1]$ of the bounded derived category $D^b(\modd \Lambda)$. An object $M\oplus P[1]$ in $\CLambda(\Lambda)$ is called \textit{support $\tau$-rigid} (resp. \textit{support $\tau$-tilting}) if $(M,P)$ is a $\tau$-rigid (resp. $\tau$-tilting) pair (see \cite[Def. 1.1]{tauExcSeq_BM}). We denote by $\staurigid\Lambda$ (resp. $\stautilt\Lambda$) the set of support $\tau$-rigid (resp. $\tau$-tilting) objects in $\CLambda(\Lambda)$.

\begin{definition}[{\cite[Def. 1.2]{tauExcSeq_BM}}]
    Let $\Lambda$ be a finite dimensional algebra. For a positive integer $t$, an ordered $t$-tuple of indecomposable objects $(\T_1,\cdots, \T_t)$ in $\CLambda(\Lambda)$ is called an \emph{ordered support $\tau$-rigid object} if $\bigoplus_{i=1}^t\T_i$ is a basic support $\tau$-rigid object. If, in addition, $t=n$, then $(\T_1,\cdots, \T_t)$ is called an ordered support $\tau$-tilting object. 
\end{definition}

For a subcategory $\mathcal{X}$ of $\modd\Lambda$, we denote by $\P(\mathcal{X})$ the subcategory of $\mathcal{X}$ consisting of $\Ext$-projective modules in $\mathcal{X}$, i.e. the modules $X$ in $\mathcal{X}$ such that $\Ext^1_\Lambda(X, Y) = 0$ for all $Y$ in $\mathcal{X}$. 

\begin{theorem}[{\cite[Section 2.2]{tau-tiling-theory}}]\label{BongartzCo-Bongartz}
    Let $U = M\oplus P[1]$ be a basic support $\tau$-rigid object in $\CLambda(\Lambda)$. Then, up to isomorphism 
    \begin{itemize}
        \item[(i)] There exists a unique $\Lambda$-module $B[U]$ such that $B[U]\oplus U$ is support $\tau$-tilting and $\add(B[U]\oplus M) = \P({^\perp}\tau M\cap P^{\perp})$; 
        \item[(ii)] There exists a unique object $C[U] = N \oplus Q[1]$ in $\CLambda(\Lambda)$ such that $C[U]\oplus U$ is support $\tau$-tilting and $\Gen(N\oplus M) = \P(\Gen M)$. 
    \end{itemize}
\end{theorem}

We refer to $B[U]$ and $C[U]$ in Theorem \ref{BongartzCo-Bongartz} as the \textit{Bongartz complement} and \textit{co-Bongartz complement} of $U$, respectively (see \cite{tauExcSeq_BM, Beyond_tau-tilting_theory}). We denote by $B_U = B[U]\oplus U$ (resp. $C_U = C[U]\oplus U$) the \textit{Bongartz} (resp. \textit{co-Bongartz}) \textit{completion} of $U$. 

Let $U = M\oplus P[1]$ be a support $\tau$-rigid object in $\CLambda(\Lambda)$. The \textit{$\tau$-perpendicular category} of $U$ is defined to be $$J(U) = M^{\perp_0} \cap{^{\perp_0}\tau M}\cap P^{\perp_0}.$$ This category was originally defined by Jasso \cite{Jasso_Reduction} for the case $P = 0$ and later generalised in \cite{Beyond_tau-tilting_theory}. The following result summarizes the main facts about $J(U)$. 

\begin{theorem}[{\cite[Thm. 3.8]{Jasso_Reduction}\cite[Thm. 4.12, 4.16]{Beyond_tau-tilting_theory}}]\label{facts_about_J}\label{tau-perp-cat=modGamma}
    Let $U = M \oplus P[1]\in \CLambda(\Lambda)$ be a basic support $\tau$-rigid object. Then, $J(U)$ is a wide subcategory of $\modd\Lambda$ and it is equivalent to the module category of a finite dimensional algebra $\Gamma_U = \End_\Lambda(B[U]\oplus M)^{\text{op}}/I$ where $I$ is the ideal generated by all the morphisms factoring through $M$. Moreover, $\abs{\Gamma_U} = \abs{\Lambda}-\abs{U}$. 
\end{theorem}

Since $J(U)$ is equivalent to a module category, we can consider $\CLambda(J(U)):= J(U)\oplus J(U)[1]\subseteq \CLambda(\Lambda)$ and we say that $V= N\oplus Q[1] \in \CLambda(J(U))$ is support $\tau$-rigid in $\CLambda(J(U))$ if $N$ is $\tau_{J(U)}$-rigid, that is $\Hom_{J(U)}(N,\tau_{J(U)}N) = 0$ where $\tau_{J(U)}$ denotes the Auslander-Reiten translation in $J(U)$ (notice that $\tau_{J(U)}\ncong \tau$ in general), $Q\in \proj J(U)$, and $\Hom_{J(U)}(Q,N) = 0$. Thus, iteratively, we can define the $\tau$-perpendicular subcategory $$J_{J(U)}(V):= {(N\oplus Q)^{\perp_0}}\cap {^{\perp_0}\tau_{J(U)}N}\cap J(U).$$
We are prepared to state the definition of a signed $\tau$-exceptional sequence. 

\begin{definition}[{\cite[Def. 1.3]{tauExcSeq_BM}}]\label{Def-signed-tau-exc-seq}
    Let $\Lambda$ be a finite dimensional algebra and let $t$ be a positive integer. An ordered $t$-tuple of indecomposable objects $(\U_1,\cdots, \U_t)$ in $\CLambda(\Lambda)$ is a \emph{signed $\tau$-exceptional sequence} if 
    \begin{enumerate}
        \item[(a)] $\U_t$ is support $\tau$-rigid in $\CLambda(\Lambda)$, and
        \item[(b)] $(\U_1,\cdots, \U_{t-1})$ is a signed $\tau$-exceptional sequence in $\CLambda(J(\U_t))$.
    \end{enumerate}
     A \emph{$\tau$-exceptional sequence} $(\U_1,\cdots, \U_t)$ is a signed $\tau$-exceptional sequence in which every indecomposable object lies in $\modd\Lambda$. If $t=n$, the sequence is said to be \emph{complete}.  
\end{definition} 

The remaining part of this section aims to review a bijection between ordered support $\tau$-tilting objects and signed $\tau$-exceptional sequences. This is based on a bijective correspondence between indecomposable direct summands of the Bongartz complement and indecomposable direct summands of the co-Bongartz complement of a $\tau$-rigid $\Lambda$-module $M$. We refer to \cite[Section 3]{tauExcSeq_BM} for further details. 

Let $\K = K^b(\proj \Lambda)$ be the bounded homotopy category of projective $\Lambda$-modules. For a $\Lambda$-module $X$, we denote by $\PP_X$ its minimal projective presentation viewed as 2-term object in $\K$, that is if $P_X^{-1}\xrightarrow{f} P_X^0 \to X\to 0$ is the minimal projective presentation of $X$ in $\modd\Lambda$, then $\PP_X = 0\to P_X^{-1}\xrightarrow{f} P_X^0\to 0$  in $\K$. 

Fix $M$ to be a $\tau$-rigid $\Lambda$-module. Let $B = B[M]$ and $C\oplus Q[1]$ denote the Bongartz and co-Bongartz complement of $M$, respectively. Let $\CC_Q = \PP_C\oplus Q[1]$.

\begin{proposition}[{\cite[Proposition 3.7]{tauExcSeq_BM}}]\label{exchange_in_Lambda}
    With the notation above, there is a triangle 
    \begin{equation}\label{trianglePPCC}
        \PP_B \xrightarrow{\beta}\PP_M'\xrightarrow{\alpha} \CC_Q\to, 
    \end{equation}
    where $\beta$ (resp. $\alpha$) is a minimal left (resp. right) $\add \PP_M$ approximation. The triangle \eqref{trianglePPCC} is the direct sum of $\abs{\Lambda} - \abs{M}$ triangles 
    \begin{equation}\label{indec_triangles}
        \PP_{B_i} \xrightarrow{\beta_i} (\PP_M')_i\xrightarrow{\alpha_i} {\XX}_i\to,
    \end{equation}
    where $B = \bigoplus_iB_i$ is a decomposition of $B$ into indecomposable direct summands and $\XX_i$ are the indecomposable direct summands of  $\CC_Q$. 

    Let $B = B'\oplus B''$ be a decomposition of $B$ where $B'$ is the direct sum of indecomposable direct summands of $B$ such that the minimal left $\add U$-approximation $B_i\to U_i$ is not an epimorphism, and let $B''$ be the complement of $B'$ in $B$. We have the following two cases. 
    \begin{itemize}
        \item[(i)] For each direct summand $B_i$ of $B$ which is a summand of $B'$ there is an exact sequence in $\modd\Lambda$
        $$B_i\xrightarrow{\mu_i} U_i'\xrightarrow{\gamma_i}C_i\to 0,$$
        where $\mu_i$ (resp. $\gamma_i$) is a minimal left (resp. right) $\add U$-approximation and $C_i$ is and indecomposable direct summand of $C[M]$. This sequence is part of the long exact sequence associated to \eqref{indec_triangles}: 
        $$H^0(\PP_{B_i})\to H^0((\PP'_U)_i)\to H^0(\XX_i)\to H^1(\PP_{B_i})=0,$$
        where $\XX_i = \PP_{C_i}$. Then, the assignment $\rho(B_i):= C_i$ defines a bijection between indecomposable direct summands of $B'$ and indecomposable direct summands of $C$. 
        \item[(ii)] For each direct summand $B_i$ of $B$ which is a summand of $B''$ there is an exact sequence in $\modd\Lambda$
        $$Q_i\xrightarrow{\delta_i}B_i\xrightarrow{\mu_i} U_i'\to 0,$$
        where $Q_i$ is an indecomposable direct summand of $Q$, $U'\in\add U$, and $\delta_i: Q_i\to B_i$ is a minimal left $\add B$-approximation. This sequence is part of the long exact sequence associated to \eqref{indec_triangles}: 
        $$H^{-1}(\XX_i)\to H^0(\PP_{B_i})\to H^0((\PP'_U)_i)\to H^0(\XX_i)=0,$$
        where $\XX_i = Q_i[1]$. The assignment $\rho(B_i):= Q_i$ defines a bijection between indecomposable direct summands of $B''$ and indecomposable direct summands of $Q$. 
        \item[(iii)] The map $\rho$ in (i)-(ii) gives a bijection between the indecomposable direct summands of $B$ and the indecomposable direct summands of $C\oplus Q[1]$ in $\CLambda(\Lambda)$. 
    \end{itemize}
\end{proposition}

    Let $X$ be an indecomposable $\Lambda$-module. Note that, if $X$ lies in $\P(\Gen M)\setminus \add M$, then $X$ is an indecomposable direct summand of the co-Bongartz complement $C[M]$, while if $X$ lies in $(\proj\Lambda)\cap {^\perp M}$, then $X$ is an indecomposable direct summand of $Q$, where $Q[1]$ is as in Proposition \ref{exchange_in_Lambda}. Let $f_M$ denote the torsion-free functor associated to the torsion pair $(\Gen M, M^\perp)$. The next result can be seen as an extension of Jasso's reduction \cite[Cor. 3.18]{Jasso_Reduction}. 

    \begin{proposition}[{\cite[Prop. 5.6]{tauExcSeq_BM}}]\label{epsilonM}
        There is a bijection
        \[
        \begin{tikzcd}
            \{X\in \ind\Lambda\setminus\ind M \mid X\oplus M \text{ is } \tau\text{-rigid}\}\cup \ind(\proj\Lambda\cap {^\perp M})[1] \arrow[d, "\Eps_M"] \\
            \{Y\in \ind J(M) \mid Y \text{ is } \tau_{J(M)}\text{-rigid}\}\cup \ind\proj(J(M))[1]
        \end{tikzcd}
        \]
        given by 
        \[
            \Eps_M(X) = 
            \begin{cases}
            f_M(X) & \text{if } X \notin\Gen M; \\
            f_M(B_M^X)[1]  & \text{if } X  \in\Gen M; \\
            f_M(B_M^X)[1] & \text{if } X\in \ind(\proj\Lambda\cap {^\perp M})[1],
            \end{cases}
        \]
     where $B_M^X$ is the indecomposable direct summand of the Bongartz complement of $M$ associated to $X$ (see Proposition \ref{exchange_in_Lambda}). 
    \end{proposition}

    \begin{proposition}[{\cite[Prop. 5.11]{tauExcSeq_BM}}]\label{epsilonP[1]}
        Let $P\in\proj\Lambda$. There is a bijection 
          \[
        \begin{tikzcd}
            \{X\in \ind\Lambda\mid X\oplus P[1] \text{ is supp. } \tau\text{-rigid}\}\cup \ind(\proj\Lambda\setminus \add P)[1] \arrow[d, "\Eps_{P[1]}"] \\
            \{Y\in \ind J(P[1]) \mid Y \text{ is } \tau_{J(M)}\text{-rigid}\}\cup \ind\proj(J(M))[1]
        \end{tikzcd}
        \]
        given by 
          \[
            \Eps_{P[1]}(X) = 
            \begin{cases}
            X & \text{if } X\in\ind\modd\Lambda \text{ and } X\oplus P[1] \text{ is supp. }\tau\text{-rigid}; \\
            f_P(X)[1]  & \text{if } X  \in \ind(\proj\Lambda\setminus \add P)[1] \\
            \end{cases}
        \]
    \end{proposition}    

Combining Proposition \ref{epsilonM} and Proposition \ref{epsilonP[1]}, Buan and Marsh proved the following theorem. 

\begin{theorem}[{\cite[Thm. 3.6]{A_category_of_wide_subcategories}}]\label{epsilonU}
    Let $U=M\oplus P[1]\in\CLambda(\Lambda)$ be support $\tau$-rigid. There is a bijection 
     \[
        \begin{tikzcd}
            \{V\in \ind\CLambda(\Lambda)\mid U\oplus V \text{ is support } \tau\text{-rigid}\} \arrow[d, "\Eps_{U}"] \\
            \{W\in \ind\CLambda(J(U)) \mid W \text{ is support } \tau_{J(M)}\text{-rigid}\}.
        \end{tikzcd}
    \]
\end{theorem}

Now we explain the construction of the map $\Eps_U$ in Theorem \ref{epsilonU} (see \cite[Section 3]{borve2021two}).

Let $\widetilde{P}[1] := \Eps_M(P[1])$ in $\proj J(M)[1]$. Then, the bijection $\Eps_M$ in Proposition \ref{epsilonM} restricts to a bijection 

\begin{equation}
    \Eps_M : \{X\in \CLambda(\Lambda)\setminus\add M\mid X\oplus M \oplus P[1] \text{ is supp. $\tau$-rigid}\}\to \{Y\in \CLambda(J(M))\mid Y\oplus \widetilde{P}[1] \text{ is supp. $\tau$-rigid}\}.
\end{equation}

There is an equivalence \cite[Thm. 3.8]{Jasso_Reduction}
$$ F_M:= \Hom_\Lambda(B_M,-): J(M)\to \modd\Gamma_M$$
which induces a bijection 

\begin{equation*}
    F_M: \staurigid J(M) \to \staurigid\Gamma_M.
\end{equation*}

The above bijection restricts to a bijection 

\begin{equation}\label{F_Meq}
    F_M : \{X\in \ind\CLambda(J(M))\mid X\oplus \widetilde{P}[1] \text{ is supp. $\tau$-rigid}\}\to \{Y\in \ind\CLambda(\Gamma_M)\mid Y\oplus{P'}[1] \text{ is supp. $\tau$-rigid}\},
\end{equation}

where $P'[1] = F_M(\widetilde{P}[1]):= F_M\left(\widetilde{P}\right)[1]$ in $(\proj\Gamma_M)[1]$. Applying Proposition \ref{epsilonP[1]} to the algebra $\Gamma_M$, we obtain a bijection

\begin{equation}
    \Eps_{P'[1]}^{\Gamma_M}: \left\{ X\in \ind\CLambda(\Lambda)\setminus\add P'[1] \;\middle|\; \begin{tabular}{@{}l@{}} $X\oplus P'[1]$ \\ \text{is supp. $\tau$-rigid} \end{tabular} \right\} \to \left\{ Y\in \ind\CLambda(J_{\Gamma_M}(P'[1])) \;\middle|\; \begin{tabular}{@{}l@{}} $Y$ \text{ is supp.} \\ $\tau_{J_{\Gamma_M}(P'[1])}$\text{-rigid} \end{tabular} \right\}.
\end{equation}

The equivalence $F_{P'[1]}^{\Gamma_M}$ (see \cite[Thm. 4.12(b)]{Beyond_tau-tilting_theory}) induces a bijection 

\begin{equation*}
    F_{P'[1]}^{\Gamma_M}: \staurigid J_{\Gamma_M}(P'[1]) \to \staurigid\Gamma_{P'[1]}^{\Gamma_M}
\end{equation*}

\begin{lemma}\label{Gamma_P'[1]^Gamma_M=Gamma_U}
    Under the assumptions above, we have $\Gamma_{P'[1]}^{\Gamma_M}\cong \Gamma_U$. 
\end{lemma}

\begin{proof}
    By \cite[Lemma 2.17]{BHM_mutation}, $\Gamma_{P'[1]}^{\Gamma_M} = \End_{\Gamma_M}(B)$, where $B$ is the Bongartz completion of $P'[1]$ in $\CLambda(\Gamma_M)$. This means that $B$ is the maximal element in the codomain of $F_M$ in \eqref{F_Meq}. Since $F_M$ maps the maximal element of its domain to the maximal element of its codomain, we have that $B = F_M(B')$, where $B'$ is the maximal module in $J(M)$ which is orthogonal to $\widetilde{P}$. In other words, $B' = \P(J(M)\cap \widetilde{P}^{\perp_0})$ (see Theorem \ref{BongartzCo-Bongartz}(i) and \cite[Lemma 2.17]{BHM_mutation}). Using \cite[Thm. 6.4]{tau-perpendicular_wide_subategories} we have 
    \begin{align*}
        J(U) = J(M\oplus P[1]) &= J_{J(M)}(\Eps_M(P[1]))\\
        &= J_{J(M)}\left(\widetilde{P}[1]\right)\\
        &= J(M)\cap \widetilde{P}^{\perp_0}.
    \end{align*}
Hence, $B' = \P(J(U))$. Now, $\Gamma_U = \End_\Lambda(\P(J(U))) = \End_\Lambda(B')$ and therefore
\begin{align*}
    \Gamma_{P'[1]}^{\Gamma_M} = \End_{\Gamma_M}(B) = \End_{\Gamma_M}(F_M(B')) \cong \End_\Lambda(B') = \End_\Lambda(\P(J(U))) = \Gamma_U. 
\end{align*}
\end{proof}

Finally, the inverse of the equivalence $F_U: \modd\Gamma_U\to J(U)$ induces a bijection 
\begin{equation*}
    F_U^{-1}: \staurigid\Gamma_U\to \staurigid J(U). 
\end{equation*}

Now we define $\Eps_U$ via the following commutative diagram of bijections 
\begin{equation}\label{Eps_U_via_comm_diagram}
\begin{tikzcd}[ampersand replacement=\&,cramped]
	{\{X\in \ind\CLambda(\Lambda)\mid X\oplus U \text{ is supp. $\tau$-rigid}\}} \&\& {\{Y_1\in \ind\CLambda(J(M))\mid Y_1\oplus \widetilde{P}[1] \text{ is supp. $\tau$-rigid}\}} \\
	\&\& {\{Y_2\in \ind\CLambda(\Gamma_M)\mid Y_2\oplus{P'}[1] \text{ is supp. $\tau$-rigid}\}} \\
	\&\& {\{Y_3\in \ind\CLambda(J_{\Gamma_M}(P'[1]))\mid Y_3 \text{ is supp. $\tau_{J_{\Gamma_M}(P'[1])}$-rigid}\}.} \\
	{\{Y\in\ind\CLambda(J(U))\mid Y \text{ is supp. $\tau_{J(U)}$-rigid}\}} \&\& {\{Y_4\in\ind\CLambda(\Gamma_U)\mid Y_4 \text{ is supp. $\tau$-rigid}\}}
	\arrow["{\Eps_M}", from=1-1, to=1-3]
	\arrow["{\Eps_U}"', from=1-1, to=4-1]
	\arrow["{F_M}", from=1-3, to=2-3]
	\arrow["{\Eps_{P'[1]}^{\Gamma_M}}", from=2-3, to=3-3]
	\arrow["{F_{P'[1]}^{\Gamma_M}}", from=3-3, to=4-3]
	\arrow["{F_U^{-1}}", from=4-3, to=4-1]
\end{tikzcd}
\end{equation}
where we set 
$$\Eps^{J(M)}_{\widetilde{P}[1]} = \Eps^{J(M)}_{\Eps_M(P[1])}:= F^{-1}_U\circ F^{\Gamma_M}_{P'[1]}\circ \Eps^{\Gamma_M}_{P'[1]}\circ F_M,$$
see \cite[Theorem 6.12]{tau-perpendicular_wide_subategories}. We conclude this section by explaining the Buan--Marsh bijection between ordered support $\tau$-rigid objects of length $t$ in $\CLambda(\Lambda)$ and signed $\tau$-exceptional sequences of length $t$ in $\CLambda(\Lambda)$.

Let $\W$ be a $\tau$-perpendicular subcategory of $\modd\Lambda$ and let $(\T_1,\cdots,\T_t)$ be an ordered support $\tau$-rigid object in $\CLambda(\W)$. Let 

    \begin{alignat*}{2}
    \W_t      &= \W              &\qquad  \U_t      &= \T_t \\
    \W_{t-1}  &= J_{\W_t}(\U_t)   &\qquad  \U_{t-1}  &= \Eps_{\U_t}^{\W_t}(\T_{t-1})\\
    &\vdots &\qquad &\vdots \\
    \W_i &= J_{\W_{i+1}}(\U_{i+1}) &\qquad \U_i &= \Eps_{\U_{i+1}}^{\W_{i+1}}\cdots\Eps_{\U_{t-1}}^{\W_{t-1}}\Eps_{\U_t}^{\W_t}(\T_i)\\
     &\vdots &\qquad &\vdots \\
    \W_1 &=J_{\W_2}(\U_2) &\qquad \U_1 &= \Eps_{\U_{2}}^{\W_{2}}\cdots\Eps_{\U_{t-1}}^{\W_{t-1}}\Eps_{\U_t}^{\W_t}(\T_1). 
\end{alignat*}

\begin{theorem}[{\cite[Thm. 5.4, Rmk. 5.13 ]{tauExcSeq_BM}}]\label{BM_bijection}
    Let $t\in\{1,\cdots,n\}$. With the notation above, there is a bijection $\psi_t$ from the set of support $\tau$-rigid objects of length $t$ in $\CLambda(\Lambda)$ to the set of signed $\tau$-exceptional sequences of length $t$ in $\CLambda(\Lambda)$ given by 
    $$\psi_t(\T_1,\cdots,\T_t) = (\U_1\,\cdots,\U_t).$$
\end{theorem}

\section{Exchange in 0-Auslander extriangulated categories}\label{Sec:Exchange}

Throughout this section, let $k$ be a field and let $\C$ be a $k$-linear, Hom-finite, Krull–Schmidt, and reduced $0$-Auslander extriangulated category with projective generator $P$, that is $\proj\C = \add P$, for some $P\in\proj\C$. Then, $I = \Sigma P$ is an injective generator of $\C$ and $\inj\C = \add I$. Let $\Lambda = \End_\C(P)$. Recall, there is an equivalence of categories $\C(P,-): \C/[I]\to \modd\Lambda$. For an object $X$ in $\C$, we denote by $X^\perp$ the subcategory of $\C$ consisting of all the objects $Y$ in $\C$ such that $\EE(X,Y) = 0$, and define $^{\perp}X$ dually. 

An object $U\in\C$ is \emph{presilting} (or \emph{rigid}) if $\EE(U,U) = 0$. If in addition $\abs{U} = \abs{P}$, then $U$ is said to be \emph{silting} (see \cite[Thm. 4.3]{GNP_HereditaryExtriangulated}). We fix a presilting object $U\in\C$. In this section, we recall the definition of Bongartz and co-Bongartz completions of $U$ in $\C$ and we show there is an explicit bijection between indecomposable direct summands of the Bongartz and indecomposable direct summands of the co-Bongartz complements of $U$. We show that this bijection is compatible with the bijection between the indecomposable direct summands of the Bongartz complement and indecomposable direct summands of the co-Bongartz complement of $\C(P,U)$ explained in Proposition \ref{exchange_in_Lambda}. We begin with the definition of Bongartz and co-Bongartz completions of $U$. 

\begin{defprop}[{\cite[Prop. 4.6]{GNP_HereditaryExtriangulated}}]\label{BongartzCoBongartz_of_presilting}
    Let $f: U'\to I$ (resp. $g: P \to U''$) be a right (resp. left) minimal $\add U$-approximation. Then $f$ (resp. $g$) is a deflation (resp. conflation). Consider the $\EE$-triangles
    \begin{equation}\label{Bongartz&co-Bongarz_in_C}
        B[U]\overset{b}{\infl}U'\overset{f}{\defl} I \quad \text{ and } \quad P\overset{g}{\infl} U'' \overset{c}{\defl} C[U]
    \end{equation}
    Then, $B_U := B[U]\oplus U$ (resp. $C_U := C[U]\oplus U$) are silting objects. We call $B_U$ (resp. $C_U$) the \emph{Bongartz} (resp. \emph{co-Bongartz}) \emph{completion} of $U$. We refer to $B[U]$ (resp. $C[U]$) as the \emph{Bongartz} (resp. \emph{co-Bongartz}) \emph{complement} of $U$.
\end{defprop}

The following gives a way to construct a (reduced) $0$-Auslander extriangulated category from a presilting object. 

\begin{defprop}[{\cite[Section 4]{PanZhu_SiltingReduction}}]\label{C_U_is_0-Auslander}
    Let $U\in\C$ be a presilting object. Let $\C_U:= {^\perp C_U}\cap B_U^{\perp}$ and let $\widetilde{\C}_U := \C_U/[U]$. Then, $\C_U$ is $0$-Auslander and $\widetilde{\C}_U$ is a reduced $0$-Auslander extriangulated category with $\proj\C_U = \add B_U$ and $\inj\C_U = \add C_U$. In particular, $\C_U = {^\perp U}\cap U^\perp$. We refer to $\widetilde{\C}_U$ as \emph{presilting reduction} of $U$. 
\end{defprop}

The next result is an analog of \cite[Lemma 3.2]{tauExcSeq_BM} (see also \cite[Lemma 3.4]{tauExcSeq_BM}) for reduced 0-Auslander extriangulated categories. 

\begin{lemma}[{\cite[Thm. 4.19]{PanZhu_SiltingReduction}}]\label{Bongartz_co-Bongartz_Lemma}
    There is an $\EE$-triangle 
    \begin{equation}\label{Bongartz_co-Bongartz_E-triangle}
        B[U]\overset{\alpha}{\infl}\overline{U}\overset{\beta}{\defl} C[U]
    \end{equation}
    where $\overline{U}\in \add U$, and $\alpha: B[U]\to \overline{U}$ (resp. $\beta:\overline{U}\to C[U]$) is a minimal left (resp. right) $\add U$-approxiamtion. 
\end{lemma}

\begin{proof}
    Since $\C$ is reduced 0-Auslander, there is an $\EE$-triangle $P\infl 0 \defl I$. Using the $\EE$-tirangles in  \eqref{Bongartz&co-Bongarz_in_C} and \cite[Prop. 3.15]{Nakaoka-Palu_extriangulated_categories}, we obtain the following commutative diagrams of $\EE$-triangles
    
\[\begin{tikzcd}[ampersand replacement=\&,cramped]
	\& {B[U]} \& {B[U]} \&\& P \& {U''} \& {C[U]} \\
	P \& {B[U]} \& {U'} \&\& {B[U]} \& {\overline{U}} \& {C[U]} \\
	P \& 0 \& I \&\& {U'} \& {U'}.
	\arrow[equals, from=1-2, to=1-3]
	\arrow["1"', tail, from=1-2, to=2-2]
	\arrow["b", tail, from=1-3, to=2-3]
	\arrow["g", tail, from=1-5, to=1-6]
	\arrow[tail, from=1-5, to=2-5]
	\arrow["c", two heads, from=1-6, to=1-7]
	\arrow[tail, from=1-6, to=2-6]
	\arrow[equals, from=1-7, to=2-7]
	\arrow[tail, from=2-1, to=2-2]
	\arrow[equals, from=2-1, to=3-1]
	\arrow[two heads, from=2-2, to=2-3]
	\arrow[two heads, from=2-2, to=3-2]
	\arrow["f", two heads, from=2-3, to=3-3]
	\arrow["\alpha", tail, from=2-5, to=2-6]
	\arrow[two heads, from=2-5, to=3-5]
	\arrow["\beta", two heads, from=2-6, to=2-7]
	\arrow[two heads, from=2-6, to=3-6]
	\arrow[tail, from=3-1, to=3-2]
	\arrow[two heads, from=3-2, to=3-3]
	\arrow[equals, from=3-5, to=3-6]
\end{tikzcd}\]

We claim that $\overline{U}\in \add U$. Adding the $\EE$-triangles $U\overset{1}{\infl} U \defl 0$ and $0\defl U\overset{1}{\defl} U$ to the middle row of the second diagram above, we get an $\EE$-triangle of the form 

\begin{equation}\label{B_U&C_U_E-triangle}
    B_U\infl \overline{U}\oplus U^2 \defl C_U.
\end{equation}

Let $\C_U = B_U^\perp\cap {^\perp C_U} = \{X\in\C\mid \EE(B_U,X) = 0 = \EE(X,C_U)\}.$
Then $\C_U$ is a 0-Auslander extriangulated subcategory of $\C$ by Definition-Proposition \ref{C_U_is_0-Auslander}. We have that the $\EE$-triangle \eqref{B_U&C_U_E-triangle} is an $\EE_{\C_U}$-triangle. Indeed, $B_U$ and $C_U$ lie in $\C_U$. Moreover, applying $\C(B_U,-)$ and $\C(-,C_U)$ to \eqref{B_U&C_U_E-triangle} we obtain two exact sequences of abelian groups 
$$\EE(B_U,B_U)\to \EE(B_U, \overline{U})\to \EE(B_U,C_U) \quad \text{and} \quad \EE(C_U,C_U)\to \EE(\overline{U},C_U)\to \EE(B_U,C_U),$$
where the outer terms of each sequence are zero by \cite[Lemma 4.13]{PanZhu_SiltingReduction} and the fact that both $B_U$ and $C_U$ are silting objects. Hence, $\EE(B_U, \overline{U})=0= \EE(\overline{U}, C_U)$. This shows that $\overline{U}$ also lies in $\C_U$ and thus \eqref{B_U&C_U_E-triangle} is an $\EE_{\C_U}$-triangle. Next we prove that $\overline{U}\in \add U$. By \cite[Lemma 4.12]{PanZhu_SiltingReduction}, this is equivalent to showing that $\overline{U}\in \proj\C_U\cap \inj\C_U = \add B_U\cap \add C_U = \add U$. Let $Y\in \C_U$. Applying $\C_U(-,Y)$ and $\C_U(Y,-)$ to the $\EE_{\C_U}$-triangle $U''\infl \overline{U}\defl U'$ we get that $\EE_{\C_U}(\overline{U},Y) = 0 = \EE_{\C_U}(Y,\overline{U})$ and the claim follows. It remains to show that $b: B[U]\to \overline{U}$ (resp. $c: \overline{U}\to C[U]$) is a minimal left (resp. right) $\add U$-approximation in $\C$. Apply $\C(-,U)$ to \eqref{Bongartz_co-Bongartz_E-triangle} to get an exact sequence 
$$\C(C[U],U)\to \C(\overline{U},U)\to \C(B[U],U)\to\EE(C[U],U) = 0,$$ 
where the last equality follows from the fact that $C_U = C[U]\oplus U$ is a silting object. Hence $ \C(\overline{U},U)\to \C(B[U],U)$ is an epimorphism and thus $b: \overline{U}\to B[U]$ is a left $\add U$-approximation. Since $\add B[U]\cap \add\overline{U} = 0$, it follows that $b$ is minimal. The argument for the proof of $c: \overline{U}\to C[U]$ being a minimal right $\add U$-approximation is dual. This concludes the proof.  
\end{proof}

The next result gives an explicit bijection between the indecomposable direct summands of $B[U]$ and the indecomposable direct summands of $C[U]$. 

\begin{proposition}\label{exchange_in_C}
    Let $B = B[U]$ and $C= C[U]$. Write $B= \bigoplus_iB_i$ and $C = \bigoplus_iC_i$ as direct sum of indecomposable objects. Let $b_i: B_i\infl \overline{U}_i$ be a minimal left $\add U$-approximation. Then, there is an $\EE$-triangle
    \begin{equation*}
    B_i\overset{b_i}{\infl} \overline{U}_i\overset{c_i}{\defl}C_i.
    \end{equation*}
where $c_i: \overline{U}_i\to C_i$ is a minimal right $\add U$-approxiamtion. The assignment $C_i\mapsto B_i$ defines a bijection between the indecomposable direct summands of the Bongartz complement $B$ and the indecomposable direct summands of the co-Bongartz complement $C$ of $U$. 
\end{proposition}

\begin{proof}
    By Lemma \ref{Bongartz_co-Bongartz_Lemma} there is an $\EE$-triangle $B\overset{b}{\infl}\overline{U}\overset{c}{\defl} C$. It easy to see that $b = \bigoplus_i b_i$ and $c = \bigoplus_i c_i$. Hence, the $\EE$-triangle \eqref{B_U&C_U_E-triangle} is the direct sum of $\abs{P}-\abs U$ $\EE$-triangles $B_i\overset{b_i}{\infl} \overline{U}_i\overset{c_i}{\defl}C_i$, with $b_i$ (resp. $c_i$) minimal left (resp. right) $\add U$-approximation. Since $B$ is unique up to isomorphisms by \cite[Rmk. 3.10]{Nakaoka-Palu_extriangulated_categories}, the result follows. 
\end{proof}

The following links presilting objects in $\C$ and support $\tau$-rigid objects in $\CLambda(\Lambda)$. For an object $X\in\C$, write $X = X'\oplus I_X$ where $I_X$ is the largest injective direct summand of $X$.  

\begin{proposition}\label{presilting_supp.tau-rigid}
    The following statements hold. 
    \begin{itemize}
        \item[(i)]\cite[Prop. 4.5, Cor. 4.6]{PanZhu_SiltingReduction} There is a bijection 
        $$H_P: \presilt\C\to \staurigid\Lambda$$ given by $(X'\oplus I_X)\mapsto \C(P,X')\oplus\C(P,\Omega I_X)[1]$. In particular, $H_P$ induces a bijection $\silt\C\to \stautilt\Lambda$; 
        \item[(ii)]\cite[Cor. 4.14]{PanZhu_SiltingReduction} For $U\in \presilt\C$, the Bongartz (resp. co-Bongartz) completion of $U$ corresponds, under the correspondence in (i), to the Bongartz completion (resp. co-Bongartz) completion of $H_P(U)$. In other words, $H_P(B_U) \cong B_{H_P(U)}$ and $H_P(C_U) \cong C_{H_P(U)}$.
    \end{itemize}
\end{proposition}

For the remaining part of this section, assume that $U\in\presilt\C$ has no injective direct summands so that $H_P(U) = \C(P,U)\in\modd\Lambda$. Let $B = B[U]$ and $C = C[U] = C'\oplus I_C $ be the Bongartz and co-Bongartz complement of $U$ in $\C$, respectively. Our goal is to show that the bijection in Proposition \ref{exchange_in_Lambda} and the bijection in Proposition \ref{exchange_in_C} are compatible via the bijection $H_P$. 

\begin{lemma}\label{compatibilityC'}
    Let $B_i$ be an indecomposable direct summand of $B$ such that the corresponding indecomposable direct summand $C_i$ of $C$ is a summand of $C'$ (see Proposition \ref{exchange_in_C}). Then, $H_P(C_i)$ is the indecomposable direct summand of $H_P(C')$ corresponding to the indecoposable direct summand $H_P(B_i)$ of $H_P(B)$ (see Proposition \ref{exchange_in_Lambda}).  
\end{lemma}

\begin{proof}
    Let 
    \begin{equation}\label{Etriangle}
        B_i\overset{b_i}{\infl} \overline{U}_i\overset{c_i}{\defl} C_i
    \end{equation}
    be the $\EE$-triangle such that $b_i: B_i \to \overline{U}_i$ (resp. $c_i: \overline{U}_i\to C_i$) is a minimal left (resp. right) $\add U$-approximation (see Proposition \ref{exchange_in_C}). Notice that none of the terms in \eqref{Etriangle} has injective direct summands. Hence, applying $H_P = \C(P,-)$ to \eqref{Etriangle}, we obtain an exact sequence in $\modd \Lambda$
    \begin{equation*}
        H_P(B_i)\xrightarrow{H_P(b_i')} H_P(\overline{U}_i)\xrightarrow{H_P(c_i)}H_p(C_i)\to \EE(P,C_i) = 0.
    \end{equation*}
    Since $b_i: B_i\to \overline{U}_i$ (resp. $c_i: \overline{U}_i\to C_i$) is a left (resp. right) $\add U$-approximation , it follows that $H_P(b_i): H_P(B_i)\to H_P(\overline{U}_i)$ (resp. $H_P(c_i): H_P(\overline{U}_i)\to H_P(C_i)$) is a left (resp. right) $\add H_P(U)$-approximation. Minimality follows from $H_P(B_i)$ and $H_P(C_i)$ being indecomposable. We conclude that $H_P(C_i)$ is the indecomposable direct summand of $H_P(C')$ corresponding to the indecoposable direct summand $H_P(B_i)$ of $H_P(B)$ by Proposition \ref{exchange_in_Lambda}(i).
\end{proof}

\begin{lemma}\label{compatibilityI_C}
    Let $B_i$ be an indecomposable direct summand of $B$ such that the corresponding indecomposable direct summand $I_i$ of $C$ is a summand of $I_C$ (see Proposition \ref{exchange_in_C}). Then, $H_P(I_i)$ is the indecomposable direct summand of $H_P(I_C)$ corresponding to the indecomposable direct summand $H_P(B_i)$ of $H_P(B)$ (see Proposition \ref{exchange_in_Lambda}).  
\end{lemma}

\begin{proof}
     Let 
    \begin{equation}\label{Etriangle1}
        B_i\overset{b_i}{\infl} \overline{U}_i\overset{c_i}{\defl} I_i
    \end{equation}
    be the $\EE$-triangle such that $b_i: B_i\to \overline{U}_i$ (resp. $c_i: \overline{U}_i\to I_i$) is a minimal left (resp. right) $\add U$-approximation (see Proposition \ref{exchange_in_C}). Consider the $\EE$-triangle $\Omega I_i\infl 0 \defl I_i$. By \cite[Prop. 3.15]{Nakaoka-Palu_extriangulated_categories}, there is a commutative diagram of $\EE$-triangles 
    \[\begin{tikzcd}[ampersand replacement=\&,cramped]
	\& {B_i} \& {B_i} \\
	{\Omega I_i} \& {B_i} \& {\overline{U}_i} \\
	{\Omega I_i} \& 0 \& {I_i}.
	\arrow[equals, from=1-2, to=1-3]
	\arrow["1"', tail, from=1-2, to=2-2]
	\arrow["{b_i}", tail, from=1-3, to=2-3]
	\arrow["i", tail, from=2-1, to=2-2]
	\arrow[equals, from=2-1, to=3-1]
	\arrow["{b_i}", two heads, from=2-2, to=2-3]
	\arrow[two heads, from=2-2, to=3-2]
	\arrow["{c_i}", two heads, from=2-3, to=3-3]
	\arrow[tail, from=3-1, to=3-2]
	\arrow[two heads, from=3-2, to=3-3]
\end{tikzcd}\]
We claim that $i: \Omega I_i\to B_i$ is a left minimal $\add B$-approximation. Applying $\C(-,B)$ to the middle row of the above diagram, we obtain an exact sequence 
$$\C(\overline{U}_i, B)\to \C(B_i,B)\xrightarrow{\C(i,B)}\C(\Omega I_i,B)\to \EE(\overline{U}_i, B) = 0.$$
It follows that $\C(i,B): \C(B_i,B)\to\C(\Omega I_i,B)$ is surjective and thus $i: \Omega I_i\to B_i$ is a left $\add B$-approximation. Since $\Omega I_i$ and $B_i$ are indecomposable, minimality follows. Now, applying $\C(P,-)$ to the middle row of the diagram above, we obtain an exact sequence 
\begin{equation}\label{seq1}
    \C(P,\Omega I_i)\xrightarrow{\C(P,i)}\C(P,B_i)\xrightarrow{\C(P,b_i)}\C(P,\overline{U}_i)\to \EE(P,\Omega I_i) = 0,
\end{equation}
 where $\C(P,\Omega I_i)$ (resp. $\C(P,\overline{U}_i)$) is an indecomposable direct summand of $\C(P,\Omega I)$ (resp. $\C(P,U)$). In particular, since $i: \Omega I_i\to B_i$ is a minimal left $\add B$-approximation,  it follows that $\C(P,i): \C(P,\Omega I_i)\to{\C(P,B_i)}$ is a minimal left $\add\C(P,B)$-approximation. On the other hand, using Proposition \ref{exchange_in_Lambda}(ii), we have an exact sequence in $\modd\Lambda$ of the form 
 \begin{equation}\label{seq2}
     \C(P,\Omega I)_i\xrightarrow{\delta_i}\C(P,B)_i\xrightarrow{\mu_i}\overline{\C(P,U)}_i\to 0,
 \end{equation}
 where $\delta_i: \C(P,\Omega I)_i\to \C(P,B)_i$ is a minimal left $\add\C(P,B)$-approximation and $\overline{\C(P,U)}_i$ is in $\add\C(P,U)$. 

By construction, $\C(P,\Omega I_i)\cong \C(P,\Omega I)_i$ (see Proposition \ref{presilting_supp.tau-rigid}(ii)). Hence, comparing the sequences \eqref{seq1} and \eqref{seq2} we obtain a commutative diagram 
\[\begin{tikzcd}[ampersand replacement=\&,cramped,column sep=2.99em]
	{\C(P,\Omega I_i)} \& {\C(P,B_i)} \& {\C(P,\overline{U}_i)} \& 0 \\
	{\C(P,\Omega I)_i} \& {\C(P,B)_i} \& {\overline{\C(P,U)}_i} \& 0
	\arrow["{\C(P,i)}", from=1-1, to=1-2]
	\arrow[from=1-1, to=2-1]
	\arrow["{\C(P,b_i)}", from=1-2, to=1-3]
	\arrow[dashed, from=1-2, to=2-2]
	\arrow[from=1-3, to=1-4]
	\arrow[dashed, from=1-3, to=2-3]
	\arrow["{\delta_i}", from=2-1, to=2-2]
	\arrow["{\mu_i}", from=2-2, to=2-3]
	\arrow[from=2-3, to=2-4]
\end{tikzcd}\]
where the middle vertical row exists since $\C(P,i)$ is a $\C(P,B)$-approximation, and it must be an isomorphism by the minimality of $\C(P,i)$ and $\delta_i$. We conclude that $\C(P,\Omega I_i)$ is the indecomposable direct summand of $C(P,\Omega I_C)$ corresponding to the indecomposable direct summand $H_P(B_i)$ of $H_P(B)$.
\end{proof}

Summarising, we have obtained the following. 

\begin{proposition}\label{compatibility_of_exchange_in_C_and_Lambda}
    Let $U\in \presilt\C$ with no injective direct summands. Let $B = B[U]$ and $C = C[U] = C'\oplus I_C $ be the Bongartz and co-Bongartz complement of $U$ in $\C$, respectively. Then, there is an $\EE$-triangle 
    \begin{equation}\label{Etriangle1}
        B\overset{b}{\infl} \overline{U}\overset{c}{\defl}C,
    \end{equation}
    where $b: B\to \overline{U}$ (resp. $c: \overline{U}\to C$) is a minimal left (resp. right) $\add U$-approximation. The $\EE$-triangle \eqref{Etriangle1} is the direct sum of $\abs P-\abs {C'}$ $\EE$-triangles 
    \begin{equation}\label{E2}
        B'_i\overset{b_i'}{\infl} \overline{U'_i}\overset{c_i'}{\defl}C'_i
    \end{equation}
    and $\abs P -\abs{I_C}$ $\EE$-triangles 
    \begin{equation}\label{E3}
        B''_i\overset{b_i''}{\infl} \overline{U''_i}\overset{c_i''}{\defl}I_i
    \end{equation}
    where $C'_i$ (resp. $I_i$) is an indecomposable direct summand of $C$ (resp. $I_C$). Moreover, $b_i': B_i'\to \overline{U_i'}$ and $b_i'': B_i''\to \overline{U_i''}$ (resp. $c_i': \overline{U'_i}\to C_i'$ and $c_i'': \overline{U_i''}\to I_i$) are left (resp. right) minimal $\add U$-approximations. In particular, $b = \bigoplus_i (b_i'\oplus b_i'')$  and $c = \bigoplus_i (c_i'\oplus c_i'')$. Write $B = B'\oplus B''$, where $B' = \bigoplus_i B_i'$ and $B'' = \bigoplus_iB_i''$.
    \begin{itemize}
        \item[(i)] The assignment $B_i'\overset{\rho'}{\mapsto} C_i'$ established via the $\EE$-triangle \eqref{E2} gives a bijection between the indecomposable direct summands of $B$ and the indecomposable direct summands of $C$ such that $B_i'$ (resp. $C_i'$) is a summand of $B'$ (resp. $C'$). In particular, there is a commutative diagram of bijections 
        \[\begin{tikzcd}[ampersand replacement=\&,cramped]
        	{\ind\add B'} \&\& {\ind\add \C(P,B')} \\
        	{\ind\add C'} \&\& {\ind\add \C(P,C')}
        	\arrow["{H_P}", from=1-1, to=1-3]
        	\arrow["{\rho'}"', from=1-1, to=2-1]
        	\arrow["\rho", from=1-3, to=2-3]
        	\arrow["{H_P}", from=2-1, to=2-3]
        \end{tikzcd}\]
        where the map $\rho: \ind\add \C(P,C')\to \ind\add \C(P,B')$ is described in Proposition \ref{exchange_in_Lambda}(i). 
        \item[(ii)] The assignment $B_i''\overset{\rho'}{\mapsto} I_i$ established via the $\EE$-triangle \eqref{E3} gives a bijection between the indecomposable direct summands of $B$ and the indecomposable direct summands of $C$ such that $B_i''$ (resp. $I_i$) is a summand of $B''$ (resp. $I_C$). In particular, there is a commutative diagram of bijections 
        \[\begin{tikzcd}[ampersand replacement=\&,cramped]
        	{\ind\add B''} \&\& {\ind\add \C(P,\Omega B'')} \\
        	{\ind\add I_C} \&\& {\ind\add \C(P,I_C)}
        	\arrow["{H_P}", from=1-1, to=1-3]
        	\arrow["{\rho'}"', from=1-1, to=2-1]
        	\arrow["\rho", from=1-3, to=2-3]
        	\arrow["{[-1]H_P}", from=2-1, to=2-3]
        \end{tikzcd}\]
        where the map $\rho: \ind\add \C(P,B'')\to \ind\add \C(P,I_C)$ is described in Proposition \ref{exchange_in_Lambda}(ii). 
    \end{itemize}
\end{proposition}

\begin{proof}
    The result follows by combining Lemma \ref{Bongartz_co-Bongartz_Lemma},  Proposition \ref{exchange_in_C},  and Lemmas \ref{compatibilityC'}, \ref{compatibilityI_C}.
\end{proof}

\begin{remark}
    Let $\Lambda$ be a finite dimensional algebra. The category $K^{[-1,0]}(\proj\Lambda)$ of 2-term complexes of projective $\Lambda$-modules is an extension closed subcategory of the triangulated category $\K = K^b(\proj\Lambda)$ and therefore extriangulated by Example \ref{extriang_examples}. In particular, $K^{[-1,0]}(\proj\Lambda)$ is a reduced 0-Auslander extriangulated category; see \cite{GNP_HereditaryExtriangulated}. In this case, Proposition \ref{exchange_in_C} recovers the Buan--Marsh bijection described in Proposition \ref{exchange_in_Lambda}. 
\end{remark}

\section{Compatibility of Pan--Zhu and Buan--Marsh bijections}\label{Sec:Compatibility}

Let $U\in \C$ be a presilting object. Let $\C_U = B_U^{\perp}\cap{^\perp}C_U$  and denote by $\widetilde{\C}_U$ the additive quotient $\C_U/[U]$. The category $\C_U$ coincides with the extension-closed subcategory $U^\perp \cap {}^\perp U$ considered in \cite{GNP_HereditaryExtriangulated}; see \cite[Remark 4.20(2)]{PanZhu_SiltingReduction}. Recall from Definition-Proposition \ref{C_U_is_0-Auslander} that $\C_U$ and $\widetilde{\C}_U$ are 0-Auslander and reduced 0-Auslander extriangulated categories, respectively. Let $\presilt_U\C = \{X\in\presilt\C \mid U\in\add X\}$. Note that $\presilt_U\C$ is isomorphic to the set 
$$\{X\in \C\setminus\add U\mid X\oplus U \in\presilt\C\} = \presilt\C_U\setminus\add U.$$ 

Denote by $\pi_U: \C_U\to \widetilde{\C}_U$ the projection functor. We have the following observation 

\begin{lemma}\label{proj_functior_bij}
    The projection functor $\pi_U: \C_U\to \widetilde{\C}_U$ induces a bijection 
    $$\pi_U: \ind\presilt\C_U\setminus\add U \to \{Y\in\ind\widetilde{\C}_U\mid  Y \text{ is presilting in } \widetilde{\C}_U\}.$$
\end{lemma}

\begin{proof}
 The domain of the above map coincides with $\presilt_U\C$. By \cite[Lemma 4.13, Thm. 3.4]{PanZhu_SiltingReduction}, $\silt_U\C = \silt\C_U$. Siltig objects coincide with maximal presilting objects in 0-Auslander extriangulated categories (see \cite[Section 4]{GNP_HereditaryExtriangulated}) and thus $\presilt_U\C = \presilt\C_U$. By \cite[Thm. 4.16]{PanZhu_SiltingReduction} the projection functor $\pi_U$ induces a bijection between $\silt\C_U$ and $\silt\widetilde{\C}_U$. The result follows.
\end{proof}

Inspired by the compatibility of the Pan--Zhu silting reduction with the Jasso reduction of $\tau$-tilting modules (see \cite[Thm. 4.22]{PanZhu_SiltingReduction}), and the results in \cite[Section 3]{borve2021two}, this section aims to link the Pan--Zhu silting reduction with the Buan--Marsh reduction of support $\tau$-rigid objects described in Proposition \ref{epsilonU}. 

The next result can be seen as a 0-Auslander version of \cite[Lemma 3.4]{borve2021two}. We include the proof for completeness.

\begin{lemma}\label{H_tilda}
    There is an equivalence defined by the composite 
    \begin{equation}\label{equiv}
        \widetilde{\C}_U/[C_U]\xrightarrow{\widetilde{\C}_U(B_U, -)}\modd\Gamma_{H_P(U)}\xrightarrow{F_{H_P(U)}^{-1}} J(H_P(U)).
    \end{equation}
    The induced composite bijection $\widetilde{H}_{B_U}$ 
    \begin{equation}\label{eq_tilde}
        \presilt\widetilde{\C}_U\xrightarrow{H_{B_U}}\staurigid\Gamma_{H_P(U)}\xrightarrow{F_{H_P(U)}^{-1}}\staurigid J(H_P(U))
    \end{equation}
    is given by 
    \begin{equation}\label{Htilda_eq}
        \widetilde{H}_{B_U}(X) = \widetilde{H}_{B_U}(X'\oplus I_X) = f_{\C(P,U)}\C(P,X')\oplus f_{\C(P,U)}(\C(P,\Omega_{\widetilde{\C}_U}I_X))[1],
    \end{equation}
    where $I_X\in\inj\widetilde{\C}_U$ and $\Omega_{\widetilde{\C}_U}$ denotes the equivalence $\Omega$ in the reduced 0-Auslander extriangulated category $\widetilde{\C}_U$.
\end{lemma}

\begin{proof}
    The first equivalence in the composite \eqref{equiv} is given in \cite[Subsec. 4.4]{PanZhu_SiltingReduction} while the second equivalence follows from \cite[Thm. 4.12(b)]{Beyond_tau-tilting_theory} (see also \cite[Thm. 3.8]{Jasso_Reduction}).  To prove the expression \eqref{Htilda_eq}, we show that
     $$F_{H_P(U)}(f_{\C(P,U)}\C(P,X'))\cong \widetilde{\C}_U(B_U, X').$$ Write $U = U'\oplus I_U$, where $I_U$ is the largest injective direct summand of $U$. Let $B_U = (B[U]\oplus U')\oplus I_U$ be the Bongartz completion of $U$ in $\C$. By Proposition \ref{presilting_supp.tau-rigid}(ii), we have that
    \begin{align*}
        B_{H_P(U)}&= (B[H_P(U)]\oplus H_P(U'))\oplus H_P(I_U)\\
        &\cong (H_P(B[U])\oplus H_P(U'))\oplus H_P(I_U)\\
        &\cong H_P(B_U). 
    \end{align*}
    Recall from \cite[Thm. 4.12(b)]{Beyond_tau-tilting_theory} that $F_{H_P(U)} = \Hom_\Lambda(H_P(B[U])\oplus H_P(U'), -)$. Then, 
    \begin{align*}
        F_{H_P(U)}(f_{\C(P,U)}\C(P,X')) &= \Hom_\Lambda(H_P(B[U])\oplus H_P(U'), f_{\C(P,U)}\C(P,X'))\\
        &  \cong \Hom_\Lambda(\C(P, B_U), f_{\C(P,U)}\C(P,X'))\\
        &\cong \widetilde{\C}_U(B_U, X'), 
    \end{align*}
    where the last isomorphism is contained in the proof of \cite[Thm. 4.22]{PanZhu_SiltingReduction}. Thus, the desired expression for $\widetilde{H}_{B_U}(X')$ follows. The proof for the expression $\widetilde{H}_{B_U}(I_X)$ is similar. 
\end{proof}

The goal of this section is to prove the following compatibility theorem. 

\begin{theorem}\label{compatibility_thm}
    Let $U$ be a presilting object in $\C$. There is a commutative diagram of bijections given by 
    \begin{equation}\label{compatibility_square}
    \begin{tikzcd}[ampersand replacement=\&,cramped]
    	{\ind\presilt \C_U \setminus\add U} \& {\{W\in\ind\CLambda(\Lambda)\setminus\add H_P(U)\mid W\oplus H_P(U) \text{ is supp. }\tau\text{-rigid}\}} \\
    	{\ind\presilt\widetilde{\C}_U} \& {\{Z\in\ind\CLambda(J(H_P(U)))\mid Z \text{ is supp. }\tau_{J(H_P(U))}\text{-rigid}\}.}
    	\arrow["{H_P}", from=1-1, to=1-2]
    	\arrow["{\pi_U}"', from=1-1, to=2-1]
    	\arrow["{\Eps_{H_P(U)}}", from=1-2, to=2-2]
    	\arrow["{\widetilde{H}_{B_U}}", from=2-1, to=2-2]
    \end{tikzcd}
    \end{equation}
\end{theorem}

Guided by the definition of the map $\Eps_{H_P(U)}$ revised in Section \ref{Sec_signed_tau_exc_seq}, the strategy to prove Theorem \ref{compatibility_thm} is to ``decompose" the square above into smaller commutative diagrams. The first step to achieve this goal is the following lemma. 

\begin{lemma}\label{commutativity_of_piU+V}
    Let $U,V\in\presilt\C$ be such that $U\oplus V$ also lies in $\presilt\C.$ Then there is a commutative diagram of bijections
    \[\begin{tikzcd}[ampersand replacement=\&,cramped]
	{\ind\presilt \C_{U\oplus V}\setminus\add(U\oplus V)} \& {\ind\presilt{\left(\widetilde{\C}_U\right)_{\pi_U (U\oplus V)}\setminus\add \pi_U(V)}} \\
	\& {\ind\presilt\widetilde{\C}_{U\oplus V}.}
	\arrow["{\pi_U}", from=1-1, to=1-2]
	\arrow["{\pi_{U\oplus V}}"', from=1-1, to=2-2]
	\arrow["{\pi_{\pi_U(V)}}", from=1-2, to=2-2]
\end{tikzcd}\]
\end{lemma}

\begin{proof}
    The proof follows a similar argument to that of \cite[Lemma 2.4]{borve2021two}. Since $U\oplus V\in\presilt\C$, it follows that $U\oplus V\in\C_U$, where $\C_U = U^{\perp}\cap{^\perp U}$ by \cite[Rmk. 4.20(2)]{PanZhu_SiltingReduction}. We have that $\pi_U(U\oplus V) = \pi_U(U)\oplus\pi_U(V) = \pi_U(V) = V$ in $\widetilde{\C}_U$. Then, 
    \begin{align*}
        \left(\widetilde{\C}_U\right)_{\pi_U(V)} = \left(\widetilde{\C}_U\right)_{V} &= {^{\perp_{\widetilde{\C}_U}}V}\cap {V^{\perp_{\widetilde{\C}_U}}}\\
        &= \{X\in\widetilde{\C}_U\mid \EE_{\widetilde{\C}_U}(X,V) = 0 = \EE_{\widetilde{\C}_U}(V,X)\}\\
        &= \left\{ X\in \C \mid \EE(X,V)\big|_{\widetilde{\C}_U}= 0= \EE(V,X)\big|_{\widetilde{\C}_U}\right\}\\
        &= ({^\perp(V\oplus U)\cap (V\oplus U)^\perp})/[U].
    \end{align*}
    Hence we get
    \begin{align*}
        \widetilde{\left(\widetilde{\C}_U\right)}_{\pi_U(V)} = \widetilde{\left(\widetilde{\C}_U\right)}_V &= \frac{({^\perp(V\oplus U)\cap (V\oplus U)^\perp})/[U]}{[V]}\\
        &\cong ({^\perp(V\oplus U)\cap (V\oplus U)^\perp})/[U\oplus V]\\
        &= \widetilde{\C}_{U\oplus V}.
    \end{align*}
    Now, all the arrows are bijections by Lemma \ref{proj_functior_bij} and the diagram clearly commutes. 
\end{proof}

In light of Lemma \ref{commutativity_of_piU+V} and the definition of $\Eps_{H_P(U)}$ via the commutative diagram \eqref{Eps_U_via_comm_diagram}, we can "decompose" the square in Theorem \ref{compatibility_thm} into smaller diagrams as follows.

\begin{equation}\label{big-diagram}
\hspace*{-7.em}%
\scalebox{0.83}{%
\begin{tikzcd}[ampersand replacement=\&,cramped, row sep=2.3em,column sep=1em]
        {{\ind\presilt\C_{U}}\setminus\add U} \&\&\& {{\left\{ X\in \ind\CLambda(\Lambda)\setminus\add H_P(U) \;\middle|\; \begin{tabular}{@{}l@{}} $X\oplus H_P(U)$ \\ \text{is supp. $\tau$-rigid} \end{tabular} \right\}}} \\
    \&\& {{\left\{ Y_1\in \ind\CLambda(J(H_P(U'))) \;\middle|\; \begin{tabular}{@{}l@{}} $Y_1\oplus \widetilde{P}[1]$ \\ \text{is supp. $\tau$-rigid} \end{tabular} \right\}}}\\
    \& {\ind\presilt\left(\widetilde{\C}_{U'}\right)_{I_U}}\setminus\add \pi_{U'}(I_U)\& {{\left\{ Y_2\in\ind\CLambda(\Gamma_{H_P(U')})\setminus\add P'[1] \;\middle|\; \begin{tabular}{@{}l@{}} $Y_2\oplus{P'}[1]$ \\ \text{is supp. $\tau$-rigid} \end{tabular} \right\}}} \\
    \&\& {{{\{Y_3\in \ind\CLambda(J_{\Gamma_{H_P(U')}}(P'[1]))\mid Y_3 \text{ is supp. $\tau$-rigid}\}}} }\\
    \&\& 
    \begin{array}{c}
    {{\left\{Y_4\in\ind\CLambda\left(\Gamma^{\Gamma_{H_P(U')}}_{P'[1]}\right)\mid Y_4 \text{ is supp. $\tau$-rigid}\right\}}}\\ \cong \\ {{\{Y_4\in \ind\CLambda(\Gamma_{H_P(U)})\mid Y_4 \text{ is supp. $\tau$-rigid}\}}}
    \end{array}\\ 
    {{\ind\presilt{\widetilde{(\widetilde{\C}_{U'})_{I_U}}}= \ind\presilt\widetilde{\C}_U} } \&\&\& {{{\{Y\in \ind\CLambda(J({H_P(U))})\mid Y \text{ is supp. $\tau$-rigid}\}}}}
    \arrow["{H_P}", from=1-1, to=1-4]
    \arrow["{\texttt{(a)}}"{description}, curve={height=18pt}, draw=none, from=1-1, to=2-3]
    \arrow["{\pi_{U'}}", from=1-1, to=3-2]
    \arrow["{\pi_U}"', from=1-1, to=6-1]
    \arrow["{\texttt{(b)}}"{pos = 0.45}, shift left=3, curve={height=-18pt}, draw=none, from=1-1, to=6-1]
    \arrow["{\Eps_{H_P(U')}}", from=1-4, to=2-3]
    \arrow["{\texttt{(g)}}"{description}, shift left=5, curve={height=-24pt}, draw=none, from=1-4, to=4-3]
    \arrow["{\Eps_{H_P(U)}}", from=1-4, to=6-4]
    \arrow["{F_{H_P(U')}}", from=2-3, to=3-3]
    \arrow["{\widetilde{H}_{B_{U'}}}", from=3-2, to=2-3]
    \arrow["{\texttt{(c)}}"{pos = 0.45}{description}, shift right=5, curve={height=30pt}, no head, draw = none, from=2-3, to=3-3]
    \arrow["{H_{B_{U'}}}"', from=3-2, to=3-3]
    \arrow["{\pi_{\pi_{U'}(I_U)}}", from=3-2, to=6-1]
    \arrow["{\Eps_{P'[1]}^{\Gamma_{H_P(U')}}}", from=3-3, to=4-3]
    \arrow["{F_{P'[1]}^{\Gamma_{H_P(U')}}}", from=4-3, to=5-3]
    \arrow["{F_{H_P(U)}^{-1}}", from=5-3, to=6-4]
    \arrow["{\texttt{(f)}}"{description, pos=0}, shift right=5, curve={height=30pt}, draw=none, from=5-3, to=6-4]
    \arrow["{\widetilde{H}_{B_{I_U}^{\widetilde{\C}_{U'}}}}", from=6-1, to=4-3]
    \arrow["\texttt{(d)}"{pos = 0.15}{description}, shift left=3, curve={height=-30pt}, draw=none, from=3-2, to=6-1]
    \arrow["{\texttt{(e)}}"{description}, shift left=4, curve={height=18pt}, draw=none, from=6-1, to=4-3]
    \arrow["{{H}_{B_{I_U}^{\widetilde{\C}_{U'}}} = H_{B_U}}"{description}, curve={height=6pt}, from=6-1, to=5-3]
    \arrow["{\widetilde{H}_{B_U}}"', from=6-1, to=6-4]
\end{tikzcd}%
}
\end{equation}

where $\widetilde{P}[1] = \Eps_{H_P(U')}(H_P(I_U))$ and $P'[1] = F_{H_P(U')}(\widetilde{P})[1]$. Note that, since $\Gamma^{\Gamma_{H_P(U)}}_{P'[1]} \cong \Gamma_{H_P(U)}$ by Lemma \ref{Gamma_P'[1]^Gamma_M=Gamma_U}, we can identify ${H}_{B_{I_U}^{\widetilde{\C}_{U'}}}$ with ${H}_{B_U}$. 

Showing that the outer square commutes is equivalent to proving that each of the subdiagrams is commutative. Observe that it suffices to show that \eqref{big-diagram}\texttt{(a)}, \texttt{(d)} commute. Indeed, the commutativity of \eqref{big-diagram}\texttt{(c)}, \texttt{(e)}, and \texttt{(f)} follows from Lemma \ref{H_tilda}; while \eqref{big-diagram}\texttt{(b)} and \texttt{(g)} commutes by Lemma \ref{commutativity_of_piU+V} and the definition of  $\Eps_{H_P(U)}$ (see diagram \eqref{Eps_U_via_comm_diagram}), respectively. We prove the commutativity of \texttt{(a)} and \texttt{(d)} in separate lemmas. We need the following observation first. 

\begin{lemma}\label{H_P(X)inGenH_P(U)}
    Let $X\in\ind\presilt{\C}_{U'}$ be such that  $X$ is not in $\inj\C$. Then, $H_P(X)\in\Gen H_P(U')$ if and only if $X\in\inj\widetilde{\C}_{U'}$.
\end{lemma}

\begin{proof}
    Let $X\in \inj\widetilde{\C}_{U'}$. Then, there exists an $\EE_{\widetilde{\C}_{U'}}$-triangle of the form $\Omega_{\widetilde{\C}_{U'}}(X) \infl 0 \defl X$. Hence, the existence of an $\EE$-triangle $\Omega_{\widetilde{\C}_{U'}}(X) \infl U'' \defl X$ with $U''\in \add U'$. Applying $H_P(-)$, we obtain an exact sequence $H_P\left(\Omega_{\widetilde{\C}_{U'}}(X)\right) \to H_P(U'')\to H_P(X)\to 0$ in $\modd\Lambda$ with $H_P(U'')\in\add H_P(U')$. Thus, $H_P(X)\in\Gen H_P(U')$.

    Conversely, let $H_P(X)\in\Gen H_P(U')$. Then, there is an epimorphism $g: \C(P,U')^r\to\C(P,X)$ for some $r\ >0$. Let $f: \C(P,U'')\to \C(P,X)$ be a right $\add\C(P, U')$-approximation. Since $\C(P,U'')$ is in $\add \C(P,U')$, it follows that there exists a morphism $h: \C(P,U')^r\to \C(P,U'')$ such that $g = f\circ h$, and because $g$ is an epimorphism, so is $h$. Since $\C(P,-)$ induces an equivalence between $\C/[I]$ and $\modd\Lambda$, and both $X$ and $U''$ have no injective direct summands, $f$ lifts to a morphism  $f': U''\to X$ in $\C$. In particular, $f'$ is a right $\add U'$-approximation. Since $\C(P,f')$ is surjective, $f'$ is a deflation by \cite[Lemma 5.12]{FGPPP_Extriangulated_Ideal_Quotients_FGPPP}. Complete $f'$ to an $\EE$-triangle $K\infl U''\overset{f'}{\defl} X$.  We claim that $K$ lies in $\C_{U'}$. Applying $\C(U',-)$ and using that $f'$ is a right $\add U'$-approximation we obtain an exact sequence $\C(U',U'')\to\C(U',X)\to \EE(U',K) = 0$. Applying $\C(-,U')$ we get an exact sequence $\EE(U'',U')\to \EE(K,U')\to 0$. Since $\EE(U'',U')=0$, it follows that $\EE(K,U') = 0$ giving the desired claim. Now, let $Y\in \C_{U'}$. Then we obtain an exact sequence $$\EE_{\C_{U'}}(Y,K)\to \EE_{\C_{U'}}(Y,U')\to \EE_{\C_{U'}}(Y,X)\to 0.$$
    Since $Y\in\C_{U'}$, we get $\EE_{\C_{U'}}(Y,U') = 0$ and thus $\EE_{\C_{U'}}(Y,X) = 0$. The result follows.
\end{proof}

The next result proves the commutativity of the diagram \eqref{big-diagram}\texttt{(a)}.

\begin{lemma}\label{comm_of_(a)}
    The following diagram of bijections commutes
    \[\begin{tikzcd}[ampersand replacement=\&,cramped]
	{{\ind\presilt\C_{U}}\setminus\add U} \&\& {{{\{Y\in \ind\CLambda(\Lambda)\setminus\add H_P(U)\mid Y\oplus H_P(U) \text{ is supp. $\tau$-rigid}\}}}} \\
	{{\ind\presilt\widetilde{\C}_{U'}}\setminus\add{\pi_{U'}(I_U)}} \&\& {{{\{Y\in\ind\CLambda(J(H_P(U')))\mid Y\oplus \widetilde{P}[1] \text{ is supp. $\tau$-rigid}\}}}}.
	\arrow["{H_P}", from=1-1, to=1-3]
	\arrow["{\pi_{U'}}"', from=1-1, to=2-1]
	\arrow["{\Eps_{H_P(U')}}", from=1-3, to=2-3]
	\arrow["{\widetilde{H}_{B_{U'}}}"', from=2-1, to=2-3]
\end{tikzcd}\]
\end{lemma}

\begin{proof}
    We distinguish three cases (see Proposition \ref{epsilonM}). 
    
    \vspace{3pt}
    
    \textbf{Case 1: } $X\in\ind\presilt\C_U$ non injective and such that $H_P(X)\notin\Gen H_P(U')$.

    By Proposition \ref{epsilonM}, we have 
    $$\Eps_{H_P(U')}(H_P(X)) = \Eps_{\C(P,U')}(\C(P,X)) = f_{\C(P,U')}\C(P,X),$$
    On the other hand, by Lemma \ref{H_tilda}
    $$\widetilde{H}_{B_{U'}}(\pi_{U'}(X)) = \widetilde{H}_{B_{U'}}(X) = f_{\C(P,U')}\C(P,X).$$
    and the diagram commutes in this case. 

    \vspace{3pt}

    \textbf{Case 2: } $X = I_X\in\ind\presilt\C_U$ is injective in $\C$ and $H_P(I_X) = \C(P,\Omega I_X)[1]\in(\ind\proj\Lambda)[1]$. Then, 
    \begin{align*}
        \Eps_{H_P(U')}(H_P(I_X)) &= \Eps_{\C(P,U')}(\C(P,\Omega I_X)[1]) \\[2pt]
        &= f_{\C(P,U')}\left(B_{\C(P,U')}^{\C(P,\Omega I_X)}\right)[1] &&\text{by Proposition \ref{epsilonM}} \\[2pt]
        &= f_{\C(P,U')}\left(B_{H_P(U')}^{H_P(I_X)[-1]}\right)[1] &&\text{by Definiton of $H_P$}\\[2pt]
        &= f_{\C(P,U')}\left(H_P(B_{U'}^{I_X})\right)[1] &&\text{by Proposition \ref{compatibility_of_exchange_in_C_and_Lambda}(ii)}\\[2pt]
        &= f_{\C(P,U')}(\C(P,B_{U'}^{I_X}))[1] &&\text{since $B_{U'}^{I_X}\notin\inj\C$}.
    \end{align*}
    On the other hand, since $I_X$ lies in $\inj\C$ and thus $I_X$ also lies in $\inj\widetilde{\C}_{U'}$, Lemma \ref{H_tilda} gives
    $$\widetilde{H}_{B_{U'}}(\pi_{U'}(I_X)) = \widetilde{H}_{B_{U'}}(I_X) = f_{\C(P,U')}\C(P,\Omega_{\widetilde{\C}_{U'}}I_X)[1].$$
    We claim $B_{U'}^{I_X}\cong\Omega_{\widetilde{\C}_{U'}}I_X$. We know that $\widetilde{\C}_{U'}$ is a reduced 0-Auslander extriangulated category with $\proj\widetilde{\C}_{U'} = \add B[U']$ and $\inj\widetilde{\C}_{U'} = \add C[U']$. Hence, $\Omega_{\widetilde{\C}_{U'}}$ induces an equivalence of categories $\Omega_{\widetilde{\C}_{U'}}: \add C[{U'}]\to \add B[{U'}]$. Since $I_X\in\inj\widetilde{\C}_{U'}$, it follows that $I_X$ is a non-zero indecomposable direct summand of $C_{U'}$. Hence, there is an $\EE_{\widetilde{\C}_{U'}}$-triangle of the form $\Omega_{\widetilde{\C}_{U'}}I_X\infl 0 \defl I_X$. Moreover, by Proposition \ref{exchange_in_C}, there is an $\EE$-triangle $$B_{U'}^{I_X}\overset{f}{\infl} \overline{U}_{I_X}\overset{g}{\defl} I_X,$$ where $f$ (resp. $g$) is a minimal left (resp. right) $\add U'$-approximation. Since all the terms are in $\widetilde{\C}_{U'}$, the above is an $\EE_{\widetilde{\C}_{U'}}$-triangle. Hence we obtain a commutative diagram of $\EE_{\widetilde{\C}_{U'}}$-triangles
    \[\begin{tikzcd}[ampersand replacement=\&,cramped]
	{\Omega_{\widetilde{\C}_{U'}}I_X} \& 0 \& {I_X} \\
	{B_{U'}^{I_X}} \& {\overline{U}_{I_X}} \& {I_X}
	\arrow[from=1-1, to=1-2]
	\arrow["a"', dashed, from=1-1, to=2-1]
	\arrow[from=1-2, to=1-3]
	\arrow["\cong", from=1-2, to=2-2]
	\arrow[equals, from=1-3, to=2-3]
	\arrow["f", from=2-1, to=2-2]
	\arrow["g", from=2-2, to=2-3]
\end{tikzcd}\]

where the middle vertical arrow is an isomorphism in $\widetilde{\C}_{U'}$ since $\overline{U}_{I_X}\in\add U'$, and the left vertical arrow exists by the axiom $(\text{ET}3)^{\mathrm{op}}$. Since the middle and right vertical arrows are isomorphisms in $\widetilde{\C}_{U'}$, it follows that $a: \Omega_{\widetilde{\C}_{U'}}I_X\to B_{U'}^{I_X}$ is also an isomorphism in $\widetilde{\C}_{U'}$ by \cite[Cor. 3.6]{Nakaoka-Palu_extriangulated_categories}, and therefore an isomorphism in $\C$ as $\Omega I_X$ and $B_{U'}^{I_X}$ have no non-zero direct summands in $\add U'$. 

Thus, 
\begin{align*}
    \widetilde{H}_{B_{U'}}(\pi_{U'}(I_X)) &= f_{\C(P,U')}\C(P,\Omega_{\widetilde{\C}_{U'}}I_X)[1]\\
    &\cong f_{\C(P,U')}\C(P,B_{U'}^{I_X})[1]\\
    &= \Eps_{H_P(U')}(H_P(I_X)),
\end{align*}
and the diagram commutes in this case. 

\vspace{3pt}

\textbf{Case 3: } $X\in\ind\presilt\C_U$ such that $H_P(X)\in\Gen H_P(U')$. 

By Lemma \ref{H_P(X)inGenH_P(U)}, we have $X\in\inj\widetilde{\C}_{U'}$. Hence,
$$\widetilde{H}_{B_{U'}}(\pi_{U'}(X)) = \widetilde{H}_{B_{U'}}(X) = f_{\C(P,U')}\C(P,\Omega_{\widetilde{\C}_{U'}}X)[1]$$
by Lemma \ref{H_tilda}. Since 
$$\Eps_{H_P(U')}(H_P(X)) = f_{\C(P,U')}\C(P,B_{U'}^X)[1]$$ by Proposition \ref{epsilonM}, the commutativity of the square in this case follows by Case 2 above. This concludes the proof. 
\end{proof}

We continue with the commutativity of the diagram \eqref{big-diagram}\texttt{(d)}. 

\begin{lemma}\label{comm_of_(d)}
     The following diagram of bijections commutes
    \[\begin{tikzcd}[ampersand replacement=\&,cramped]
	{\ind\presilt(\widetilde{\C}_{U'})_{I_U}\setminus\add \pi_{U'}(I_U)} \& {{\{Y\in \ind\CLambda({\Gamma_{H_P(U')}})\setminus\add P'[1]\mid Y\oplus{P'}[1] \text{ is supp. $\tau$-rigid}\}}} \\
	{\ind\presilt{\widetilde{(\widetilde{\C}_{U'})_{I_U}}}= \ind\presilt\widetilde{\C}_U} \& {\{Y\in\ind\CLambda(J_{\Gamma_{H_P(U')}}(P'[1]))\mid Y \text{ is supp. $\tau_{J_{\Gamma_{H_P(U')}}(P'[1])}$-rigid}\}}
	\arrow["{H_{B_{U'}}}", from=1-1, to=1-2]
	\arrow["{\pi_{\pi_{U'}(I_U)}}"', from=1-1, to=2-1]
	\arrow["{\Eps_{P'[1]}^{\Gamma_{H_P(U')}}}", from=1-2, to=2-2]
	\arrow["{\widetilde{H}}_{B_{I_U}^{\widetilde{\C}_{U'}}}", from=2-1, to=2-2]
    \end{tikzcd}\]
    where $P'[1] = F_{H_P(U')}(\Eps_{H_P(U')}(H_P(I_U)))[1]$.
\end{lemma}

\begin{proof}

    Since $I_U\in\inj\C$ and $I_U\in\widetilde{\C}_{U'}$, it follows that $I_U\in\inj\widetilde{\C}_{U'}$. Recall that $\widetilde{\C}_{U'}$ is reduced $0$-Auslander, and thus $\widetilde{\C}_{U'}\left(\proj\widetilde{\C}_{U'}, \inj\widetilde{\C}_{U'}\right) = 0$. Hence, since $\proj\widetilde{\C}_{U'} = \add B[U]$, we obtain $\widetilde{\C}_{U'}(B_{U'},I_U) = 0$.
    By Lemma \ref{H_tilda} applied in the reduced 0-Auslander extriangulated category $\widetilde{\C}_{U'}$, we have 
    \begin{align*}
        {\widetilde{H}}_{B_{I_U}^{\widetilde{\C}_{U'}}}(X'\oplus I_X) &= f^{\Gamma_{H_P(U')}}_{\widetilde{\C}_{U'}(B_{U'},I_U)}{\widetilde{\C}_{U'}(B_{U'},X')}\oplus f^{\Gamma_{H_P(U')}}_{\widetilde{\C}_{U'}(B_{U'},I_U)}{\widetilde{\C}_{U'}(B_{U'},\Omega_{\widetilde{(\widetilde{\C}_{U'})}_{I_U}}} I_X)[1]\\
        &={\widetilde{\C}_{U'}(B_{U'},X')}\oplus {\widetilde{\C}_{U'}(B_{U'},\Omega_{\widetilde{(\widetilde{\C}_{U'})}_{I_U}}} I_X)[1], 
    \end{align*}
    where the last equality uses the fact that ${\widetilde{\C}_{U'}(B_{U'},I_U)}=0$. We distinguish two cases (see Proposition \ref{epsilonP[1]}). 

    \vspace{3pt}

    \textbf{Case 1: } $X\in\ind\left(\presilt(\widetilde{\C}_{U'})_{I_U}\setminus\add \pi_{U'}(I_U)\right)$ not in $\inj\widetilde{\C}_U$. 

    By Proposition \ref{epsilonP[1]}, 
    $$\Eps_{P'[1]}^{\Gamma_{H_P(U')}}(H_{B_{U'}}(X)) = \Eps_{P'[1]}^{\Gamma_{H_P(U')}}\left({\widetilde{\C}_{U'}(B_{U'},X)}\right) = {\widetilde{\C}_{U'}(B_{U'},X)},$$ while 
    $${\widetilde{H}}_{B_{I_U}^{\widetilde{\C}_{U'}}}(\pi_{\pi_{U'}(I_U)}(X)) = {\widetilde{H}}_{B_{I_U}^{\widetilde{\C}_{U'}}}(X) = {\widetilde{\C}_{U'}(B_{U'},X)},$$
    and the diagram commutes in this case. 

    \vspace{3pt}
    
    \textbf{Case 2: } $X = I_X\in\ind\left(\inj\widetilde{\C}_{U'}\cap\presilt\left({\widetilde{\C}_{U'}}\right)_{I_U}\right)$.

    By definition, $P'[1] = F_{H_P(U')}(\Eps_{H_P(U')}(H_P(I_U)))[1]$. The commutativity of the square in Lemma \ref{comm_of_(a)} together with that of \eqref{big-diagram}\texttt{(c)} imply that 
    $$P'[1] = H_{B_{U'}}(\pi_U(I_U)) =   H_{B_{U'}}(I_U) = \widetilde{\C}_{U'}\left(B_{U'},\Omega_{\widetilde{\C}_{U'}}I_U\right)[1].$$
    Then, 
    \begin{align*}
        \Eps_{P'[1]}^{\Gamma_{H_P(U')}}(H_{B_{U'}}(I_X)) &= \Eps_{ H_{B_{U'}}(I_U)}^{\Gamma_{H_P(U')}}(H_{B_{U'}}(I_X))\\
        &= \Eps_{ {\widetilde{\C}_{U'}(B_{U'},\Omega_{\widetilde{\C}_{U'}}I_U)}[1]}^{\Gamma_{H_P(U')}}\left(\widetilde{\C}_{U'}(B_{U'},\Omega_{\widetilde{\C}_{U'}}I_X)[1]\right)\\
        &= f^{\Gamma_{H_P(U')}}_{\widetilde{\C}_{U'}(B_{U'},\Omega_{\widetilde{\C}_{U'}}I_U)}\left(\widetilde{\C}_{U'}(B_{U'},\Omega_{\widetilde{\C}_{U'}}I_X)\right)[1]
    \end{align*}
    where the last equality follows by Proposition \ref{epsilonP[1]}. On the other hand, 
    $${\widetilde{H}}_{B_{I_U}^{\widetilde{\C}_{U'}}}(\pi_{\pi_{U'}(I_U)}(I_X)) = {\widetilde{H}}_{B_{I_U}^{\widetilde{\C}_{U'}}}(I_X) = {\widetilde{\C}_{U'}}\left(B_{U'},\Omega_{\widetilde{(\widetilde{\C}_{U'})}_{I_U}} I_X\right)[1].$$

    We claim that 
    \begin{equation}\label{claim}
        f^{\Gamma_{H_P(U')}}_{\widetilde{\C}_{U'}\left(B_{U'},\Omega_{\widetilde{\C}_{U'}}I_U\right)}\left(\widetilde{\C}_{U'}(B_{U'},\Omega_{\widetilde{\C}_{U'}}I_X)\right) = \widetilde{\C}_{U'}\left(B_{U'},\Omega_{\widetilde{\C}_{U'}}I_X\right).
    \end{equation}

    This is the case if and only if $\widetilde{\C}_{U'}\left(B_{U'},\Omega_{\widetilde{\C}_{U'}}I_X\right)\in \widetilde{\C}_{U'}\left(B_{U'},\Omega_{\widetilde{\C}_{U'}}I_U\right)^{\perp_0}$, that is $$\Hom_{\Gamma_{H_P(U')}}\left(\widetilde{\C}_{U'}\left(B_{U'},\Omega_{\widetilde{\C}_{U'}}I_U\right), \widetilde{\C}_{U'}\left(B_{U'},\Omega_{\widetilde{\C}_{U'}}I_X\right)\right) = 0.$$ Since $\widetilde{\C}_{U'}(B_{U'},-)$ is an equivalence, this is equivalent to $\widetilde{\C}_{U'}(\Omega_{\widetilde{\C}_{U'}}I_U, \Omega_{\widetilde{\C}_{U'}}I_X) = 0$, which, since $\Omega_{\widetilde{\C}_{U'}}(-)$ is also an equivalence, reduces to $\widetilde{\C}_{U'}(I_U,I_X) = 0$. There is an $\EE_{\widetilde{\C}_{U'}}$-triangle $\Omega_{\widetilde{\C}_{U'}}I_X\infl 0\defl I_X$. Applying ${\widetilde{\C}_{U'}}(I_U,-)$ we get an exact sequence 
    $${\widetilde{\C}_{U'}}(I_U, \Omega_{\widetilde{\C}_{U'}}I_X)\to 0 \to {\widetilde{\C}_{U'}}(I_U,I_X)\to \EE_{{\widetilde{\C}_{U'}}}(I_U, \Omega_{\widetilde{\C}_{U'}}I_X )\to 0.$$ Then ${\widetilde{\C}_{U'}}(I_U,I_X) = 0$ if and only if $\EE_{{\widetilde{\C}_{U'}}}(I_U, \Omega_{\widetilde{\C}_{U'}}I_X) = 0$. 
    
    Observe that, since $\Omega_{\widetilde{\C}_U}I_X\in \widetilde{\C}_U$ and $U = U'\oplus I_U$, we have $\Omega_{\widetilde{\C}_U}I_X\in\widetilde{\C}_{U'}$. Hence, there is a commutative diagram of $\EE_{{\widetilde{\C}_{U'}}}$-triangles 

    \begin{equation}\label{OmegaIso}
    \begin{tikzcd}[ampersand replacement=\&,cramped]
	{\Omega_{\widetilde{\C}_{U'}}I_X} \& 0 \& {I_X} \\
	{\Omega_{\widetilde{\C}_{U}}I_X}\& 0 \& {I_X}
	\arrow[from=1-1, to=1-2]
	\arrow[from=1-1, to=2-1]
	\arrow[from=1-2, to=1-3]
	\arrow[from=1-2, to=2-2]
	\arrow[equals, from=1-3, to=2-3]
	\arrow[from=2-1, to=2-2]
	\arrow[from=2-2, to=2-3]
    \end{tikzcd}
    \end{equation}

where the left vertical morphism exists by (ET3)$^\text{op}$. Since the middle and right vertical arrows are isomorphisms, so is the third vertical one by \cite[Cor. 3.6]{Nakaoka-Palu_extriangulated_categories}. Hence, 
\begin{align*}
    \EE_{\widetilde{\C}_{U'}}\left(I_U, \Omega_{\widetilde{\C}_{U'}}I_X\right) &\cong \EE_{\widetilde{\C}_{U'}}\left(I_U, \Omega_{\widetilde{\C}_{U}}I_X\right) &&\text{by Diagram \eqref{OmegaIso}}\\
    &\cong \EE\left(I_U, \Omega_{\widetilde{\C}_{U}}I_X\right) &&\text{by \cite[Propo. 3.30]{Nakaoka-Palu_extriangulated_categories}}\\
    &=0 &&\text{since $\Omega_{\widetilde{\C}_{U}}I_X\in\add B_U$ and $B_U$ is silting}.
\end{align*}
Therefore, $\widetilde{\C}_{U'}(I_U, I_X) = 0$ and Claim \eqref{claim} follows.   Thus, we obtain 
\begin{align*}
    {\widetilde{H}}_{B_{I_U}^{\widetilde{\C}_{U'}}}(\pi_{\pi_{U'}(I_U)}(I_X)) &\cong \widetilde{\C}_{U'}\left(B_{U'}, \Omega_{\widetilde{(\widetilde{\C}_{U'})}_{I_U}} I_X\right)[1] &&\text{by Lemma \ref{H_tilda}}\\
    &\cong \widetilde{\C}_{U'}\left(B_{U'}, \Omega_{\widetilde{\C}_{U}} I_X\right)[1] &&\text{by the proof of Lemma \ref{commutativity_of_piU+V}}\\
    &\cong \widetilde{\C}_{U'}\left(B_{U'}, \Omega_{\widetilde{\C}_{U'}} I_X\right)[1] &&\text{by Diagram \eqref{OmegaIso}}\\
    &\cong f^{\Gamma_{H_P(U')}}_{\widetilde{\C}_{U'}(B_{U'},\Omega_{\widetilde{\C}_{U'}}I_U)}\left(\widetilde{\C}_{U'}(B_{U'},\Omega_{\widetilde{\C}_{U'}}I_X)\right)[1] &&\text{by Claim \eqref{claim}}\\
    &\cong \Eps_{P'[1]}^{\Gamma_{H_P(U')}}\left(H_{B_{U'}}(I_X)\right) &&\text{by Proposition \ref{epsilonP[1]}}. 
\end{align*}

In other words, the diagram commutes in this case. This finishes the proof. 
\end{proof}

We now have all the ingredients to prove the main result of this section. 

\begin{proof}[Proof of Theorem \ref{compatibility_thm}]
    The horizontal maps are bijections by Proposition \ref{presilting_supp.tau-rigid} and Lemma \ref{H_tilda}, while the vertical maps are bijections by Lemma \ref{proj_functior_bij} and Theorem \ref{epsilonU}. Proving the commutativity of diagram \eqref{compatibility_square} is equal to showing that the outer square in diagram \eqref{big-diagram} commutes. This amounts to checking that each of the subdiagrams inside \eqref{big-diagram} is commutative. Note that each arrow in \eqref{big-diagram} is a bijection. Commutativity of \eqref{big-diagram}\texttt{(a)} and \eqref{big-diagram}\texttt{(d)} is proved in Lemma \ref{comm_of_(a)} and Lemma \ref{comm_of_(d)}, respectively. Commutavity of \eqref{big-diagram}\texttt{(b)} is given in Lemma \ref{commutativity_of_piU+V}.  The diagram \eqref{big-diagram}\texttt{(c)} commutes by Lemma \ref{H_tilda}, while the commutativity of \eqref{big-diagram}\texttt{(e)} and \texttt{(f)} follow from Lemma \ref{H_tilda} and Lemma  \ref{Gamma_P'[1]^Gamma_M=Gamma_U}. Finally, \eqref{big-diagram}\texttt{(g)} commutes by definition (see Section \ref{Sec_signed_tau_exc_seq}). The claim follows.
\end{proof}

\section{Ordered presilting objects and signed $\tau$-exceptional sequences}\label{Sec:presilting_seq}

In this section, we define \textit{(signed) presilting sequences} and provide an explicit bijection between (signed) presilting sequences in $\C$ and (signed) $\tau$-exceptional sequences in $\CLambda(\Lambda)$. Since the definition of signed presilting sequence turns out to be equivalent to that of an ordered presilting object in $\C$, the above mentioned bijection gives an alternative approach to describe the Buan--Marsh bijection between ordered support $\tau$-rigid objects in $\CLambda(\Lambda)$ and signed $\tau$-exceptional sequences in $\CLambda(\Lambda)$ explained in Theorem \ref{BM_bijection}.  

Motivated by the recursive definition of a signed $\tau$-exceptional sequence (see Definition \ref{Def-signed-tau-exc-seq}) and the compatibility of the Pan--Zhu silting reduction with the Buan--Marsh reduction of support $\tau$-rigid objects (see Theorem \ref{compatibility_thm}), it is natural to recursively define a \textit{(signed) presilting sequence} in the following way. 

\begin{definition}\label{def_signed_presilt_seq}
    Let $t\in\{1,\cdots,\abs{P}\}$. A $t$-tuple of indecomposable objects $(U_1,\cdots,U_t)$ in $\C$ is called \emph{signed presilting sequence} if 
    \begin{itemize}
        \item[(a)] The object $U_t$ is non-injective presilting in $\C$ or $U_t = \Sigma P_{U_t}\in \inj\C$ for an indecomposable projective object $P_{U_t}$ of $\C$; 
        \item[(b)] The sequence 
        $(U_1,\cdots,U_{t-1})$ is a signed presilting sequence in $\widetilde{\C}_{U_t}$.
    \end{itemize}
    If $t = \abs{P}$, $(U_1,\cdots,U_t)$ is called a \emph{signed silting sequence}. A signed presilting sequence in which every object $U_{i}$ is non-injective in $\C_{(U_{i+1},\cdots, U_t)}$ where  $$ \C_{(U_{i+1},\cdots, U_t)}  = \C^{\C_{(U_{i+2},\cdots,U_t)}}_{U_{i+1}},$$ is called a \emph{presilting sequence}. 
\end{definition}

\begin{remark}\label{signed_presilting=ordered_presilting}
    Let $U\in\presilt\C$. By the definition of $\C_U$, we see that if an object $V$ is presilting in $\widetilde{\C}_U$, then it is presilting in $\C$. It follows that a sequence $(U_1,\cdots, U_t)$ of indecomposable objects in $\C$ is a signed presilting sequence if and only if $U_1\oplus\cdots\oplus U_t$ is an ordered presilting object. 
\end{remark}

Denote the iterated $\tau$-perpendicular subcategory by
$$J(M_i,\cdots,M_t) := J_{(M_{i+1},\cdots, M_t)}(M_i).$$

\begin{lemma}\label{J(H_P)=J(HTilde)}
    Let $(U_1,\cdots,U_t)$ be a signed presilting sequence. Then,
    $$ J\left(H_P\left(\bigoplus_{j = i}^t U_j\right)\right) = J\left(\widetilde{H}_{B_{\bigoplus_{j=i+1}^t U_j}}(U_{i}),\cdots,\widetilde{H}_{B_{U_t}}(U_{t-1}), H_P(U_t)\right),$$
     for all $1\leq i <t$.
\end{lemma}

\begin{proof}
   We proceed by induction on the length of the sequence. If $t= 1$, there is nothing to prove. So, let $(U_1,\cdots,U_i,U_{i+1},\cdots, U_t)$ be a signed presilting sequence and suppose that the claim holds for $i+1$, that 
   is, 
   \begin{equation}\label{induction_hp}
       J\left(H_P\left(\bigoplus_{j = i+1}^t U_j\right)\right) = J\left(\widetilde{H}_{B_{\bigoplus_{j=i+2}^t U_j}}(U_{i+1}),\cdots,\widetilde{H}_{B_{U_t}}(U_{t-1}), H_P(U_t)\right).
   \end{equation}
   We want to show that the claim holds for $i$, that is,
   $$ J\left(H_P\left(\bigoplus_{j = i}^t U_j\right)\right) = J\left(\widetilde{H}_{B_{\bigoplus_{j=i+1}^t U_j}}(U_{i}),\widetilde{H}_{B_{\bigoplus_{j=i+2}^t U_j}}(U_{i+1}),\cdots,\widetilde{H}_{B_{U_t}}(U_{t-1}), H_P(U_t)\right).$$
   Then, 
   \begin{align*}
       J\left(H_P\left(\bigoplus_{j = i}^t U_j\right)\right) &= J\left(H_P\left(\bigoplus_{j = i+1}^t U_j \right) \oplus H_P(U_i) \right) \\[3pt]
       &= J_{J(H_P(\bigoplus_{j = i+1}^t U_j)}\left(\Eps_{H_P\left(\bigoplus_{j = i+1}^t U_j\right)}\left(H_P(U_i)\right)\right) &&\text{by \cite[Theorem 4.3]{tau-perpendicular_wide_subategories}}\\[3pt]
       &= J_{J(H_P(\bigoplus_{j = i+1}^t U_j)}\left(\widetilde{H}_{B_{\bigoplus_{j = i+1}^t U_j}}(H_P(U_i))\right) &&\text{by Theorem \ref{compatibility_thm}}\\[3pt]
       &=J_{J\left(\widetilde{H}_{B_{\bigoplus_{j=i+2}^t U_j}}(U_{i+1}),\cdots,\widetilde{H}_{B_{U_t}}(U_{t-1}), H_P(U_t)\right)}\left(\widetilde{H}_{B_{\bigoplus_{j = i+1}^t U_j}}(H_P(U_i))\right) &&\text{by the induction hypothesis \eqref{induction_hp}}\\[3pt]
       &= J\left(\widetilde{H}_{B_{\bigoplus_{j=i+1}^t U_j}}(U_i), \widetilde{H}_{B_{\bigoplus_{j=i+2}^t U_j}}(U_{i+1}),\cdots,\widetilde{H}_{B_{U_t}}(U_{t-1}), H_P(U_t)\right).
   \end{align*}
The claim follows. 
\end{proof}

The next result gives an explicit bijection between signed presilting sequences of length $t$ in $\C$ and signed $\tau$-exceptional sequences in of length $t$ in $\CLambda(\Lambda)$. 

\begin{proposition}\label{signed_presilt_seq-signed_tau-exc_seq_bijection}
    Let $t\in\{1,\cdots,n\}$. There is a bijection 
    \[
        \xi_t :
        \vcenter{\hbox{$
        \left\{
        \begin{array}{c}
        \text{signed presilting sequences} \\
        \text{of length } t \text{ in } \C
        \end{array}
        \right\}
        $}}
        \;\longrightarrow\;
        \vcenter{\hbox{$
        \left\{
        \begin{array}{c}
        \text{signed $\tau$-exceptional sequences} \\
        \text{of length } t \text{ in } \mathcal{C}(\Lambda)
        \end{array}
        \right\},
        $}}
        \]
        given by $$ (U_1,\cdots, U_t)\mapsto \left(\widetilde{H}_{B_{\bigoplus_{i=2}^t U_i}}(U_1), \widetilde{H}_{B_{\bigoplus_{i=3}^t U_i}}(U_2),\cdots, \widetilde{H}_{B_{U_t}}(U_{t-1}), H_P(U_t)\right).$$
\end{proposition}

\begin{proof}
   Let $t\in\{1,\cdots,n\}$ and let $(U_1,\cdots, U_t)$ be a signed presilting sequence in $\C$. Then, $H_P(U_t)$ is support $\tau$-rigid in $\modd\Lambda$ by Proposition \ref{presilting_supp.tau-rigid}. Furthermore, for each $1\leq i < t$,  Lemma \ref{H_tilda} implies that $\widetilde{H}_{B_{U_i\oplus\cdots\oplus U_t}}(U_{i-1})$ is support $\tau$-rigid in
   $$J(H_P(U_i\oplus\cdots\oplus U_t)) = J\left( \widetilde{H}_{B_{\bigoplus_{j=i+1}^t}U_j}(U_i), \cdots, \widetilde{H}_{B_{U_t}}(U_{t-1}), H_p(U_t)\right),$$
   where the last equality follows from Lemma \ref{J(H_P)=J(HTilde)}. Hence, $\xi_t$ is well-defined. 
      
   We start by proving injectivity. Let $\mathsf{V} = (V_1,\cdots,V_t)$ and $\mathsf{U} = (U_1,\cdots, U_t)$ be signed presilting sequences such that $\xi_t(\mathsf{U}) = \xi_t(\mathsf{V})$. Then $H_P(U_t) = H_P(V_t)$ implies that $U_t = V_t$ since $H_P$ is bijective by Proposition \ref{presilting_supp.tau-rigid}. This gives $\widetilde{H}_{B_{U_t}} = \widetilde{H}_{B_{V_t}}$ and thus, since $\widetilde{H}_{B_X}$ is bijective for a presilting object $X$ in $\C$ by Lemma \ref{H_tilda}, $\widetilde{H}_{B_{U_t}}(U_{t-1}) = \widetilde{H}_{B_{V_t}}(V_{t-1})$ implies $U_{t-1} = V_{t-1}$. Iterating this argument, we get $\mathsf{U} = \mathsf{V}$, giving injectivity. 

   To prove surjectivity, let $\X = (X_1,\cdots, X_t)$ be a signed $\tau$-exceptional sequence in $\CLambda(\Lambda)$. By definition of signed $\tau$-exceptional sequence, $X_t$ is support $\tau$-rigid in $\CLambda(\Lambda)$, and thus Proposition \ref{presilting_supp.tau-rigid} implies $X_t = H_P(U_t)$ for an indecomposable presilting object $U_t$ in $\C$. Next, $X_{t-1}$ is support $\tau$-rigid in $\CLambda(J(H_P(U_t)))$. Then, $X_{t-1} =  \widetilde{H}_{B_{U_t}}(U_{t-1})$ for a presilting object $U_{t-1}$ in $\widetilde{\C}_{U_t}$ by Lemma \ref{H_tilda}. By assumption, $X_{t-2}$ is support $\tau$-rigid in $\CLambda\left(J_{H_P{(U_t)}}(\widetilde{H}_{B_{U_t}}(U_{t-1}))\right)$. 
   By Lemma \ref{J(H_P)=J(HTilde)}, $$J_{H_P{(U_t)}}(\widetilde{H}_{B_{U_t}}(U_{t-1})) = J(H_P(U_{t-1})\oplus H_P(U_t)) = J(H_P(U_{t-1}\oplus U_t)),$$ and thus, Lemma \ref{H_tilda} gives $X_{t-2} = \widetilde{H}_{B_{U_{t-1}\oplus U_t}}(U_{t-2})$ for an indecomposable presilting object $U_{t-2}$ in $\widetilde{\C}_{U_{t-1}\oplus U_t} \cong \widetilde{(\widetilde{\C}_{U_t})}_{U_{t-1}}$. Iterating this argument on the length of the sequence, we can construct a presilting sequence $\mathsf{U} = (U_1,\cdots, U_t)$ in $\C$ such that $\xi_t (\mathsf{U}) = \X$. This proves surjectivity and concludes the proof. 
\end{proof}

\begin{corollary}\label{presilt_seq-tau-exc_seq-bijection}
    The bijection $\xi_t$ restricts to a bijection 
    \[
        \xi_t :
        \vcenter{\hbox{$
        \left\{
        \begin{array}{c}
        \text{presilting sequences} \\
        \text{of length } t \text{ in } \C
        \end{array}
        \right\}
        $}}
        \;\longrightarrow\;
        \vcenter{\hbox{$
        \left\{
        \begin{array}{c}
        \text{$\tau$-exceptional sequences} \\
        \text{of length } t \text{ in } \mathcal{C}(\Lambda)
        \end{array}
        \right\},
        $}}
        \]
\end{corollary}

\begin{proof}
    The claim follows directly from Lemma \ref{H_tilda} and Proposition \ref{signed_presilt_seq-signed_tau-exc_seq_bijection}.
\end{proof}

Keeping the notation from Section \ref{Sec_signed_tau_exc_seq}, the next result reformulates the Buan--Marsh bijection $\psi_t$ between ordered support $\tau$-rigid objects in $\CLambda(\Lambda)$ and signed $\tau$-exceptional sequences in $\CLambda(\Lambda)$ (see Theorem \ref{BM_bijection}) in terms of the bijection $\xi_t$ from Proposition \ref{signed_presilt_seq-signed_tau-exc_seq_bijection}.

\begin{theorem}\label{xi-H_P-varphi}
    Let $t\in\{1,\cdots,n\}$. Then, there is a commutative diagram of bijections given by

\[\begin{tikzcd}[ampersand replacement=\&,cramped]
	\begin{array}{c} \vcenter{\hbox{$         \left\{         \begin{array}{c}         \text{ordered presilting objects} \\   \text{of length } t \text{ in } \C  \end{array}         \right\}         $}} \end{array} \& \begin{array}{c}  \vcenter{\hbox{$         \left\{         \begin{array}{c}         \text{ordered support $\tau$-rigid objects} \\         \text{of length } t \text{ in } \mathcal{C}(\Lambda)         \end{array}         \right\}        $}} \end{array} \\
	\begin{array}{c} \vcenter{\hbox{$         \left\{         \begin{array}{c}         \text{signed presilting sequences} \\         \text{of length } t \text{ in } \C         \end{array}         \right\}         $}} \end{array} \& \begin{array}{c}  \vcenter{\hbox{$         \left\{         \begin{array}{c}         \text{signed $\tau$-exceptional sequences} \\         \text{of length } t \text{ in } \mathcal{C}(\Lambda)         \end{array}         \right\}.        $}} \end{array}
	\arrow["{H_P}", from=1-1, to=1-2]
	\arrow[equals, from=1-1, to=2-1]
	\arrow["{\psi_t}", from=1-2, to=2-2]
	\arrow["{\xi_t}", from=2-1, to=2-2]
\end{tikzcd}\]
\end{theorem}

\begin{proof}
    All arrows involved in the diagram are bijections by Theorem \ref{BM_bijection}, Propositions \ref{H_tilda} and \ref{signed_presilt_seq-signed_tau-exc_seq_bijection}, and Remark \ref{signed_presilting=ordered_presilting}. Let $(T_1,\cdots, T_t)$ be an ordered presilting object (equivalently a signed presilting sequence) in $\C$. Then, $$\psi_t(H_P(T_1,\cdots, T_t)) = \psi_t(H_P(T_1),\cdots, H_P(T_t)) = (\U_1,\cdots, \U_t),$$ where $(\U_1,\cdots, \U_t)$ is as in Theorem \ref{BM_bijection}.  
    Then, 
    \begin{description}
        \item[(t)] $\W_t = \modd\Lambda$ and $\U_t = H_P(T_t)$.
        \item[(t-1)] $\W_{t-1} = J(\U_t) = J(H_P(T_t))$ and 
    $$\U_{t-1} = \Eps_{\U_t}^{\W_t}(\U_{t-1}) = \Eps_{H_P(T_t)}(H_P({T_{t-1})}) = \widetilde{H}_{B_{T_t}}(T_{t-1}),$$
    by Theorem \ref{compatibility_thm}.
    \item[(t-2)]
       \begin{subequations}
    \begin{align*}
        \W_{t-2} &= J_{\W_{t-1}}(\U_{t-1})\\
        &= J_{J(H_P(T_t))}\left(\widetilde{H}_{B_t}(T_{t-1})\right)\\[3pt]
        &= J(H_P(T_{t-1}\oplus T_t)), &&\text{by Lemma \ref{J(H_P)=J(HTilde)}}
    \end{align*}
    and
    \begin{align*}
        \U_{t-2} &= \Eps_{\U_{t-1}}^{\W_{t-1}}\Eps_{\U_t}^{\W_t}(H_P(T_{t-2}))\\[3pt]
                &= \Eps_{\Eps_{H_P(T_t)}(H_P(T_{t-1}))}^{J(H_P(T_t))}\left(\Eps_{H_P(T_t)}\left(H_P(T_{t-2})\right)\right)\\[3pt]
                &= \Eps_{{H_P(T_t)\oplus H_P(T_{t-1})}}\left(H_P(T_{t-2})\right) 
                    &&\text{by \cite[Theorem 6.12]{tau-perpendicular_wide_subategories}}\\[3pt]
                &= \Eps_{H_P(T_{t-1}\oplus T_t)}\left(H_P(T_{t-2})\right)\\[3pt]
                &= \widetilde{H}_{B_{T_{t-1}\oplus T_t}}\left(T_{t-2}\right) 
                    &&\text{by Theorem \ref{compatibility_thm}.}
    \end{align*}
    \end{subequations}
    \item[(t-3)]  
           \begin{subequations}
    \begin{align*}
        \W_{t-3} &= J_{\W_{t-2}}(\U_{t-2})\\
        &= J_{J(\widetilde{H}_{B_{T_t}}(T_{t-1}), H_P(T_t))}(\widetilde{H}_{B_{T_{t-1}\oplus T_t}}(T_{t-2})) &&\text{by Case (t-2)}\\[3pt]
        &= J(\widetilde{H}_{B_{T_{t-1}\oplus T_t}}(T_{t-2}), \widetilde{H}_{B_{T_t}}(T_{t-1}), H_P(T_t))\\[3pt]
        &= J(H_P(T_{t-2}\oplus T_{t-1}\oplus T_t)) &&\text{by Lemma \ref{J(H_P)=J(HTilde)}}, 
    \end{align*}
    and
    \begin{align*}
        \U_{t-3} &= \Eps_{\U_{t-2}}^{\W_{t-2}}\Eps_{\U_{t-1}}^{\W_{t-1}}\Eps_{\U_t}^{\W_t}(H_P(T_{t-3}))\\[3pt]
                &= \Eps_{\Eps_{H_P(T_{t-1}\oplus T_t)}(H_P(T_{t-2}))}^{J(H_P(T_{t-1}\oplus T_t))}\left(\Eps_{H_P(T_{t-1}\oplus T_t)}\left(H_P(T_{t-3})\right)\right) &&\text{by Case (t-2)}\\[3pt]
                &= \Eps_{{H_P(\bigoplus_{i=t-2}^t T_i)}}\left(H_P(T_{t-3})\right) 
                    &&\text{by \cite[Theorem 6.12]{tau-perpendicular_wide_subategories}}\\[3pt]
                &= \widetilde{H}_{B_{\bigoplus_{i=t-2}^t T_i}}\left(H_P(T_{t-3})\right) 
                   &&\text{by Theorem \ref{compatibility_thm}.}
    \end{align*}
    \end{subequations}
    \vspace{-0.5em}
    \[
    \vdots
    \]
    \vspace{-0.5em}
    \item[(i)] $\W_i = J(H_P(T_{i+1}\oplus\cdots\oplus T_t))$ and $\U_{i} = \widetilde{H}_{B_{T_{i+1}}\oplus\cdots\oplus B_{T_t}}(T_i).$\\
    \vspace{-0.5em}
    \[
    \vdots
    \]
    \vspace{-0.5em}
    \item[(1)] $\W_1 = J(H_P(\bigoplus_{i=2}^t(T_i)))$ and $\U_1 = \widetilde{H}_{\bigoplus_{i=2}^t B_{T_i}}(T_1).$
    \end{description}
    
        We conclude that $$\psi_t(H_P(T_1,\cdots,T_t)) = \xi_t(T_1,\cdots, T_t).$$ This finishes the proof. 
\end{proof}

\section{The $\tau$-cluster morphism category of a 0-Auslander extriangulated category}\label{sec:tau-cluster_morph_cat}

Let $\C$ be a reduced $0$-Auslander extriangulated category with enough projectives, and let $P$ be a projective generator. In this section, we define a category, called the \textit{$\tau$-cluster morphism category} of $\C$ and denoted by $\Mcluster(\C)$, whose objects are extension-closed subcategories of $\C$ of the form $\C_U$, and whose morphisms can be described in terms of signed presilting sequences (or, equivalently, ordered presilting objects). 

\begin{definition}\label{tau-cluster-morph-cat-of-C}
    The \emph{$\tau$-cluster morphism category} of $\C$, denoted by $\Mcluster(\C)$, consists of the following data. 
    \begin{itemize}
        \item[(a)] The objects of $\Mcluster(\C)$ are subcategories of $\C$ of the form ${\C}_U$, for $U\in \presilt\C$; 
        \item[(b)] For a an object $\B$ in $\Mcluster(\C)$ and $U\in\presilt\B$, define a formal symbol $h^\B_U$; 
        \item[(c)] Given two objects $\B_1 = \C_U$ and $\B_2$ in $\Mcluster(\C)$, define 
        $$\Hom_{\Mcluster(\C)}(\B_1, \B_2) = {\left\{  h_V^{\B_1} \;\middle|\; \begin{tabular}{@{}l@{}} $V\in\presilt\C_U$, $\add V\cap\add U = 0$ \\ \text{and} $\B_2 = {(\B_1)}_V$ \end{tabular} \right\}}.$$
        In particular, 
        \begin{itemize}
            \item[(i)] if $\B_1 \nsupseteq \B_2$, then $\Hom_{\Mcluster(\C)}(\B_1, \B_2) = \emptyset$, 
            \item[(ii)] $\Hom_{\Mcluster(\C)}(\B_1,\B_1) = \left\{h_0^{\B_1}\right\}$; 
        \end{itemize}
        \item[(d)] Given $h_U^{\B_1}: \B_1\to \B_2$ and $h_V^{\B_2}: \B_2\to \B_3$, define the composition 
        $$h_V^{\B_2}\circ h_U^{\B_2}:= h_{U\oplus V}^{\B_1}.$$
        \end{itemize}
\end{definition}

\begin{remark}
    Let $\C_U$ be an object in $\Mcluster(\C)$ and let $V$ be a presilting object in $\C_U$. Recall that $$\C_U = \{X\in \C \mid \EE(U,X) = 0 = \EE(X, U)\}.$$ Then, 
    \begin{align*}
        \left(\C_U\right)_V &= \{ Y\in\C_U\mid \EE(V,Y) = 0 = \EE(Y,V)\}\\
        &= \left\{Y\in \C\mid \EE(V\oplus U, Y) = 0 = \EE(Y, U\oplus V)\right\}\\
        &= \C_{U\oplus V}.
    \end{align*}
\end{remark}

\begin{lemma}
    $\Mcluster(\C)$ is a well-defined category. 
\end{lemma}

\begin{proof}
    We start by proving that the composition law is well-defined. So let $h^{\C_U}_{V}: \C_U\to \left(\C_U\right)_V = \C_{U\oplus V}$ and $h^{\C_{U\oplus V}}_W: \C_{U\oplus V}\to \left(\C_{U\oplus V}\right)_W = \C_{U\oplus V\oplus W}$ be composable morphisms in $\Mcluster(\C)$. Then $\add U\cap\add V = 0$ and $\add (U\oplus V)\cap\add W = 0$, and therefore $\add U\cap\add(V\oplus W) = 0$. Moreover, $V\oplus W$ is presilting in $\C_U$. Hence, the composite $h^{\C_{U\oplus V}}_W\circ h^{\C_U}_V = h^{\C_U}_{V\oplus W}$ is well-defined. 

    For $\B$ in $\Mcluster(\C)$, it is straightforward to check that the morphism $h_0^\B$ satisfies the axioms required for an identity map. We prove that composition is associative. So, let $h_U^{\B}: \B\to \B_U$, $h_V^{\B_U}: \B_U\to (\B_U)_V = \B_{U\oplus V}$, and $h_W^{\B_{U\oplus V}}: \B_{U\oplus V}\to (\B_{U\oplus V})_W = \B_{U\oplus V\oplus W}$ be composable morphisms in $\Mcluster(\C)$. Then, 
    \begin{equation*}
        h_W^{\B_{U\oplus V}}\circ (h_V^{\B_U}\circ h^\B_U) = h^{\B_{U\oplus V}}_W\circ h^\B_{U\oplus V} = h^\B_{U\oplus V\oplus W},
    \end{equation*}
    and 
    \begin{equation*}
        (h^{\B_{U\oplus V}}_W\circ h^{\B_U}_V)\circ h^\B_U = h^{\B_U}_{V\oplus W}\circ h^\B_U = h^\B_{U\oplus V\oplus W}.
    \end{equation*}
    The claim follows. 
\end{proof}

A morphism $h$ in $\Mcluster(\C)$ is called \textit{irreducible} if whenever $h=h_1\circ h_2$, either $h_1 = \id$ or $h_2=\id$. Our first goal is to prove that morphisms in $\Mcluster(\C)$ can be described in terms of signed presilting sequences. In order to make this statement precise, we need some preparation. We begin with the following observation.

\begin{lemma}\label{irreducible_mor_in_M(C)}
    Let $h = h_V^\B$ be a morphism in $\Mcluster(\C)$. Then, $h$ is irreducible if and only if $V$ is indecomposable. 
\end{lemma}

\begin{proof}
    Assume $V$ is indecomposable and write $h = h_V^\B = h_{V_2}^{\B_2}\circ h_{V_1}^{\B_1}$, where $\B_1 = \B$ and $\B_2 = \B_{V_1}$. Then, $h_V^{\B} = h^{\B}_{V_1\oplus V_2}$, and thus $V = V_1\oplus V_2$. Since $V$ is indecomposable, $V_1= 0$ or $V_2 = 0$ which imply $h_1 = \id$ or $h_2 = 0$, and therefore $h$ is irreducible. 

    Conversely, suppose $V=V_1\oplus V_2$. Then, 
    $$ h = h_{V_1\oplus V_2}^{\B} = h^{\B_2}_{V_2}\circ h_{V_1}^{\B_1},$$ where $\B = \B_1$. Since $V_1$ and ${V_2}$ are non-zero, $h^{\B_1}_{V_1}$ and $ h_{V_2}^{\B_2}$ are not identity maps, and thus $h$ is not irreducible. 
\end{proof}

\begin{lemma}\label{signed_presilt_seq&composition_of_maps_in_M(C)}
    Let $\B = \C_U$ be an object in $\Mcluster(\C)$. The following statements are equivalent. 
    \begin{itemize}
        \item[(i)] $(U_1,\cdots, U_t)$ is a signed presilting sequence in $\widetilde{\B}$; 
        \item[(ii)] There are objects $\B_1,\cdots, \B_t$ in $\Mcluster(\C)$ and irreducible morphisms $h_{U_i}^{\B_i}$, for $i=1,\cdots, t$, such that the composite $h^{\B_1}_{U_1}\cdots h^{\B_t}_{U_t}$ is well-defined.  
    \end{itemize}
\end{lemma}

\begin{proof}
  We prove (i) implies (ii). Let $(U_1,\cdots, U_t)$ be a signed presilting sequence in $\widetilde{\C}_U$. Then, by definition, $U_t$ is indecomposable and presilting in $\widetilde{\C}_U$. Hence, $U_t$ is presilting in $\C_U$ and $\add U_t\cap\add U = 0$, and thus there is a map $h^{\C_U}_{U_t}: \C_U\to \C_{U\oplus U_t}$ in $\Mcluster(\C)$. In particular, $h^{\C_U}_{U_t}$ is irreducible by Lemma \ref{irreducible_mor_in_M(C)}. Now, $U_{t-1}$ is an indecomposable presilting object in $\widetilde{\C}_{U\oplus U_t}$. Hence, $U_{t-1}$ is presilting in $\C_{U\oplus U_t}$ and $\add U_{t-1}\cap\add(U\oplus U_t) = 0$. Thus, we obtain an irreducible morphism $h^{\C_{U\oplus U_t}}_{U_{t-1}}: \C_{U\oplus U_t}\to \C_{U\oplus U_t\oplus U_{t-1}}$, and hence the composite $h^{\C_{U\oplus U_t}}_{U_{t-1}}h^{\C_U}_{U_t}$ is well-defined in $\Mcluster(\C)$. Iterating this argument on the length of the sequence, and setting $\B_t = \C_U$ and $\B_i =(\B_{i+1})_{U_i}$ for $1\leq i \leq t-1$, we can construct irreducible morphisms $h^{\B_i}_{U_i}$ such that the composite $h^{\B_1}_{U_1}\cdots h^{\B_t}_{U_t}$ is well-defined in $\Mcluster(\C)$.

  Conversely, assume that there exist objects $\B_i$ as in (ii) such that the composite $h^{\B_1}_{U_1}\cdots h^{\B_t}_{U_t}$ is well-defined. Set $\B_t = \C_U$. Then, by definition, $U_t$ is presilting in $\C_U$ and $\add U\cap \add U_t = 0$, and hence $U_t$ is presilting in $\widetilde{\C}_U$. Moreover, since $h^{\B_t}_{U_t}$ is irreducible, Lemma \ref{irreducible_mor_in_M(C)} implies that $U_t$ is indecomposable. Next, observe that since the composite $h^{\B_{t-1}}_{U_{t-1}}h^{\B_t}_{U_t}$ is well-defined, it follows that $\B_{t-1} = \C_{U\oplus U_t}$. By definition, $U_{t-1}$ is presilting in $\C_{U\oplus U_t}$ and $\add U_{t-1}\cap \add (U\oplus U_t) = 0$. Hence, $U_{t-1}$ is presilting in $\widetilde{\C}_{U\oplus U_t}\cong \widetilde{\left(\widetilde{\C}_U\right)}_{U_t}$. Moreover, $U_{t-1}$ is indecomposable by Lemma \ref{irreducible_mor_in_M(C)} as $h^{\B_t}_{U_t}$ is irreducible. Iterating this argument, we obtain a signed presilting sequence $(U_1,\cdots, U_t)$ in $\widetilde{\C}_U$. This finishes the proof. 
\end{proof}

\begin{lemma}\label{h_V=h_U1+...+Ut}
    Let $\B$ be an object in $\Mcluster(\C)$, and let $(U_1,\cdots, U_t)$ be a signed presilting sequence in $\widetilde{\B}$. Set $\B_t = \B$ and $\B_i = {\B}_{U_{i+1}\oplus\cdots \oplus U_t}$ for $i= 1,\cdots, t-1$. Then 
    $$ h_{U_1}^{\B_1}\cdots h_{U_t}^{\B_t} = h_{U_1\oplus\cdots \oplus U_t}^{\B_t}.$$
\end{lemma}

\begin{proof}
    This follows from the fact that $\B_i = {\B}_{U_{i+1}\oplus \cdots\oplus U_t}$ and Definition \ref{tau-cluster-morph-cat-of-C}(d).
\end{proof}

We are now ready to state and prove an extriangulated version of \cite[Proposition 11.8]{A_category_of_wide_subcategories}. 

\begin{proposition}
    Let $\B$ be an object in $\Mcluster(\C)$ and let $V\in\presilt\widetilde{\B}$ with $\abs{V} = t$. There is a bijection between the following sets. 
    \begin{itemize}
        \item[(a)] Signed presilting sequences $(U_1,\cdots, U_t)$ with $V = U_1\oplus \cdots \oplus U_t$ in $\widetilde{\B}$; 
        \item[(b)] Factorisations of $h_V^{\B}$ into composition of irreducible morphisms in $\Mcluster(\C)$. 
    \end{itemize}
\end{proposition}

\begin{proof}
    Given a sequence $(U_1,\cdots,U_t)$ as in (a), set $\B_t = \B$ and $\B_i = {(\B_{i+1})}_{U_{i+1}}$. Lemma \ref{signed_presilt_seq&composition_of_maps_in_M(C)} implies that the composite $h_{U_1}^{\B_1}\cdots h_{U_t}^{\B_t}$ is well-defined. Lemma \ref{h_V=h_U1+...+Ut} gives $h_{U_1}^{\B_1}\cdots h_{U_t}^{\B_t} = h^{\B}_{U_1\oplus\cdots\oplus U_t}$ where each $h_{U_i}^{\B_i}$ is irreducible by Lemma \ref{irreducible_mor_in_M(C)}. 

    Conversely, let $h_V^{\B} = h^{\B_1}_{U_1}\cdots h_{U_{t'}}^{\B_t}$ be a factorisation of $h_V^{\B}$ into irreducible morphisms in $\Mcluster(\C)$. Then, $t = t'$ by Lemma \ref{h_V=h_U1+...+Ut}. Each $U_i$ is indecomposable by Lemma \ref{irreducible_mor_in_M(C)}. Since $h^{\B_1}_{U_1}\cdots h_{U_t}^{\B_t}$ is well-defined, Lemma \ref{signed_presilt_seq&composition_of_maps_in_M(C)} gives a signed presilting sequence $(U_1,\cdots, U_t)$ in $\B$ with $V = U_1\oplus \cdots \oplus U_t$ by Lemma \ref{h_V=h_U1+...+Ut}. One easily checks that these constructions are inverse of each other obtaining the desired bijection. 
\end{proof}

\section{Recovering the $\tau$-cluster morphism category of $\Lambda$}\label{sec:tau-cluster-of-Lambda}

The main aim of this section is to reconstruct the $\tau$-cluster morphism category of $\Lambda = \End_\C(P)$ from the $\tau$-cluster morphism category of $\C$. We recall the definition of the $\tau$-cluster morphism category of a finite dimensional algebra from \cite{tau-perpendicular_wide_subategories}. 

\begin{definition}[{\cite[Def. 6.1]{tau-perpendicular_wide_subategories}}]\label{def_tau-cluster_morph_cat}
    Let $\Lambda$ be a finite dimensional algebra. The \emph{$\tau$-cluster morphism category} of $\Lambda$, denoted as $\Mcluster(\Lambda)$, consists of the following data. 
    \begin{enumerate}
        \item[(a)] The objects of $\Mcluster(\Lambda)$ are the $\tau$-perpendicular subcategories of $\modd\Lambda$. 
        \item[(b)] For a $\tau$-perpendicular subcategory $\W\subseteq \modd{\Lambda}$ and $U\in \C(\W)$ support $\tau$-rigid and basic, define a formal symbol $g_U^\W$. 
        \item[(c)]  Given $\W_1, \W_2$ two $\tau$-perpendicular subcategories of $\modd{\Lambda}$, we define 
        $$\Hom_{\Mcluster(\Lambda)}(\W_1,\W_2) = {\left\{ g_U^{\W_1} \;\middle|\; \begin{tabular}{@{}l@{}} U \text{ is a basic support $\tau$-rigid object in } $\C(\W_1)$ \\ \text{and} $\W_2 = J_{\W_1}(U)$ \end{tabular} \right\}}.$$
        In particular: 
        \begin{enumerate}
            \item[(i)] If $\W_1 \not\supseteq \W_2$, then $\Hom_{\Mcluster(\Lambda)}(\W_1,\W_2) = \emptyset$;
            \item[(ii)] $\Hom_{\Mcluster(\Lambda)}(\W_1,\W_1) = \left\{g_0^{\W_1}\right\}$. 
        \end{enumerate}
        \item[(d)]  Given $g_U^{\W_1}:\W_1\to \W_2$ and $g_V^{\W_2}: \W_2\to \W_3$ in $\Mcluster(\Lambda)$, denote $\widetilde{V}:= (\mathcal{E}_U^{\W_1})^{-1}(V)$. We define the composition to be 
        $$ g_V^{\W_2}\circ g_U^{\W_1}:= g_{U\oplus\widetilde{V}}^{\W_1}.$$
    \end{enumerate}
\end{definition}

We construct a functor from $\Mcluster(\C)$ to $\Mcluster(\Lambda)$ as follows. 

\begin{proposition}\label{functor-tau-cluster-morph-cat}
    There is a functor $F: \Mcluster(\C)\to \Mcluster(\Lambda)$ given by
    $$ {\C}_U\mapsto J(H_P(U))$$
    on objects, and 
    $$h_{V}^{{\C}_U} \mapsto g^{J(H_P(U))}_{\widetilde{H}_{B_U}(V)}$$
    on morphisms. 
\end{proposition}

\begin{proof}
    $F$ is well-defined on objects by Proposition  \ref{presilting_supp.tau-rigid}. Let $h_V^{{\C}_U}: {\C}_U \to {\left({\C}_U\right)}_V$ be a morphism in $\Mcluster(\C)$. Notice that
    \begin{align*}
        F\left({\left({\C}_U\right)}_V\right) = F\left({\C}_{U\oplus V}\right) = J(H_P(U\oplus V)).
    \end{align*}
    Moreover,
    $$F\left(h_V^{{\C}_U}\right) = g_{\widetilde{H}_{B_U}(V)}^{J(H_P(U))}: J(H_P(U))\to J_{J(H_P(U))}\left(\widetilde{H}_{B_U}(V)\right),$$
    where 
    $$J_{J(H_P(U))}\left(\widetilde{H}_{B_U}(V)\right) = J(H_P(U\oplus V))$$ 
    by Lemma \ref{J(H_P)=J(HTilde)}. Hence, $F$ is well-defined on morphisms. Clearly, $F\left(h_0^{{\C}_U}\right) = g_0^{J(H_P(U))}$, and thus the identity morphism is preserved. It remains to show that $F$ preserves composition.  So let $h_V^{\C_U}: {\C}_U\to {\C}_{U\oplus V}$ and $h_W^{{\C}_{U\oplus V}}: {\C}_{U\oplus V}\to {\C}_{U\oplus V\oplus W}$ be composable morphisms in $\Mcluster(\C)$. Then, 
    \begin{align*}
        F\left(h_{W}^{{\C}_{U\oplus V}}\right)\circ F\left( h_V^{{\C}_U}\right) &= g^{J(H_P(U\oplus V))}_{\widetilde{H}_{B_{U\oplus V}}(W)}\circ g^{J(H_P(U)))}_{\widetilde{H}_{B_U}(V)} &&\text{by definition of $F$}\\[3pt]
        &= g^{J(H_P(U))}_{\widetilde{H}_{B_U}(V)\oplus \left( \Eps ^{J(H_P(U))}_{\widetilde{H}_{B_U}(V)}\right)^{-1}\left(\widetilde{H}_{B_{U\oplus V}}(W)\right)} &&\text{by Definition \ref{def_tau-cluster_morph_cat}(d)}.
    \end{align*}
    We compute the second direct summand of the subscript of the above expression. 
    \begin{align*}
        \left( \Eps ^{J(H_P(U))}_{\widetilde{H}_{B_U}(V)}\right)^{-1}\left(\widetilde{H}_{B_{U\oplus V}}(W)\right) &= \left( \Eps ^{J(H_P(U))}_{\widetilde{H}_{B_U}(V)}\right)^{-1}\left( \Eps_{H_P(U\oplus V)}(H_P(W))\right) &&\text{by Theorem \ref{compatibility_thm}}\\
        &= \left( \Eps ^{J(H_P(U))}_{\widetilde{H}_{B_U}(V)}\right)^{-1}\left(\Eps_{H_P(U)\oplus H_P(V)}(H_P(W))\right) &&\text{$H_P$ is additive}\\
        &= \left( \Eps ^{J(H_P(U))}_{\widetilde{H}_{B_U}(V)}\right)^{-1}\left( \Eps^{J(H_P(U))}_{\Eps_{H_P(U)}(H_P(V))}\left(\Eps_{H_P(U)}(H_P(W))\right)\right) &&\text{by \cite[Theorem 6.12]{tau-perpendicular_wide_subategories}}\\
        &= \left( \Eps ^{J(H_P(U))}_{\Eps_{H_P(U)}(H_P(V))}\right)^{-1}\left( \Eps^{J(H_P(U))}_{\Eps_{H_P(U)}(H_P(V))}\left(\Eps_{H_P(U)}(H_P(W))\right)\right) &&\text{by Theorem \ref{compatibility_thm}}\\[3pt]
        &= \Eps_{H_P(U)}(H_P(W)).
    \end{align*}
    Hence, we obtain 
    \begin{align*}
        F\left(h_{W}^{{\C}_{U\oplus V}}\right)\circ F\left( h_V^{{\C}_U}\right) &= g^{J(H_P(U))}_{\widetilde{H}_{B_U}(V)\oplus \Eps_{H_P(U)}(H_P(W))}&&\text{by the above computation}\\[3pt]
        &= g^{J(H_P(U))}_{\Eps_{H_P(U)}(H_P(V))\oplus \Eps_{H_P(U)}(H_P(W))} &&\text{by Theorem \ref{compatibility_thm}}\\
        &= g^{J(H_P(U))}_{\Eps_{H_P(U)}(H_P(V\oplus W))} &&\text{$\Eps_{H_P(U)}$ is additive}\\[3pt]
        &= g^{J(H_P(U))}_{\widetilde{H}_{B_U}(V\oplus W)} &&\text{by Theorem \ref{compatibility_thm}}.
    \end{align*}
    On the other hand, 
    \begin{align*}
        F\left(h_{W}^{{\C}_{U\oplus V}}\circ  h_V^{{\C}_U}\right) &= F\left( h^{{\C}_U}_{V\oplus W}\right)&&\text{by Definition \ref{tau-cluster-morph-cat-of-C}(d)}\\[3pt]
        &= g^{J(H_P(U))}_{\widetilde{H}_{B_U}(V\oplus W)} &&\text{by definition of $F$}\\[3pt]
        &= F\left(h_{W}^{{\C}_{U\oplus V}}\right)\circ F\left( h_V^{{\C}_U}\right) &&\text{by the above computation}.
    \end{align*}
    Thus, $F$ preserves composition, and the claim follows. 
\end{proof}

Recall that a functor $G: \mathcal{C}\to\mathcal{D}$ is a \textit{discrete fibration} if for all objects $C\in\mathcal{C}$ and every morphism $g: G(C)\to D$ in $\mathcal{D}$, there exists a unique morphism $f: C\to C'$ in $\mathcal{C}$ such that $G(f) = g$. 

The following is the main result of this section. 

\begin{theorem}\label{F_dense_faithful_discrete_fibration}
    The functor $F:\Mcluster(\C)\to \Mcluster(\Lambda)$ from Proposition \ref{functor-tau-cluster-morph-cat} is dense, faithful, and a discrete fibration. 
\end{theorem}

\begin{proof}
    We start by proving that the functor $F$ is dense. So, let $\W$ be an object in $\Mcluster(\Lambda)$. Then, $\W = J(X)$ for a support $\tau$-rigid object $X\in\CLambda(\Lambda)$. By Proposition \ref{presilting_supp.tau-rigid}, $X$ is isomorphic to $H_P(U)$ for a presilting object $U$ in $\C$, and thus $J(X) = J(H_P(U))$.  Hence, by definition, $F(\C_U) = \W$, proving $F$ is dense. 

    To prove that $F$ is faithful, let $h,h'\in \Hom_{\Mcluster(\C)}\left({\C}_U,{\C}_{U\oplus V}\right)$ be such that $F(h) = g = F(h')$ in $\Hom_{\Mcluster(\Lambda)}(J(H_P(U)), J(H_P(U\oplus V))$. Decompose $h$ and $h'$ into irreducible morphisms 
    $$h = h^{\B_1}_{U_1}\cdots h^{\B_t}_{U_t} \quad \text{ and } \quad h' = h^{\B_1'}_{U'_1}\cdots h^{\B'_t}_{U_t'},$$
    where $\B_t = {\C}_U = \B_t'$. Then $(U_1,\cdots,U_t)$ and $(U_1',\cdots, U_t')$ are signed presilting sequences in $\widetilde{\C}_U$ by Lemma \ref{irreducible_mor_in_M(C)}. Applying Proposition \ref{signed_presilt_seq-signed_tau-exc_seq_bijection} in the reduced 0-Auslander extriangulated category $\widetilde{\C}_U$, we obtain 
    \begin{align*}
        \xi_t(U_1,\cdots, U_t) = \left( \widetilde{H}_{B_{U\oplus \bigoplus_{i=2}^{t-1}U_i}}(U_1),\cdots, \widetilde{H}_{B_{U\oplus U_t}}(U_{t-1}), \widetilde{H}_{B_U}(U_t)\right).
    \end{align*}
    Then, 
    $$F(h) = F\left(h^{\B_1}_{U_1}\right)\cdots F\left(h^{\B_t}_{U_t}\right) = g^{J(H_P(U\oplus \bigoplus_{i=2}^t U_i))}_{\widetilde{H}_{B_{U\oplus \bigoplus_{i=2}^t U_i}}(U_1)}\cdots g^{J(H_P(U))}_{\widetilde{H}_{B_U}(U_t)} = g_{\overline{\varphi}_t(\xi_t(U_1,\cdots, U_t))}^{J(H_P(U))},$$ 
    where the last equality follows from \cite[Proposition 11.8]{A_category_of_wide_subcategories}. Here  $\overline{\varphi}_t(\xi_t(U_1,\cdots, U_t))$ denotes the sum of the entries of the support $\tau$-rigid objects $\varphi_t(\xi_t(U_1,\cdots, U_t))$; see \cite{tauExcSeq_BM}. Similarly, one gets 
    $$F(h') = F\left(h^{\B_1'}_{U_1'}\right)\cdots F\left(h^{\B_t'}_{U_t'}\right) = g^{J(H_P(U\oplus \bigoplus_{i=2}^t U_i'))}_{\widetilde{H}_{B_{U\oplus \bigoplus_{i=2}^t U_i'}}(U_1')}\cdots g^{J(H_P(U))}_{\widetilde{H}_{B_U}(U_t')} = g_{\overline{\varphi}_t(\xi_t(U'_1,\cdots, U'_t))}^{J(H_P(U))}.$$
    The assumption $F(h) = g = F(h')$ implies 
    \begin{align*}
       \overline{\varphi}_t(\xi_t(U_1,\cdots, U_t)) &=  \overline{\varphi}_t\left(\widetilde{H}_{B_{U\oplus \bigoplus_{i=2}^t U_i}}(U_1), \cdots, \widetilde{H}_{B_U}(U_t)\right) \\
       &= \overline{\varphi}_t\left(\widetilde{H}_{B_{U\oplus \bigoplus_{i=2}^t U'_i}}(U'_1), \cdots, \widetilde{H}_{B_U}(U'_t)\right) \\
        &= \overline{\varphi}_t(\xi_t(U_1',\cdots, U_t')),
    \end{align*}
    and thus $\xi_t(U_1,\cdots,U_t) = \xi_t(U_1',\cdots, U_t')$ by the injectivity of $\overline{\varphi}_t$. Since $\xi_t$ is bijective, it follows that $(U_1,\cdots, U_t) = (U_1',\cdots, U_t')$ which implies $h = h'$ by Lemma \ref{h_V=h_U1+...+Ut}, and thus $F$ is faithful.  

    It remains to show that $F$ is a discrete fibration. Let $\C_U$ be an object in $\Mcluster{(\C)}$ and let $g: F(\C_U)\to \W$ be a morphism in $\Mcluster{(\Lambda)}$. Then $F(\C_U)= J(H_P(U))$ and $g = g_\V^{J(H_P(U))}: J(H_P(U))\to J_{J(H_P(U))}(\V)$, where $\V = \widetilde{H}_{B_U}(V)$ for a presilting object $V$ in $\widetilde{\C}_U$ by Lemma \ref{H_tilda}. Notice that $J_{J(H_P(U))}\left(\widetilde{H}_{B_U}(V)\right) = J(H_P(U\oplus V))$ by Lemma \ref{J(H_P)=J(HTilde)}. Let $$g = g^{J(H_P(U))}_{\widetilde{H}_{B_U}(V)}= g^{\W_1}_{\U_1}\cdots g_{\U_t}^{\W_t},$$
    be a decomposition of $g$ into irreducible maps with $\W_t = J(H_P(U))$ and $\W_i = J_{\W_{i+1}}(\U_{i+1})$ for $i= 1,\cdots, t-1$. By \cite[Proposition 11.8]{A_category_of_wide_subcategories}, this decomposition corresponds to a signed $\tau$-exceptional sequence $\U = (\U_1, \cdots, \U_t)$ with $\overline{\varphi}_t(\U) = \widetilde{H}_{B_U}(V)$. Proposition~\ref{signed_presilt_seq-signed_tau-exc_seq_bijection} gives the existence of a unique signed presilting sequence $\X = (X_1,\dots,X_t)$ in $\widetilde{\C}_U$ such that $\xi_t(\X)=\U$. This yields a morphism $ h = h^{{\C}_U}_{X_1\oplus\cdots\oplus X_t}$ in $\Mcluster(\C)$ together with a decomposition \[ h = h^{{\C}_U}_{X_1\oplus\cdots\oplus X_t} = h^{\B_1}_{X_1}\cdots h^{\B_t}_{X_t} \] into irreducible maps, where $\B_t={\C}_U$ and $ \B_i={\C}_{U\oplus X_t\oplus\cdots\oplus X_{t-i+1}}$ for $i=1,\dots,t-1$. By definition of $\xi_t$ applied in the reduced 0-Auslander extriangulated category $\widetilde{\C}_U$, we have 
    \begin{align*}
        \xi_t(\X) = \left( \widetilde{H}_{B_U\oplus \bigoplus_{i=2}^{t-1}X_i}(X_1),\cdots, \widetilde{H}_{B_U\oplus X_t}(X_{t-1}), \widetilde{H}_{B_U}(X_t)\right)
        = \left(\U_1,\cdots, \U_{t-1}, \U_t\right).
    \end{align*}
    Using Lemma \ref{J(H_P)=J(HTilde)} and the definition of $F$, we inductively obtain
    \begin{align*}
        &F\left(h_{X_t}^{\B_t}\right) = F\left( h_{X_t}^{{\C}_U}\right) = g^{J(H_P(U))}_{\widetilde{H}_{B_U}(X_t)} = g^{J(H_P(U)))}_{\U_t};\\
        &F\left( h_{X_{t-1}}^{\B_{t-1}}\right) = F\left( h_{X_{t-1}}^{{\C}_{U\oplus X_t}}\right) = g^{J(H_P(U\oplus X_t))}_{\widetilde{H}_{B_{U\oplus X_t}}(X_{t-1})} = g^{J_{J(H_P(U))}(\widetilde{H}_{B_U}(X_t))}_{\widetilde{H}_{B_{U\oplus X_t}(X_{t-1})}} = g_{\U_{t-1}}^{\W_{t-1}};\\
        &\quad\vdots\\
        &F\left( h_{X_{1}}^{\B_{1}}\right) = F\left( h_{X_{1}}^{{\C}_{U\oplus \bigoplus _{i=2}^{t} X_i}}\right) = g^{J(H_P(U\oplus \bigoplus _{i=2}^{t} X_i))}_{\widetilde{H}_{B_{U\oplus \bigoplus _{i=2}^{t} X_i}}(X_{1})} = g^{J_{J(H_P(U\oplus \bigoplus _{i=2}^{t} X_i))}(\widetilde{H}_{B_{U\oplus \bigoplus _{i=2}^{t} X_i}}(X_1))}_{\widetilde{H}_{B_{U\oplus \bigoplus _{i=2}^{t} X_i}(X_{1})}} = g_{\U_{1}}^{\W_{1}}.
    \end{align*}
    Hence, 
    \begin{align*}
        F\left( h^{{\C}_U}_{\bigoplus_{i=1}^{t}X_i}\right) = F\left( h^{\B_1}_{X_1}\cdots h^{\B_t}_{X_t}\right) = F\left(h^{\B_1}_{X_1}\right)\cdots F\left(h^{\B_t}_{X_t}\right) = g^{\W_1}_{\U_1}\cdots g_{\U_t}^{\W_t} = g^{J(H_P(U))}_{\widetilde{H}_{B_U}(V)} = g
    \end{align*}
    with $\overline{\varphi}_t(\xi_t(\X)) = \overline{\varphi}_t(\U) = \widetilde{H}_{B_U}(V)$. This shows that $F$ is a discrete fibration and concludes the proof. 
\end{proof}

\begin{remark}
    The functor $F: \Mcluster{(\C)\to \Mcluster(\Lambda)}$ is not full in general. Indeed, consider $\Lambda = k(\bullet)$ and take $\C$ to be $K^{[-1,0]}(\proj\Lambda)$. Then, the $\tau$-cluster morphism category of $\C$ and the $\tau$-cluster morphism category of $\Lambda$ are given by 
     \begin{equation*}
     \begin{tikzcd}[ampersand replacement=\&,cramped]
    	\& \C \\
    	{\C_P} \&\& {\C_{\Sigma P}}
    	\arrow["P"', from=1-2, to=2-1]
    	\arrow["{\Sigma P}", from=1-2, to=2-3]
    \end{tikzcd}
    \text{and} \qquad
     \begin{tikzcd}[ampersand replacement=\&,cramped]
    	{\modd\Lambda} \\
    	0
    	\arrow["P"', shift right, from=1-1, to=2-1]
    	\arrow["{P[1]}", shift left, from=1-1, to=2-1]
    \end{tikzcd}
    \end{equation*}
    respectively. Notice that, since there is a non-split $\EE$-triangle $P\infl 0\defl \Sigma P$, we always have $\C_P\neq \C_{\Sigma P}$. Clearly, the map 
    $$\Hom_{\Mcluster(\C)}(\C, \C_{\Sigma P})\to \Hom_{\Mcluster{\Lambda}}(F(\C), F(\C_{\Sigma P})) =  \Hom_{\Mcluster(\Lambda)}(\modd\Lambda, 0)$$
    is not surjective. In other words, $F$ is not full.
\end{remark}

Our next goal is to reconstruct the $\tau$-cluster morphism category $\Mcluster(\Lambda)$ from $\Mcluster(\C)$ and the functor $F: \Mcluster(\C)\to \Mcluster(\Lambda)$. To do this, we construct a quotient of the $\tau$-cluster morphism of $\C$, denoted by ${\Mcluster(\C)}/{\sim}$, given by an equivalence relation on the objects and arrows of $\Mcluster(\C)$ that identifies objects and morphisms that are mapped to the same image in $\Mcluster(\Lambda)$.

\begin{definition}\label{congruence_cat_def}
    Let $\mathcal{C}, \mathcal{D}$ be small categories and let $F: \mathcal{C}\to \mathcal{D}$ be a discrete fibration.
    \begin{itemize}
        \item Define an equivalence relation $\sim$ on the objects of $\mathcal{C}$ as follows. For $x, y\in \mathcal{C}$, we set $x\sim y$ if and only if $F(x) = F(y)$. We denote the equivalence class of the object $c$ in ${\mathcal{C}}/{\sim}$ by $[x]$.
        \item For $x,y\in\mathcal{C}$, define an equivalence relation $\sim$ on 
        $$\bigsqcup_{\substack{x', y' \in \mathcal{C} \\ F(x') =  F(x),\; F(y') = F(y)}} \mathrm{Hom}_{\mathcal{C}}(x', y')$$
        by setting $f\sim g$ if and only if $F(f) = F(g)$. We denote the equivalence class of $f$ by $[f]$.
    \end{itemize}
    The category ${\mathcal{C}}/{\sim}$ consists of the following data. 
    \begin{itemize}
        \item[(a)] The objects of ${\mathcal{C}}/{\sim}$ are given by $\text{ob}(\mathcal{C})/{\sim}$. 
        \item[(b)] Given two objects in $[x]$ and $[y]$ in ${\mathcal{C}}/{\sim}$, define 
        $$\Hom_\mathcal{{\mathcal{C}/{\sim}}}([x],[y]) = \left(\bigsqcup_{\substack{x', y' \in \mathcal{C} \\ F(x') = F(x),\; F(y') = F(y)}} \mathrm{Hom}_{\mathcal{C}}(x', y')\right) \Big/ {\sim}$$
        For an object $[x]$ in ${\mathcal{C}}/{\sim}$ define $\id_{[x]} = [\id_x]$. 
        \item[(c)] Let $\alpha \colon X \to Y$ and $\beta \colon Y \to Z$ be composable morphisms in ${\mathcal{C}}/{\sim}$. Then there exist morphisms $f \colon x \to y$ and $g \colon y' \to z$ in $\mathcal{C}$ such that $[f] = \alpha$, $[g] = \beta$, and $F(y) = F(y')$. Since $F$ is a discrete fibration, there exists a unique lift $\overline{g} \colon y \to z'$ such that $F(\overline{g}) = F(g)$. We then define the composition by
        \[
        \beta \circ \alpha := [\overline{g} \circ f].
        \]
    \end{itemize}
\end{definition}

\begin{lemma}
    Let $F:\mathcal{C}\to\mathcal{D}$ be a discrete fibration. The category $\mathcal{C}/{\sim}$ from Definition \ref{congruence_cat_def} is well-defined. 
\end{lemma}

\begin{proof}
    We begin by verifying that the composition law is independent of the choice of the representative. Let $\alpha: X\to Y$ and $\beta: Y\to Z$ be composable morphisms in ${\mathcal{C}}/{\sim}$. Then there exist $f: x\to y$ and $g:y'\to z$ in $\mathcal{C}$ with $F(y) = F(y')$ such that $\alpha = [f]$ and $\beta = [g]$. Since $F$ is a discrete fibration, there exists a unique lift $\overline{g}: y\to z'$ such that $F(g) = F(\overline{g})$. Then $[g]\circ [f] = [\overline{g}\circ f]$. Now let $\widetilde{f}: \widetilde{x}\to\widetilde{y}$ and $\widetilde{g}: \widetilde{y'}\to \widetilde{z}$ be two other representatives of $[f]$ and $[g]$, respectively. By the discrete fibration property, there exists a unique lift ${\overline{\widetilde{g}}}: \widetilde{y}\to\widetilde{z'}$ such that $F\left({\overline{\widetilde{g}}}\right) = F\left(\widetilde{g}\right)$. Then $\left[\widetilde{g}\right]\circ\left[\widetilde{f}\right] = \left[\overline{\widetilde{g}}\circ\widetilde{f}\right]$. Notice that 
    \begin{align*}
        F\left(\overline{g}\circ f\right) = F\left(\overline{g}\right)\circ F(f) = F(g)\circ F\left(\widetilde{f}\right) = F\left(\widetilde{g}\right)\circ F\left(\widetilde{f}\right) = F\left(\overline{\widetilde{g}}\right)\circ F\left(\widetilde{f}\right) = F\left(\overline{\widetilde{g}}\circ\widetilde{f}\right).
    \end{align*}
    In other words $\overline{g}\circ f \sim \overline{\widetilde{g}}\circ\widetilde{f}$, and therefore $\left[\overline{g}\circ f\right] = \left[\overline{\widetilde{g}}\circ\widetilde{f}\right]$.
    
    It is easy to see that for an object $[x]$ in ${\mathcal{C}}/{\sim}$, the map $\id_{[x]} = [\id_x]$ satisfies the axioms required for an identity map. It remains to show that the composition law is associative. Let $\alpha: X\to Y$, $\beta: Y\to W$, and $\gamma: W\to Z$ be composable morphisms in ${\mathcal{C}}/{\sim}$. Then there exist morphisms $f: x\to y$, $g: y'\to w$, and $h: w'\to z$ in $\mathcal{C}$ with $F(y) = F(y')$ and $F(w) = F(w')$ such that $\alpha = [f]$, $\beta = [g]$, and $\gamma = [h]$. Since $F$ is a discrete fibration, there exists a unique lift $\overline{g}:  y\to w_1$ such that $F(g)=F(\overline{g})$. Hence, $[g]\circ[f] = [\overline{g}\circ f]$. Since $F(w') = F(w_1)$, there is a unique lift $\overline{h_1}: w_1\to z_1$ such that $F\left(\overline{h_1}\right) = F(h)$. Thus, 
    $$[h]\circ ([g]\circ [f]) = \left[\overline{h_1}\circ\overline{g}\circ f\right].$$    
    Now, compute $[h]\circ[g]$ first. Since $F(w) = F(w')$, we get a unique morphism $\overline{h_2}: w\to z_2$ such that $F\left(\overline{h_2}\right) =F(h)$. Then, $[h]\circ[g] = \left[\overline{h_2}\circ g\right]$. Because $F(y) = F(y')$, we obtain a unique lift $\overline{\overline{h_2}\circ g}: y\to z_2'$ such that $F\left(\overline{h_2}\circ g\right) = F\left(\overline{\overline{h_2}\circ g}\right)$. Hence, 
    $$ ([h]\circ [g])\circ[f] = \left[\left(\overline{\overline{h_2}\circ g}\right)\circ f\right].$$ We observe that 
    \begin{align*}
        F\left(\overline{\overline{h_2}\circ g}\circ f\right) &= F\left(\overline{\overline{h_2}\circ g}\right)F(f) = F\left(\overline{h_2}\circ g\right)F(f)=F\left(\overline{h_2}\right)F(g)F(f) \\ &=F(h)F(g)F(f) = F\left(\overline{h_1}\right)F\left(\overline{g}\right)F(f) = F\left(\overline{h_1}\circ\overline{g}\circ f\right).
    \end{align*}
    Hence $\overline{\overline{h_2}\circ g}\circ f\sim \overline{h_1}\circ\overline{g}\circ f$, and therefore 
    $$[h]\circ([g]\circ [f]) = \left[\overline{h_1}\circ\overline{g}\circ f\right] = \left[\overline{\overline{h_2}\circ g}\circ f\right] = ([h]\circ[g])\circ[f].$$
    This proves that the composition law is associative, and the claim follows. 
\end{proof}

\begin{lemma}\label{discete_fibration_equivalence}
    Let $F:\mathcal{C}\to\mathcal{D}$ be dense, faithful, and a discrete fibration. Then, the category ${\mathcal{C}}/{\sim}$ is equivalent to $\mathcal{D}$.
\end{lemma}

\begin{proof}
    Let $H: {\mathcal{C}}/{\sim}\to \mathcal{D}$ be the functor defined as $H([c]) = F(c)$ on objects and $H([f]) = F(f)$ on morphisms. We claim that $H$ is an equivalence of categories. Let $d\in\mathcal{D}$. Since $F$ is dense, $d\cong F(c)$ for some object $c\in\mathcal{D}$. Thus, taking the corresponding equivalence class in ${\mathcal{C}}/{\mathcal{\sim}}$, we get $H([c]) = F(c)\cong d$, proving that $H$ is dense. We show that $\Hom_{{\mathcal{C}/{\sim}}}([c],[c'])\cong \Hom_\mathcal{D}(F(c),F(c'))$. Let $f: F(c)\to F(c')$ be a morphism in $\Hom_\mathcal{D}(F(c),F(c'))$. Since $F$ is a discrete fibration, there exists a unique morphism $l: c\to x$ in $\mathcal{C}$ with $F(x)=F(c')$ and $F(l) = f$. This defines a morphism $[l]: [c]\to[x]$ with $H([l]) = F(l) = f$. Now let $[f],[f']: [c]\to [c']$ be two morphisms in $\Hom_{{\mathcal{C}/{\sim}}}([c],[c'])$ such that $H([f]) = H([f'])$. Then $F(f) = F(f')$. Since $F$ is faithful, $f = f'$, so $[f] = [f']$. In other words, $H$ is also fully faithful and therefore an equivalence.
\end{proof}

Applying the above construction to $\Mcluster(\C)$ we obtain the following. 

\begin{theorem}
    The categories ${\Mcluster(\C)}/{\sim}$ and $\Mcluster(\Lambda)$ are equivalent.
\end{theorem}

\begin{proof}
    The claim follows directly by combining Theorem \ref{F_dense_faithful_discrete_fibration} and Lemma \ref{discete_fibration_equivalence}. 
\end{proof}

\section{Example}\label{sec:Examples}

\begin{example}
    Let $\Lambda = kQ/I$ where $Q = $ 
    \begin{tikzcd}[ampersand replacement=\&,cramped]
	1 \& 2,
	\arrow["\alpha", curve={height=-6pt}, from=1-1, to=1-2]
	\arrow["\beta", curve={height=-6pt}, from=1-2, to=1-1]
    \end{tikzcd} 
    and $I$ is the ideal generated by $\alpha\beta$ and $\beta\alpha$. Let $\C := K^{[-1,0]}(\proj\Lambda)$. Then, $\C$ is an extension closed subcategory of the triangulated category $K^b(\proj\Lambda)$ and therefore extriangulated. Moreover, $\C$ is reduced $0$-Auslander; see \cite[Rmk. 5.5]{some_appl_of_extriang_cat}. The AR quiver of $\Lambda$ is given by 
    
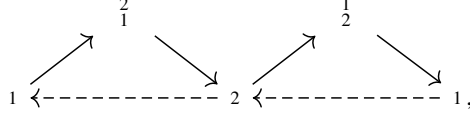
\begin{equation*}
    \begin{tikzcd}[ampersand replacement=\&,cramped]
	   \& \begin{array}{c} \begin{smallmatrix}  2 \\ 1 \end{smallmatrix} \end{array} \&\& \begin{array}{c} \begin{smallmatrix} 1 \\ 2 \end{smallmatrix} \end{array} \& \\
	       {\begin{smallmatrix} 1  \end{smallmatrix}} \&\& {\begin{smallmatrix} 2 \end{smallmatrix}} \&\& {\begin{smallmatrix} 1 \end{smallmatrix}},
	   \arrow[from=1-2, to=2-3]
	   \arrow[from=1-4, to=2-5]
	   \arrow[from=2-1, to=1-2]
	   \arrow[from=2-3, to=1-4]
	   \arrow[dashed, from=2-3, to=2-1]
	   \arrow[dashed, from=2-5, to=2-3]
    \end{tikzcd}
    \end{equation*}
    while the AR quiver of $\C$ is given by
    
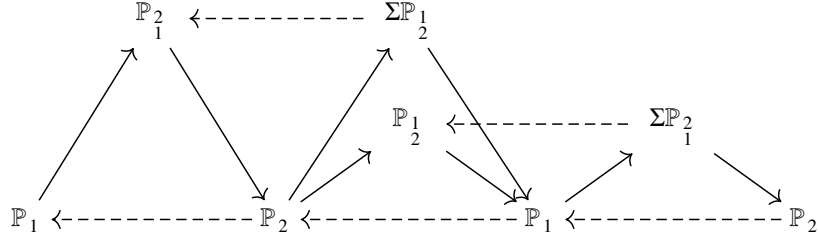
\begin{equation*}
        \begin{tikzcd}[ampersand replacement=\&,cramped]
    	\& \begin{array}{c} \PP_{\begin{smallmatrix}  2 \\ 1 \end{smallmatrix}} \end{array} \&\& \begin{array}{c} \Sigma\PP_{\begin{smallmatrix}  1 \\ 2 \end{smallmatrix}} \end{array} \&\&\& \\
    	\&\&\& \begin{array}{c} \PP_{\begin{smallmatrix} 1 \\ 2 \end{smallmatrix}} \end{array} \&\& \begin{array}{c} \Sigma\PP_{\begin{smallmatrix}  2 \\ 1 \end{smallmatrix}} \end{array} \\
    	{\PP_{\begin{smallmatrix}  1 \end{smallmatrix}}} \&\& {\PP_{\begin{smallmatrix} 2 \end{smallmatrix}}} \&\& {\PP_{\begin{smallmatrix} 1 \end{smallmatrix}}} \&\& {\PP_{\begin{smallmatrix}  2  \end{smallmatrix}}}
    	\arrow[from=1-2, to=3-3]
    	\arrow[dashed, from=1-4, to=1-2]
    	\arrow[from=1-4, to=3-5]
    	\arrow[from=2-4, to=3-5]
    	\arrow[dashed, from=2-6, to=2-4]
    	\arrow[from=2-6, to=3-7]
    	\arrow[from=3-1, to=1-2]
    	\arrow[from=3-3, to=1-4]
    	\arrow[from=3-3, to=2-4]
    	\arrow[dashed, from=3-3, to=3-1]
    	\arrow[from=3-5, to=2-6]
    	\arrow[dashed, from=3-5, to=3-3]
    	\arrow[dashed, from=3-7, to=3-5]
    \end{tikzcd}
    \end{equation*}
    where $\PP_M$ denotes the minimal projective presentation of a $\Lambda$-module $M$. Let $P = \PP_{\begin{smallmatrix}1\end{smallmatrix}}\oplus \PP_{\begin{smallmatrix}2\\1\end{smallmatrix}}$ be a projective generator for $\C$. Observe that the functor $\C(P,-)$ induces an equivalence between $\C/[\Sigma P]$ and $\modd\Lambda$. The $\tau$-cluster morphism category of $\Lambda$ is depicted in Figure \ref{fig:tau-cluster-morphism-category}. The vertices of $\Mcluster(\Lambda)$ are $\tau$-perpendicular subcategories $J(X)\subseteq \modd\Lambda$ for a support $\tau$-rigid object $X$ in $\mathcal{C}(\Lambda)$. For an object $\W$ in $\Mcluster(\Lambda)$, there is a non-identity irreducible morphism $g^\W_V: \W\to J_\W(V)$ for every indecomposable support $\tau$-rigid object $V$ in $\mathcal{C}(\W)$ shown as an arrow between $\W$ and $J_\W(V)$ labeled by $V$. Since $J_\W(P) = J_\W(P[1])$ for an indecomposable projective object $P$ in $\W$, there are two arrows from $\W$ to $J_{\W}(P)$ labeled by $P$ and $P[1]$. Notice that complete signed $\tau$-exceptional sequences can be read off from Figure \ref{fig:tau-cluster-morphism-category} composing arrows from $\modd\Lambda$ to $0$. For example, there is a path $\modd\Lambda\xrightarrow{2} J(2)\xrightarrow{\begin{smallmatrix} 2\\1\end{smallmatrix}} J_{J(2)}\left(\begin{smallmatrix}2\\1\end{smallmatrix}\right) = 0$. This corresponds to the (signed) $\tau$-exceptional sequence $\left( \begin{smallmatrix} 2\\1\end{smallmatrix}, 2\right)$. Since $\begin{smallmatrix} 2\\1\end{smallmatrix}$ is a projective object in $J(2)$, it can be assigned a shift sign to get the signed $\tau$-exceptional sequence $\left( \begin{smallmatrix} 2\\1\end{smallmatrix}[1], 2\right)$. 

    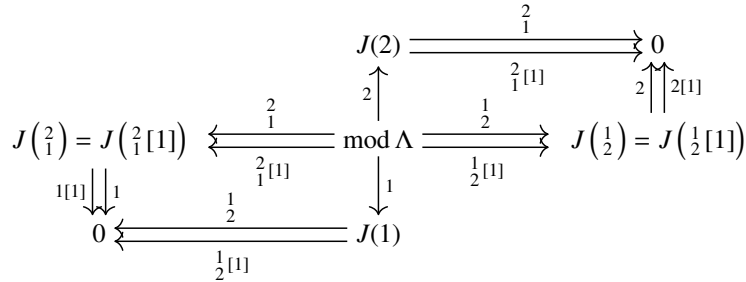
\begin{figure}[H]
    \centering
    \begin{tikzcd}[ampersand replacement=\&,cramped]
    	\&\& {J(2)} \&\& 0 \\
    	\begin{array}{c} J\left(\begin{smallmatrix} 2\\1 \end{smallmatrix}\right) = J\left(\begin{smallmatrix} 2\\1 \end{smallmatrix}{[1]}\right) \end{array} \&\& {\modd\Lambda} \&\& \begin{array}{c} J\left(\begin{smallmatrix} 1\\2 \end{smallmatrix}\right) = J\left(\begin{smallmatrix} 1\\2 \end{smallmatrix}{[1]}\right) \end{array} \\
    	0 \&\& {J(1)}
    	\arrow["\begin{array}{c} \begin{smallmatrix} 2\\1 \end{smallmatrix}{[1]} \end{array}"', shift right, from=1-3, to=1-5]
    	\arrow["\begin{array}{c} \begin{smallmatrix} 2\\1 \end{smallmatrix} \end{array}", shift left, from=1-3, to=1-5]
    	\arrow["{1[1]}"', shift right, from=2-1, to=3-1]
    	\arrow["1", shift left, from=2-1, to=3-1]
    	\arrow["2", from=2-3, to=1-3]
    	\arrow["\begin{array}{c} \begin{smallmatrix} 2\\1 \end{smallmatrix} \end{array}"', shift right, from=2-3, to=2-1]
    	\arrow["\begin{array}{c} \begin{smallmatrix} 2\\1 \end{smallmatrix}{[1]} \end{array}", shift left, from=2-3, to=2-1]
    	\arrow["\begin{array}{c} \begin{smallmatrix} 1\\2 \end{smallmatrix}{[1]} \end{array}"', shift right, from=2-3, to=2-5]
    	\arrow["\begin{array}{c} \begin{smallmatrix} 1\\2 \end{smallmatrix} \end{array}", shift left, from=2-3, to=2-5]
    	\arrow["1", from=2-3, to=3-3]
    	\arrow["{2[1]}"', shift right, from=2-5, to=1-5]
    	\arrow["2", shift left, from=2-5, to=1-5]
    	\arrow["\begin{array}{c} \begin{smallmatrix} 1\\2 \end{smallmatrix} \end{array}"', shift right, from=3-3, to=3-1]
    	\arrow["\begin{array}{c} \begin{smallmatrix} 1\\2 \end{smallmatrix}{[1]} \end{array}", shift left, from=3-3, to=3-1]
    \end{tikzcd}
    \caption{The $\tau$-cluster morphism category of $\Lambda$.}
    \label{fig:tau-cluster-morphism-category}
    \end{figure}
On the other hand, the $\tau$-cluster morphism category of $\C$ is displayed in Figure \ref{fig:tau-cluster-morphism-category-C}. The objects of $\Mcluster(\C)$ are subcategories of $\C$ of the form $\C_U = U^\perp\cap {^\perp U}$ for a presilting object $U$ in $\C$. For an object $\C_U$ in $\Mcluster(\C)$, there is a non-identity irreducible morphism $h^{\C_U}_V: \C_U\to \C_{U\oplus V}$, denoted by $V$, for every indecomposable presilting object $V\in\C_U$ with $\add V\cap\add U = 0$, see Definition \ref{tau-cluster-morph-cat-of-C}. Similarly to the $\Mcluster(\Lambda)$ case, one can read off signed presilting sequences from Figure \ref{fig:tau-cluster-morphism-category-C} concatenating arrows from $\C$ to $\C_U$, where $\abs{U} = 2$. For example, there is a path $$\C\xrightarrow{\PP_{\begin{smallmatrix}2\\1\end{smallmatrix}}} \C_{\PP_{\begin{smallmatrix}2\\1\end{smallmatrix}}}\xrightarrow{\PP_{\begin{smallmatrix}1\\2\end{smallmatrix}}}\C_{\PP_{\begin{smallmatrix}2\\1\end{smallmatrix}}\oplus\PP_{\begin{smallmatrix}1\\2\end{smallmatrix}}},$$ which corresponds to the signed presilting sequence $\left(\PP_{\begin{smallmatrix}1\\2\end{smallmatrix}}, \PP_{\begin{smallmatrix}2\\1\end{smallmatrix}}\right)$. We have that $\C_{\PP_{\begin{smallmatrix}2\\1\end{smallmatrix}}} = \left\{\PP_{\begin{smallmatrix}2\\1\end{smallmatrix}}, \PP_{\begin{smallmatrix}1\\2\end{smallmatrix}}, \PP_{\begin{smallmatrix}2\end{smallmatrix}}\right\}$, where $\proj\C_{\PP_{\begin{smallmatrix}2\\1\end{smallmatrix}}} = \left\{\PP_{\begin{smallmatrix}2\\1\end{smallmatrix}}, \PP_{\begin{smallmatrix}1\\2\end{smallmatrix}}\right\}$ and $\inj\C_{\PP_{\begin{smallmatrix}2\\1\end{smallmatrix}}} = \left\{\PP_{\begin{smallmatrix}2\\1\end{smallmatrix}}, \PP_{\begin{smallmatrix}2\end{smallmatrix}}\right\}$. Hence, since $\PP_{\begin{smallmatrix}1\\2\end{smallmatrix}}$ is projective in $\C_{\PP_{\begin{smallmatrix}2\\1\end{smallmatrix}}}$, it can be shifted (in $\C_{\PP_{\begin{smallmatrix}2\\1\end{smallmatrix}}}$) to get the signed presilting sequence $\left(\Sigma_{\C_{\PP_{\begin{smallmatrix}2\\1\end{smallmatrix}}}}\left(\PP_{\begin{smallmatrix}1\\2\end{smallmatrix}}\right), \PP_{\begin{smallmatrix}2\\1\end{smallmatrix}}\right) = \left(\PP_{\begin{smallmatrix}2\end{smallmatrix}}, \PP_{\begin{smallmatrix}2\\1\end{smallmatrix}}\right)$.
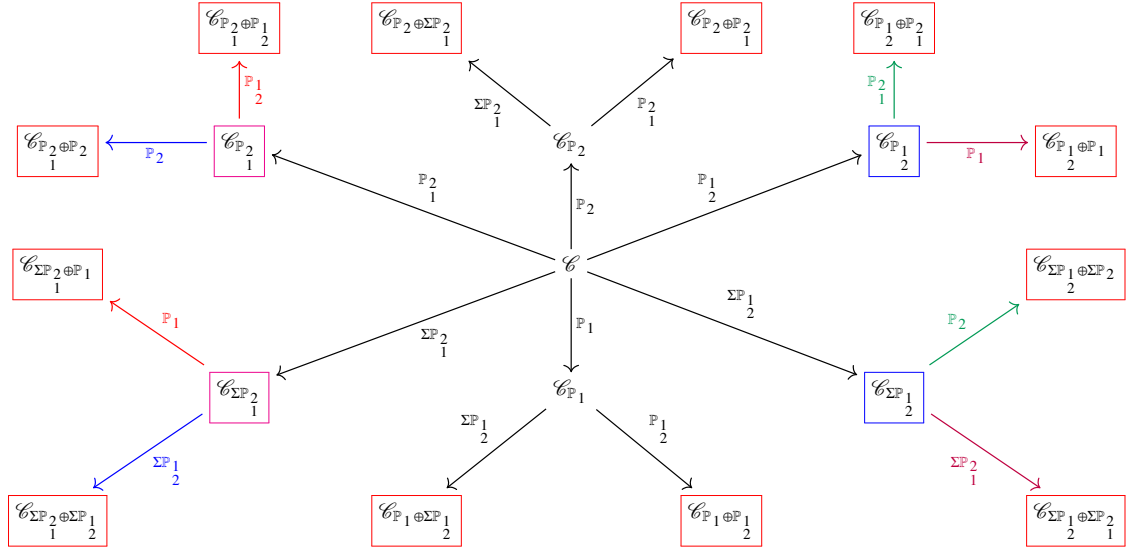
\begin{figure}[H]
\centering
\resizebox{\textwidth}{!}{%
\begin{tikzcd}[ampersand replacement=\&, cramped, column sep=large, row sep=large]
    \& {\redbox{\C_{\PP_{\begin{smallmatrix} 2\\1 \end{smallmatrix}}\oplus{\PP_{\begin{smallmatrix} 1\\2 \end{smallmatrix}}}}}} \& {\redbox{\C_{\PP_{\begin{smallmatrix} 2 \end{smallmatrix}}\oplus{\Sigma\PP_{\begin{smallmatrix} 2\\1 \end{smallmatrix}}}}}} \&\& {\redbox{\C_{\PP_{\begin{smallmatrix} 2 \end{smallmatrix}}\oplus{\PP_{\begin{smallmatrix} 2\\1 \end{smallmatrix}}}}}} \& {\redbox{\C_{\PP_{\begin{smallmatrix} 1\\2 \end{smallmatrix}}\oplus{\PP_{\begin{smallmatrix} 2\\1 \end{smallmatrix}}}}}} \& \\
    {\redbox{\C_{\PP_{\begin{smallmatrix} 2\\1 \end{smallmatrix}}\oplus{\PP_{\begin{smallmatrix} 2 \end{smallmatrix}}}}}} \& {\magentabox{\C_{\PP_{\begin{smallmatrix} 2\\1 \end{smallmatrix}}}}} \&\& {\displaystyle\C_{\PP_{\begin{smallmatrix} 2 \end{smallmatrix}}}} \&\& {\bluebox{\C_{\PP_{\begin{smallmatrix} 1\\2 \end{smallmatrix}}}}} \& {\redbox{\C_{\PP_{\begin{smallmatrix} 1\\2 \end{smallmatrix}}\oplus{\PP_{\begin{smallmatrix} 1 \end{smallmatrix}}}}}} \\
    {\redbox{\C_{\Sigma\PP_{\begin{smallmatrix} 2\\1 \end{smallmatrix}}\oplus{\PP_{\begin{smallmatrix} 1 \end{smallmatrix}}}}}} \&\&\& \displaystyle\C \&\&\& {\redbox{\C_{\Sigma\PP_{\begin{smallmatrix} 1\\2 \end{smallmatrix}}\oplus{\Sigma\PP_{\begin{smallmatrix} 2 \end{smallmatrix}}}}}} \\
    \& {\magentabox{\C_{\Sigma\PP_{\begin{smallmatrix} 2\\1 \end{smallmatrix}}}}} \&\& {\displaystyle\C_{\PP_{\begin{smallmatrix} 1 \end{smallmatrix}}}} \&\& {\bluebox{\C_{\Sigma\PP_{\begin{smallmatrix} 1\\2 \end{smallmatrix}}}}} \\
    {\redbox{\C_{\Sigma\PP_{\begin{smallmatrix} 2\\1 \end{smallmatrix}}\oplus{\Sigma\PP_{\begin{smallmatrix} 1\\2 \end{smallmatrix}}}}}} \&\& {\redbox{\C_{\PP_{\begin{smallmatrix} 1 \end{smallmatrix}}\oplus{\Sigma\PP_{\begin{smallmatrix} 1\\2 \end{smallmatrix}}}}}} \&\& {\redbox{\C_{\PP_{\begin{smallmatrix} 1 \end{smallmatrix}}\oplus{\PP_{\begin{smallmatrix} 1\\2 \end{smallmatrix}}}}}} \&\& {\redbox{\C_{\Sigma\PP_{\begin{smallmatrix} 1\\2 \end{smallmatrix}}\oplus{\Sigma\PP_{\begin{smallmatrix} 2\\1 \end{smallmatrix}}}}}}
    \arrow["{\PP_{\begin{smallmatrix} 1\\2 \end{smallmatrix}}}"', color=red, from=2-2, to=1-2]
    \arrow["{\PP_{\begin{smallmatrix} 2 \end{smallmatrix}}}", color=blue, from=2-2, to=2-1]
    \arrow["{\Sigma\PP_{\begin{smallmatrix} 2\\1 \end{smallmatrix}}}", from=2-4, to=1-3]
    \arrow["{\PP_{\begin{smallmatrix} 2\\1 \end{smallmatrix}}}"', from=2-4, to=1-5]
    \arrow["{\PP_{\begin{smallmatrix} 2\\1 \end{smallmatrix}}}", color=ForestGreen, from=2-6, to=1-6]
    \arrow["{\PP_{\begin{smallmatrix} 1 \end{smallmatrix}}}"', color=purple, from=2-6, to=2-7]
    \arrow["{\PP_{\begin{smallmatrix} 2\\1 \end{smallmatrix}}}"', from=3-4, to=2-2]
    \arrow["{\PP_{\begin{smallmatrix} 2 \end{smallmatrix}}}"', from=3-4, to=2-4]
    \arrow["{\PP_{\begin{smallmatrix} 1\\2 \end{smallmatrix}}}", from=3-4, to=2-6]
    \arrow["{\Sigma\PP_{\begin{smallmatrix} 2\\1 \end{smallmatrix}}}", from=3-4, to=4-2]
    \arrow["{\PP_{\begin{smallmatrix} 1 \end{smallmatrix}}}", from=3-4, to=4-4]
    \arrow["{\Sigma\PP_{\begin{smallmatrix} 1\\2 \end{smallmatrix}}}", from=3-4, to=4-6]
    \arrow["{\PP_{\begin{smallmatrix} 1 \end{smallmatrix}}}"', color=red, from=4-2, to=3-1]
    \arrow["{\Sigma\PP_{\begin{smallmatrix} 1\\2 \end{smallmatrix}}}", color=blue, from=4-2, to=5-1]
    \arrow["{\Sigma\PP_{\begin{smallmatrix} 1\\2 \end{smallmatrix}}}"', from=4-4, to=5-3]
    \arrow["{\PP_{\begin{smallmatrix} 1\\2 \end{smallmatrix}}}", from=4-4, to=5-5]
    \arrow["{\PP_{\begin{smallmatrix} 2 \end{smallmatrix}}}", color=ForestGreen, from=4-6, to=3-7]
    \arrow["{\Sigma\PP_{\begin{smallmatrix} 2\\1 \end{smallmatrix}}}"', color=purple, from=4-6, to=5-7]
\end{tikzcd}
}
\caption{The $\tau$-cluster morphism category of $\C$.}
\label{fig:tau-cluster-morphism-category-C}
\end{figure}
There is a functor $F: \Mcluster(\C)\to \Mcluster(\Lambda)$ given by $F(H_P(U))$ on objects and $F\left(h^{\C_U}_V\right) = g^{J(H_P(U))}_{\widetilde{H}_{B_U}(V)}$ on morphisms; see Proposition \ref{functor-tau-cluster-morph-cat}. We have $F\left(\C_{\PP_{\begin{smallmatrix} 2 \end{smallmatrix}}}\right) = J(2)$, $F\left(\C_{{\PP_{\begin{smallmatrix} 2\\1 \end{smallmatrix}}}}\right) = J\left({{\begin{smallmatrix} 2\\1 \end{smallmatrix}}}\right) = J\left({{\begin{smallmatrix} 2\\1 \end{smallmatrix}}[1]}\right) = F\left(\C_{\Sigma\PP_{\begin{smallmatrix} 2\\1 \end{smallmatrix}}}\right)$. Moreover, $F\left(\C_U\right) = 0$ for every $U$ of rank $2$. Noting that the category $\Mcluster(\C)$ is symmetric, one can readily compute the images of the remaining objects under $F$. Objects of $\Mcluster(\C)$ that are sent to the same image in $\Mcluster(\Lambda)$ via $F$ are indicated by boxes of the same colour. We also observe that $F$ is surjective on objects. The functor $F$ sends the arrow $\PP_{\begin{smallmatrix} 2\\1 \end{smallmatrix}} \colon \C \to \C_{\PP_{\begin{smallmatrix} 2\\1 \end{smallmatrix}}}$ in $\Mcluster(\C)$ to the arrow ${\begin{smallmatrix} 2\\1 \end{smallmatrix}} \colon \modd\Lambda \to J(2)$ in $\Mcluster(\Lambda)$. Continuing in the same way, the arrows $\PP_{\begin{smallmatrix} 2\\1 \end{smallmatrix}}: \C_{\PP_{\begin{smallmatrix} 2\\1 \end{smallmatrix}}} \to \C_{\PP_{\begin{smallmatrix} 2\\1 \end{smallmatrix}} \oplus \PP_{\begin{smallmatrix} 1\\2 \end{smallmatrix}}}$ and $\PP_{\begin{smallmatrix} 2 \end{smallmatrix}} \colon \C_{\PP_{\begin{smallmatrix} 2\\1 \end{smallmatrix}}} \to \C_{\PP_{\begin{smallmatrix} 2\\1 \end{smallmatrix}} \oplus \PP_{\begin{smallmatrix} 2 \end{smallmatrix}}}$ in $\Mcluster(\C)$ are mapped to the arrows $1 \colon J(2) \to 0$ and $1[1] \colon J(2) \to 0$ in $\Mcluster(\Lambda)$, respectively. In other words, the signed presilting sequence $\left(\PP_{\begin{smallmatrix} 2 \end{smallmatrix}}, \PP_{\begin{smallmatrix} 2\\1 \end{smallmatrix}}\right)$ in $\C$ is sent to the signed $\tau$-exceptional sequence $\left(1[1], {\begin{smallmatrix} 2\\1 \end{smallmatrix}}\right)$ in $\CLambda(\Lambda)$; see Proposition~\ref{signed_presilt_seq-signed_tau-exc_seq_bijection}. In a similar way, one can compute the images of morphisms in $\Mcluster(\C)$ under $F$. Morphisms in $\Mcluster(\C)$ that are sent to the same image via $F$ are highlighted in the same colour.. We observe that $F$ is faithful. Moreover, for each object $F(\C_U)$ in $\Mcluster(\Lambda)$ and every morphism $g \colon F(\C_U) \to \W$ in $\Mcluster(\Lambda)$, there exists a unique morphism $h \colon \C_U \to \C_W$ in $\Mcluster(\C)$ such that $F(h) = g$. In other words, $F$ is a discrete fibration; see Theorem~\ref{F_dense_faithful_discrete_fibration}.

By identifying objects and morphisms that are sent to the same image under $F$, we obtain a quotient category ${\Mcluster(\C)}/{\sim}$, which is depicted in Figure~\ref{fig:Mcluster-quotient}. Clearly, ${\Mcluster(\C)}/{\sim}$ is equivalent to the $\tau$-cluster morphism category $\Mcluster(\Lambda)$. In this way, the latter is recovered from the $\tau$-cluster morphism category of $\C$ together with the functor $F$.

\begin{figure}[H]
\centering
\begin{tikzcd}[ampersand replacement=\&,cramped]
    \&\& {\left[\C_{\PP_{\begin{smallmatrix} 2 \end{smallmatrix}}}\right]} \&\& \begin{array}{c} \left[\C_{\PP_{\begin{smallmatrix} 1\\2 \end{smallmatrix}}\oplus{\PP_{\begin{smallmatrix} 2\\1 \end{smallmatrix}}}}\right] \end{array} \\
    \\
    \begin{array}{c} \left[\C_{\PP_{\begin{smallmatrix} 2\\1 \end{smallmatrix}}}\right] \end{array} \&\& {[\C]} \&\& \begin{array}{c} \left[\C_{\PP_{\begin{smallmatrix} 1\\2 \end{smallmatrix}}}\right] \end{array} \\
    \\
    \begin{array}{c} \left[\C_{\PP_{\begin{smallmatrix} 1\\2 \end{smallmatrix}}\oplus{\PP_{\begin{smallmatrix} 2\\1 \end{smallmatrix}}}}\right] \end{array} \&\& {\left[\C_{\PP_{\begin{smallmatrix} 1 \end{smallmatrix}}}\right]}
    \arrow["{\left[\PP_{\begin{smallmatrix}2\\1\end{smallmatrix}}\right]}"', shift right, from=1-3, to=1-5]
    \arrow["{\left[\Sigma\PP_{\begin{smallmatrix}2\\1\end{smallmatrix}}\right]}", shift left, from=1-3, to=1-5]
    \arrow["{\left[\PP_{\begin{smallmatrix}2\end{smallmatrix}}\right]}"', shift right, from=3-1, to=5-1]
    \arrow["{\left[\PP_{\begin{smallmatrix}1\end{smallmatrix}}\right]}", shift left, from=3-1, to=5-1]
    \arrow["{\left[\PP_{\begin{smallmatrix}2\end{smallmatrix}}\right]}", from=3-3, to=1-3]
    \arrow["{\left[\Sigma\PP_{\begin{smallmatrix}2\\1\end{smallmatrix}}\right]}"', shift right, from=3-3, to=3-1]
    \arrow["{\left[\Sigma\PP_{\begin{smallmatrix}2\\1\end{smallmatrix}}\right]}", shift left, from=3-3, to=3-1]
    \arrow["{\left[\Sigma\PP_{\begin{smallmatrix}1\\2\end{smallmatrix}}\right]}"', shift right, from=3-3, to=3-5]
    \arrow["{\left[\PP_{\begin{smallmatrix}1\\2\end{smallmatrix}}\right]}",  shift left, from=3-3, to=3-5]
    \arrow["{\left[\PP_{\begin{smallmatrix}1\end{smallmatrix}}\right]}", from=3-3, to=5-3]
    \arrow["{\left[\PP_{\begin{smallmatrix}2\end{smallmatrix}}\right]}", from=3-5, to=1-5]
    \arrow["{\left[\Sigma\PP_{\begin{smallmatrix}2\end{smallmatrix}}\right]}"', shift right=2, from=3-5, to=1-5]
    \arrow["{\left[\Sigma\PP_{\begin{smallmatrix}1\\2\end{smallmatrix}}\right]}"', shift right, from=5-3, to=5-1]
    \arrow["{\left[\PP_{\begin{smallmatrix}1\\2\end{smallmatrix}}\right]}", shift left, from=5-3, to=5-1]
\end{tikzcd}
\caption{The category ${\Mcluster(\C)}/{\sim}$.}
\label{fig:Mcluster-quotient}
\end{figure}
\end{example}

\paragraph{\textbf{Acknowledgment:}} This work forms part of my PhD at the University of Leeds, funded by the School of Mathematics. I am also affiliated with the EPSRC Programme Grant EP/W007509/1, Combinatorial Representation Theory: Algebra and its Interfaces with Geometry, Topology and Combinatorics. I am deeply grateful to Bethany R. Marsh for her supervision and for our many inspiring discussions.

\maketitle

\nocite{*}
\bibliographystyle{alpha}
\bibliography{references}

\end{document}

%% file: macro.tex
\newcommand{\abs}[1]{\left\lvert#1\right\rvert}
\newcommand{\news}[1]{{\color{green} #1}}
\newcommand{\new}[1]{{\color{blue} #1}}
\newcommand{\newp}[1]{{\color{purple} #1}}
\newcommand{\old}[1]{{\color{red} #1}}
\definecolor{darkgreen}{rgb}{0.33, 0.42, 0.18}
\newcommand{\rjm}[1]{{\color{darkgreen} #1}}

\newcommand{\proj}{\mathop{\rm proj}\nolimits}%
\newcommand{\inj}{\mathop{\rm inj}\nolimits}%
\providecommand{\Gen}{\mathop{\rm Gen}\nolimits}%
\providecommand{\Cogen}{\mathop{\rm Cogen}\nolimits}%
\providecommand{\Filt}{\mathop{\rm Filt}\nolimits}%
\providecommand{\cone}{\mathop{\rm cone}\nolimits}%
\providecommand{\cocone}{\mathop{\rm cocone}\nolimits}%
\providecommand{\Sub}{\mathop{\rm Sub}\nolimits}%
\providecommand{\rad}{\mathop{\rm rad}\nolimits}%
\providecommand{\coker}{\mathop{\rm coker}\nolimits}%
\providecommand{\im}{\mathop{\rm im}\nolimits}%
\providecommand{\btop}{\mathop{\rm top}\nolimits}%
\newcommand{\infl}{\rightarrowtail}
\newcommand{\defl}{\twoheadrightarrow}

\def\A{\mathcal{A}}
\def\C{\mathscr{C}}
\def\B{\mathscr{B}}
\def\D{\mathcal{D}}
\def\P{\mathcal{P}}
\def\Y{\mathbb{Y}}
\def\Z{\mathcal{Z}}
\def\K{\mathcal{K}}
\def\F{\mathcal{F}}
\def\Eps{\mathcal{E}}
\def\T{\mathcal{T}}
\def\U{\mathcal{U}}
\def\V{\mathcal{V}}
\def\W{\mathcal{W}}
\def\E{\mathsf{E}}
\def\CLambda{\mathcal{C}}
\def\M{\mathsf{M}}
\def\N{\mathsf{N}}
\def\X{\mathsf{X}}
\def\L{\mathsf{L}}
\def\Mcluster{\mathfrak{M}}

\def\CC{\mathbb{C}}
\def\XX{\mathbb{X}}
\def\PP{{\mathbb P}}
\def\EE{{\mathbb E}}

\providecommand{\add}{\mathop{\rm add}\nolimits}%
\providecommand{\End}{\mathop{\rm End}\nolimits}%
\providecommand{\Ext}{\mathop{\rm Ext}\nolimits}%
\providecommand{\Hom}{\mathop{\rm Hom}\nolimits}%
\providecommand{\ind}{\mathop{\rm ind}\nolimits}%
\providecommand{\pd}{\mathop{\rm pd}\nolimits}%
\providecommand{\id}{\mathop{\rm id}\nolimits}%
\providecommand{\gldim}{\mathop{\rm gl.dim}\nolimits}%
\newcommand{\modd}{\mathop{ \rm mod}\nolimits}%
\newcommand{\stautilt}{\textnormal{s$\tau$-tilt }}
\newcommand{\taugenmin}{\textnormal{$\tau$-gen-min }}
\newcommand{\ftors}{\textnormal{f-tors }}

\providecommand{\presilt}{\mathop{\rm presilt}\nolimits}%
\newcommand{\staurigid}{\textnormal{s$\tau$-rigid }}%
\providecommand{\silt}{\mathop{\rm silt}\nolimits}%